\documentclass{amsart}

\usepackage{url}

\input{macros.sty}

\author{Rebecca Bellovin}
\title{Generic smoothness for $G$-valued potentially semi-stable deformation rings}
\address{
Department of Mathematics,
180 Queen's Gate,
London SW7 2AZ, UK
}\email{r.bellovin@imperial.ac.uk}

\begin{document}

\begin{abstract}
We extend Kisin's results on the structure of characteristic $0$ Galois deformation rings to deformation rings of Galois representations valued in arbitrary connected reductive groups $G$.  In particular, we show that such Galois deformation rings are complete intersections.  In addition, we study explicitly the structure of the moduli space $X_{\varphi,N}$ of (framed) $(\varphi,N)$-modules when $G=\GL_n$.  We show that when $G=\GL_3$ and $K_0=\Q_p$, $X_{\varphi,N}$ has a singular component, and we construct a moduli-theoretic resolution of singularities.
\end{abstract}

\setlength{\parskip}{0ex} 
\maketitle

\tableofcontents

\setlength{\parskip}{0ex} 
\setlength{\parindent}{2ex}

\section{Introduction}

Let $K/\Q_p$ be a finite extension, and let $V$ be a finite dimensional $\F_p$-vector space equipped with a continuous $\F_p$-linear action of $\Gal_K$.  Let $R_V^\square$ be the universal (framed) deformation ring of $V$.  Then Kisin proved that for a fixed $p$-adic Hodge type $\mathbf{v}$ and a fixed Galois type $\tau$, there is a quotient $R_V^\square\twoheadrightarrow R_V^{\square,\tau,\mathbf{v}}$ whose characteristic $0$ points are the potentially semi-stable deformations of $V$ with $p$-adic Hodge type $\mathbf{v}$ and Galois type $\tau$~\cite{kisin}.  He further showed that $\Spec R_V^{\square,\tau,\mathbf{v}}[1/p]$ is equi-dimensional and generically smooth, and computed its dimension.  More recently, Hartl and Hellmann have shown that the moduli space of $(\varphi,N)$-modules is reduced and Cohen--Macaulay~\cite[Theorem 3.2]{hartl-hellmann}, which implies the same properties for $\Spec R_V^{\square,\tau,\mathbf{v}}[1/p]$.

However, it is natural to study Galois representations valued in connected reductive groups other than $\GL_n$.  If $G$ is a connected reductive group over a $p$-adic field which admits a smooth reductive integral model, Balaji has used techniques from integral $p$-adic Hodge theory to construct potentially semi-stable and potentially crystalline integral deformation rings~\cite{balaji}, extending the work of Kisin~\cite{kisin}.  He further showed that potentially crystalline deformation rings are smooth and computed their dimensions.

In this paper, we extend Kisin's results on the local structure of $\Spec R_V^{\square,\tau,\mathbf{v}}[1/p]$ to the study of the characteristic $0$ deformation rings of Galois representations $\rho:\Gal_K\rightarrow G(E)$, where $G$ is a connected reductive group defined over a finite extension $E/\Q_p$.  More precisely, we show the following:
\begin{thm}\label{galois-def-rings}
Let $\rho:\Gal_K\rightarrow G(E)$ be a continuous homomorphism.  Fix a finite totally ramified Galois extension $L/K$, a corresponding Galois type $\tau$, and a $p$-adic Hodge type $\mathbf{v}$.  Then there is a complete local noetherian $E$-algebra $R_\rho^{\square,\tau,\mathbf{v}}$ which pro-represents the deformation problem 
\begin{multline*}
\Def_\rho^{\square,\tau,\mathbf{v}}(R):=\{\widetilde\rho:\Gal_K\rightarrow G(R)| \widetilde\rho|_{\Gal_L}\text{ is a semi-stable lift of }\rho|_{\Gal_L}\\ 
\text{ with Galois type }\tau\text{ and }p\text{-adic Hodge type }\mathbf{v}\}
\end{multline*}
Furthermore, $R_\rho^{\square,\tau,\mathbf{v}}$ is generically smooth, locally a complete intersection, and equidimensional of dimension $\dim_EG+\dim_E(\Res_{E\otimes K/E}G)/P_\mathbf{v}$ (where $P_{\mathbf{v}}$ is the parabolic associated to the $p$-adic Hodge type $\mathbf{v}$).
\end{thm}
Even if $G=\GL_n$, this improves on the generic smoothness result of~\cite{kisin}. 

As in~\cite{kisin}, to prove Theorem~\ref{galois-def-rings} we study Galois deformation rings and their singularities by studying a certain moduli space of linear algebra data.  Fontaine's theory defines an equivalence of categories between potentially semi-stable Galois representations (valued in $\GL_n$) and ``weakly admissible filtered $(\varphi,N,\Gal_K)$-modules''.  We use the theory of Tannakian categories to define $G$-valued filtered $(\varphi,N,\Gal_K)$-modules in ~\textsection\ref{definitions}, and we review the relevant Tannakian theory in detail in the appendix (in particular, we show that the $p$-adic Hodge type is locally constant).  The analogue of the weak admissibility condition is not clear in general, but infinitesimal deformations of admissible filtered $(\varphi,N,\Gal_K)$-modules are admissible so this suffices to study the deformation theory of potentially semi-stable $G$-valued Galois representations.  

More precisely, in order to prove Theorem~\ref{galois-def-rings}, we first prove the following:
\begin{thm}\label{x-phi-n-tau-sings}
Let $X_{\varphi,N,\tau}$ denote the moduli space of framed $G$-valued $(\varphi,N,\Gal_{L/K})$-modules, where $L/K$ is a finite totally ramified Galois extension.  Then $X_{\varphi,N,\tau}$ is reduced and locally a complete intersection, and each irreducible component has dimension $\dim \Res_{E\otimes L_0/E}G$, where $L_0=K_0$ is the maximal unramified subfield of $L$.  

For any $p$-adic Hodge type $\mathbf{v}$, there is a parabolic subgroup $P_{\mathbf{v}}\subset \Res_{E\otimes L/E}G$ attached to $\mathbf{v}$; the moduli space of framed $G$-valued filtered $(\varphi,N,\Gal_{L/K})$-modules is reduced and locally a complete intersection, and each irreducible component has dimension $\dim \Res_{E\otimes L_0/E}G+\dim (\Res_{E\otimes L/E}G)/P_{\mathbf{v}}$.
\end{thm}
This extends the work of Hartl and Hellmann~\cite{hartl-hellmann}, who showed that if $G=\GL_n$ and $L=K$, then $X_{\varphi,N,\tau}$ is reduced and Cohen--Macaulay.  Using the relationship between potentially semi-stable Galois representations and filtered $(\varphi,N,\Gal_K)$-modules, and following the arguments of~\cite{kisin}, this permits us to deduce Theorem~\ref{galois-def-rings} and its crystalline analogue in~\textsection\ref{section-def-rings}.

In order to prove Theorem~\ref{x-phi-n-tau-sings}, we first recall the deformation theory of $G$-torsors and morphisms between them in ~\textsection\ref{deformations}.  This permits us to write down a tangent-obstruction theory for deformations of filtered $(\varphi,N,\Gal_K)$-modules.  Using this tangent-obstruction theory and the theory of cocharacters associated to nilpotent elements of a Lie algebra (recalled in ~\textsection\ref{associated}), we show in ~\textsection\ref{reducedness} that the moduli space of (framed) $G$-filtered $(\varphi,N,\Gal_K)$ modules is generically smooth and equidimensional.

Although we set up the deformation theory for filtered $(\varphi,N,\Gal_{L/K})$-modules with $L/K$ an arbitrary finite Galois extension, we need to assume $L/K$ is totally ramified in order to carry out our calculations.  This is because we need the centralizer of $\tau$ in $\Res_{L_0\otimes E/E}G$ to be an algebraic group, which is not the case unless $\tau$ is linear, rather than semi-linear.  Both \cite{kisin} and \cite{balaji} assume that $\tau$ factors through the inertia group $I_{L/K}$, so this additional hypothesis is not overly restrictive.

However, this still leaves open the question of the structure of the irreducible components of $X_{\varphi,N,\tau}$.  We partially address this question in~\textsection\ref{section-calcs} for $G=\GL_n$ when $\tau$ is trivial.  For $G=\GL_2$, we show that $X_{\varphi,N}$ is geometrically the union of two smooth schemes intersecting in a smooth divisor, recovering the result of~\cite[Lemma A.3]{kisin-fmc}, and we give explicit equations.  For $G=\GL_n$, we show that when $K_0=\Q_p$, the irreducible component corresponding to $N$ being regular nilpotent is smooth.  However, for $G=\GL_3$, we show that geometrically there are three irreducible components, which we write $X_{\reg}$, $X_{\sub}$, and $X_0$, and $X_{\sub}$ is singular.  Loosely speaking, the three irreducible components correspond to the three (geometric) nilpotent conjugacy classes of $\mathfrak{gl}_3$, and $X_{\sub}$ is the one corresponding to the subregular nilpotent orbit.  One might hope that the singular points $X_{\sub}$ arise as degenerations from a different irreducible component.  However, we provide a moduli-theoretic resolution of $X_{\sub}$ and use it to show that while some singular points of this component do come from $X_{\reg}$ (in a sense we make more precise in ~\textsection\ref{subreg}), there are others which do not.  This answers a question of Kisin in the negative~\cite{kisin3}.  However, we still know very little about the singularities of $X_{\sub}$.  For example, we do not know whether $X_{\sub}$ is locally a complete intersection.

\subsection*{Acknowledgements}
The results in this paper formed a portion of my Ph.D. thesis, and I am grateful to my advisor, Brian Conrad, for many helpful conversations, as well as for reading multiple drafts of this paper.  I am also grateful to Daniel Erman, Brandon Levin, Peter McNamara, and Jonathan Wise for helpful discussions.  This work was partially supported by an NSF Graduate Research Fellowship.

\section{Definitions}\label{definitions}

Let $E$ and $K$ be finite extensions of $\Q_p$.  Suppose that $\rho:\Gal_K\rightarrow\GL_d(E)$ is a potentially semi-stable Galois representation, becoming semi-stable when restricted to $\Gal_L$ for some finite Galois extension $L/K$.  Then we can associate to $\rho$ a filtered $(\varphi,N,\Gal_{L/K})$-module, which satisfies an additional weak admissibility condition.

\begin{definition}
A \emph{filtered $(\varphi,N,\Gal_{L/K})$-module} is a finite dimensional $L_0$-vector space $D$ equipped with a bijection $\Phi:D\rightarrow D$ which is semi-linear over $\varphi$, a linear endomorphism $N$ such that $N\circ\Phi=p\Phi\circ N$, an action $\tau$ of $\Gal_{L/K}$ which is semi-linear over the Galois action on $L_0$ and commutes with $\Phi$ and $N$, and a separated exhaustive decreasing $\Gal_{L/K}$-stable filtration $\Fil^\bullet D_L$ by $L$-vector spaces.  
\end{definition}

More generally, we will consider continuous Galois representations 
$$\rho:\Gal_K\rightarrow \Aut_G(X)(E)$$
where $G$ is a connected reductive algebraic group defined over $E$ and $X$ is a trivial right $G$-torsor over $E$.  
\begin{remark}
We work with trivial $G$-torsors and their automorphism schemes, rather than copies of $G$, in order to avoid making auxiliary choices of trivializing sections.  This allows us to preserve the traditional distinction between a vector space, and a vector space together with a choice of basis.
\end{remark}

A Galois representation $\rho$ is said to be \emph{potentially semi-stable} if $\sigma\circ\rho:\Gal_K\rightarrow\GL_d(E)$ is potentially semi-stable for some faithful representation $\sigma:G\rightarrow\GL_d$, in which case this holds for all representations $\sigma:G\rightarrow \GL_d$ over $E$.

If $\rho$ is potentially semi-stable, we use the Tannakian formalism to construct a $G$-valued version of the filtered $(\varphi,N,\Gal_{L/K})$-module $\D_{\st}^L(V)$; we refer the reader to ~\textsection\ref{tann-filt-phi-N} for details of the constructions and some of the notation.  Briefly, for every representation $\sigma:G\rightarrow\GL_d$, $\sigma\circ\rho$ is a potentially semi-stable representation, and $\D_{\st}^L(\sigma\circ\rho)$ is a weakly admissible filtered $(\varphi,N,\Gal_{L/K})$-module.  The formation of $\D_{\st}^L(\sigma\circ\rho)$ is exact and tensor-compatible in $\sigma$, and if $\mathbf{1}$ denotes the trivial representation of $G$, then $\D_{\st}^L(\mathbf{1}\circ\rho)$ is the trivial filtered $(\varphi,N,\Gal_{L/K})$-module.  

Therefore, $\sigma\mapsto \D_{\st}^L(\sigma\circ\rho)$ is a fiber functor $\eta:\Rep_E(G)\rightarrow \Vect_{L_0}$, and we obtain a $G$-torsor $Y$ over $L_0$ equipped with 
\begin{itemize}
\item	an isomorphism $\Phi:\varphi^\ast Y\rightarrow Y$,
\item	a nilpotent element $N\in\Lie\Aut_GY$,
\item	for each $g\in\Gal_{L/K}$, an isomorphism $\tau(g):g^\ast Y\rightarrow Y$,
\item	a $\Gal_{L/K}$-stable $\otimes$-filtration on $Y_L$, or equivalently, a $\otimes$-filtration on the $G$-torsor $Y_L^{\Gal_{L/K}}$ over $K$.
\end{itemize}
These are required to satisfy the following compatibilities:
\begin{itemize}
\item	$\underline\Ad\Phi (N)=\frac{1}{p}N$
\item	$\underline\Ad\tau(g)(N)=N$ for all $g\in\Gal_{L/K}$
\item	$\tau(g_1g_2)=\tau(g_2)\circ g_2^\ast\tau(g_1)$ for all $g_1,g_2\in\Gal_{L/K}$
\item	$\tau(g) \circ g^\ast\Phi= \Phi\circ \varphi^\ast\tau(g)$ for all $g\in\Gal_{L/K}$
\end{itemize}
Here $\underline\Ad\Phi$ and $\underline\Ad\tau(g)$ are ``twisted adjoint'' actions on $\Lie\Aut_GY$; after pushing out $Y$ by a representation $\sigma\in \Rep_E(G)$, they are given by $M\mapsto \Phi_\sigma\circ M\circ \Phi_\sigma^{-1}$ and $M\mapsto \tau(g)_\sigma\circ M\circ \tau(g)_\sigma^{-1}$, respectively.

The $G$-valued semi-linear representation $\tau$ is the \emph{Galois type} of $\rho$, and the type of the $\otimes$-filtration is the \emph{$p$-adic Hodge type} of $\rho$.  We refer the reader to \textsection\ref{tann-semilinear} and \textsection\ref{tann-filt} for more details, and in particular for the definition and a discussion of the type of a $\otimes$-filtration.

We wish to study moduli spaces of these objects, because the completed local rings of such moduli spaces will be related to local Galois deformation rings.  In fact, to study potentially semi-stable Galois deformation rings, it suffices to study the local structure of moduli spaces of linear algebra data.

\begin{definition}
Let $\rho:\Gal_K\rightarrow G(E)$ be a continuous representation, and let $R$ be an $E$-finite artin local ring with maximal ideal $\mathfrak{m}$ and residue field $E$.  A \emph{lift} of $\rho$ is a continuous homomorphism $\widetilde\rho:\Gal_K\rightarrow G(R)$ which is $\rho$ modulo $\mathfrak{m}$.  A \emph{deformation} of $\rho$ is a continuous homomorphism $\widetilde\rho:\Gal_K\rightarrow \Aut_G(X)(R)$ which is isomorphic to $\rho$ modulo $\mathfrak{m}$, where $X$ is a trivial $G$-torsor over $R$.  That is, a deformation is a lift where we have forgotten about the trivializing section.
\end{definition}

Suppose that $\rho$ is a potentially semi-stable representation, becoming semi-stable when restricted to $\Gal_L$.  If $\widetilde\rho$ is a lift of $\rho$ to $R$ which is potentially semi-stable, then $\D_{\st}^L(\widetilde\rho)$ is a $G$-valued filtered $(\varphi,N,\Gal_{L/K})$-module over $R$ lifting $\D_{\st}^L(\rho)$.  Similarly, if $\rho$ is semi-stable (resp. crystalline), a semi-stable (resp. crystalline) lift $\widetilde\rho$ yields a $G$-valued filtered $(\varphi,N)$-module (resp. a $G$-valued filtered $\varphi$-module).

\begin{prop}\label{def-reps}
Let $\rho:\Gal_K\rightarrow \Aut_G(X)(E)$ be a potentially semi-stable representation, where $X$ is a trivial $G$-torsor over $E$, and let $R$ be an $E$-finite artin local ring with residue field $E$.  Then $\widetilde\rho\rightsquigarrow \D_{\st}^L(\widetilde\rho)$ is an equivalence of categories from the category of potentially semi-stable deformations of $\rho$ to the category of filtered $(\varphi,N,\Gal_{L/K})$-modules over $R$ deforming $\D_{\st}^L(\rho)$.
\end{prop}
\begin{proof}
The formation of $\D_{\st}^L(\widetilde\rho)$ is clearly functorial in $\widetilde\rho$, so it suffices to construct a quasi-inverse.

Suppose $\widetilde\D$ is a deformation of $\D_{\st}^L(\rho)$ (as a $G$-valued filtered $(\varphi,N,\Gal_{L/K})$-module).  Then for every representation $\sigma:G\rightarrow \GL(V)$, the push-out $\widetilde\D_\sigma$ of $\widetilde \D$ is a filtered $(\varphi,N,\Gal_{L/K})$-module over $R$ which deforms $\D_{\st}^L(\sigma\circ\rho)$.  Since $\rho$ is potentially semi-stable, $\D_{\st}^L(\sigma\circ\rho)$ is weakly admissible for all $\sigma\in\Rep_E(G)$, and since deformations of weakly admissible filtered $(\varphi,N,\Gal_{L/K})$-modules are themselves weakly admissible, $\widetilde\D_\sigma$ is weakly admissible for all $\sigma\in\Rep_E(G)$.  

But weakly admissible filtered $(\varphi,N,\Gal_{L/K})$-modules are admissible when the coefficients are a $\Q_p$-finite artin ring, so we have an exact tensor compatible family $(\widetilde\rho_\sigma)$ of potentially semi-stable representations of $\Gal_K$, where $\widetilde\rho_\sigma$ is a deformation to $R$ of the push-out $\rho_\sigma$ of $\rho$.  Therefore, by the discussion in Appendix~\ref{cont-galois}, we obtain a continuous representation $\widetilde\rho:\Gal_K\rightarrow G(R)$ such that $\D_{\st}^L(\widetilde\rho)=\widetilde\D$ and $\widetilde\rho$ is a deformation of $\rho$.
\end{proof}

\begin{definition}
Let $R$ be an $E$-finite artin local ring.  We say that a $G$-valued filtered $(\varphi,N,\Gal_{L/K})$-module $\D$ over $R$ is \emph{admissible} if $\D=\D_{\st}^L(\rho)$ for some potentially semi-stable representation $\rho:\Gal_K\rightarrow \Aut_G(X)(R)$ for some $G$-torsor $X$ over $R$.
\end{definition}

\begin{remark}
In the course of the proof of Proposition~\ref{def-reps} we showed that a deformation of an admissible filtered $(\varphi,N,\Gal_{L/K})$-module is itself admissible.
\end{remark}

Thus, in order to study potentially semi-stable deformations of a specified $G$-valued potentially semi-stable Galois representation, it suffices to study the deformation theory of the associated linear algebra.

\begin{notation}
We will often need to consider tensor products $A\otimes_{\Q_p}L_0$, where $A$ is an $E$-algebra.  In order to simplify notation, particularly in subscripts, we adopt the convention that $A\otimes L_0$ means $A\otimes_{\Q_p}L_0$.
\end{notation}

\section{Deformation theory}\label{deformations}

Fix a finite extension $E/\Q_p$, a connected reductive group $G$ defined over $E$, a finite extension $K/\Q_p$, and a finite Galois extension $L/K$.  We will study the deformation theory of $G$-valued $(\varphi,N,\Gal_{L/K})$-modules, following~\cite{kisin}.

We first introduce two groupoids on the category of $E$-algebras.  Let $\mathfrak{Mod}_N$ be the groupoid whose fiber over an $E$-algebra $A$ consists of the category of $\Res_{E\otimes_{\Q_p}L_0/E}G$-bundles $D_A$ over $A$ equipped with a nilpotent $N\in\Lie\Aut_{G}D_A$ and a family of isomorphisms $\tau(g):g^\ast D_A\rightarrow D_A$ for $g\in \Gal_{L/K}$ such that $\tau(g_1g_2)=\tau(g_2)\circ g_2^\ast\tau(g_1)$ for all $g_1,g_2\in\Gal_{L/K}$, and such that $\underline\Ad(\tau(g))(N)=N$ for all $g\in\Gal_{L/K}$, where $\underline\Ad(\tau(g)):\Lie\Aut_{G}D_A\rightarrow \Lie\Aut_GD_A$ is the induced homomorphism.

Let $\mathfrak{Mod}_{\varphi,N}$ be the groupoid whose fiber over an $E$-algebra $A$ consists of the category of $G$-bundles $D_A$ over $A\otimes_{\Q_p}L_0$ equipped with $\tau$ and $N$ as before, and also an isomorphism $\Phi:\varphi^\ast D_A\rightarrow D_A$ such that $\tau(g) \circ g^\ast\Phi= \Phi\circ \varphi^\ast\tau(g)$ and $\underline\Ad(\Phi) (N)=\frac{1}{p}N$, where $\underline\Ad(\Phi):\Lie\Aut_{G}D_A\rightarrow \Lie\Aut_GD_A$ is the induced homomorphism.

\begin{remark}
To motivate these definitions, we refer the reader to Sections~\ref{tann-semilinear} and \ref{tann-fam-pst}.  We remark only that if $G=\GL_d$ and $D_A$ is a free $A\otimes_{\Q_p}L_0$-module of rank $d$, and $\Phi_D:D_A\rightarrow D_A$ is a semi-linear bijection represented by a matrix $M$, then $\Aut_{G}D_A=(\Res_{L_0/\Q_p}\GL_d)_A$, and $\Phi_D$ induces a map $\Aut_{G}D_A\rightarrow\Aut_{G}D_A$ sending $g\in \GL(A\otimes_{\Q_p}K_0)$ to $\Phi\circ g\circ\Phi^{-1}$, which is represented by the matrix $M\varphi(g)M^{-1}$.

Suppose more generally that $D_A$ is a split $\Res_{E\otimes L_0/E}G$-torsor over an $D$-algebra $A$.  If we choose a trivializing section of $D_A$, it induces a trivializing section of $\varphi^\ast D_A$, and an isomorphism $\varphi^\ast D_A\rightarrow D_A$ of $G$-torsors is given by multiplication by an element $b\in\Res_{E\otimes L_0/E}G$.  If we change our choice of trivializing section by multiplying it by $g\in (\Res_{E\otimes L_0/E}G)(A)$, then $b$ turns into $g^{-1}b\varphi(g)$.  Thus, the linearization of Frobenius $\varphi^\ast D_A\xrightarrow{\sim}D_A$ is given by $b\in\Res_{E\otimes L_0/E}G$, up to ``twisted conjugation''.

If we choose a trivializing section of $D_A$, we obtain an identification of $\Aut_GD_A$ with $(\Res_{E\otimes L_0/E}G)_A$.  Using this identification we can view $N$ as an element of $(\Res_{E\otimes L_0/E}\ad G)(A)$.  Then since $N=p\Ad(b)(\varphi(N))$ holds after pushing out by any representation of $G$, it holds in $(\Res_{E\otimes L_0/E}\ad G)(A)$ as well.
\end{remark}

Given $D_A\in\mathfrak{Mod}_{\varphi,N}$, we let $\ad D_A$ be the $(\varphi,N,\Gal_{L/K})$-module over $A$ induced on $\Lie\Aut_GD_A$.  We denote the Frobenius, nilpotent operator, and action of $\Gal_{L/K}$ on $\ad D_A$ by $\underline\Ad(\Phi)$, $\ad_N$, and $\underline\Ad(\tau)$, as well.  Consider the anti-commutative diagram
\begin{equation*}
\xymatrix{
(\ad D_A)^{\Gal_{L/K}}	\ar[r]^{1-\underline\Ad(\Phi)}\ar[d]^{\ad_N} &	(\ad D_A)^{\Gal_{L/K}}\ar[d]^{\ad_N}	\\
(\ad D_A)^{\Gal_{L/K}}	\ar[r]^{p\underline\Ad(\Phi)-1} &	(\ad D_A)^{\Gal_{L/K}}	
}
\end{equation*}

We write $C^\bullet(D_A)$ for the total complex of this double complex, concentrated in degrees $0$, $1$, and $2$, with differentials $d^0$, $d^1$, and $d^2$, and we write $\H^\bullet(D_A)$ for the cohomology of $C^\bullet(D_A)$.  This is the $G$-valued analogue of the complex which appears in~\cite[\textsection 3]{kisin}  The total complex of this double complex controls the deformation theory of $\mathfrak{Mod}_{\varphi,N}$:

\begin{prop}\label{def-phi-n-tau}
Let $A$ be an artin local $E$-algebra with maximal ideal $\mathfrak{m}_A$, and let $I\subset A$ be an ideal with $I\mathfrak{m}_A=0$.  Let $D_{A/I}$ be an object of $\mathfrak{Mod}_{\varphi,N}(A/I)$, and set $D_{A/\mathfrak{m}_A}=D_{A/I}\otimes_{A/I}A/\mathfrak{m}_A$.  Then
\begin{enumerate}
\item	If $\H^2(D_{A/\mathfrak{m}_A})=0$ then there exists $D_A$ in $\mathfrak{Mod}_{\varphi,N}(A)$ lifting $D_{A/I}$.
\item	The set of isomorphism classes of liftings of $D_{A/I}$ over $A$ is either empty or a torsor under the cohomology group $\H^1(D_{A/\mathfrak{m}_A})\otimes_{A/\mathfrak{m}_A}I$.
\end{enumerate}
\end{prop}

Before we can prove this, we will need some preparatory results on deformations of torsors and on deformations of isomorphisms of torsors.  Note that $I\otimes_{A/\mathfrak{m}_A}D_{A/\mathfrak{m}_A}$ is the set of elements of $(\Aut_GD_A)(A)$ which are the identity modulo $I$.  It is an $A/\mathfrak{m}_A$-vector space, and we sometimes view it additively and sometimes multiplicatively.  We refer to such elements of $(\Aut_GD_A)(A)$ as \emph{infinitesimal automorphisms}.

Let $A$ be a henselian local $E$-algebra with maximal ideal $\mathfrak{m}_A$, and let $I\subset A$ be an ideal with $I\mathfrak{m}_A=0$.  

\begin{lemma}\label{auts-def}
Let $H$ be an affine algebraic group over $E$, and let $D_A$ be an $H$-torsor over $A$.  Let $\overline{f}:D_{A/I}\rightarrow D_{A/I}$ be an automorphism of $H$-torsors.  Then there exists a lift $f:D_A\rightarrow D_A$, and the set of such lifts is a left and right torsor under $I\otimes_{A/\mathfrak{m}_A}\ad D_{A/\mathfrak{m}_A}$.
\end{lemma}
\begin{proof}
If $D_{A}$ is split, the existence of a lift is clear.  Otherwise, there is a finite \'etale cover $A\rightarrow A'$ such that $D_{A'}$ is split, and we may choose a lift $f':D_{A'}\rightarrow D_{A'}$ of $\overline{f}':D_{A'/I}\rightarrow D_{A'/I}$.  Then $f'$ yields a cocycle $[f']\in\H_{\et}^1(\Spec A,I\otimes_{A/\mathfrak{m}_A}\ad D_{A/\mathfrak{m}_A})$, and $f'$ can be modified to descend to an isomorphism $f:D_A\rightarrow D_A$ if and only if $[f']=0$.  But $\Spec A$ is affine and $I\otimes_{A/\mathfrak{m}_A}\ad D_{A/\mathfrak{m}_A}$ is a finite $A$-module.  Therefore, $\H_{\et}^1(\Spec A,I\otimes_{A/\mathfrak{m}_A}\ad D_{A/\mathfrak{m}_A})=0$ and $f$ exists.

The second assertion is clear, because if $f,f':D_A\rightrightarrows D_A$ are two lifts, then $f\circ f'$ and $f'\circ f$ are elements of $I\otimes_{A/\mathfrak{m}_A}\ad D_{A/\mathfrak{m}_A}$.
\end{proof}

The next result also follows from~\cite[Expos\'e XXIV, Lemme 8.1.8]{SGA3}:

\begin{lemma}\label{torsor-def}
Let $H$ be an affine algebraic group over $E$, and let $D_{A/I}$ be an $H$-torsor over $A/I$.  Then there is an $H$-torsor $D_A$ over $A$ such that $D\otimes_AA/I\cong D_{A/I}$, and $D_A$ is unique up to isomorphism.
\end{lemma}
\begin{proof}
If $D_{A/I}$ is a split $H$-torsor, the existence of a lift is clear.  Otherwise, there is a finite \'etale cover $A\rightarrow A'$ such that $D_{A'/I}$ is split, and we may choose a trivial lift $D_{A'}$ of $D_{A'/I}$.  Then $D_{A'}$ yields a cocycle $[D_{A'}]\in \H_{\et}^2(\Spec A,I\otimes_{A/\mathfrak{m}_A}\ad D_{A/\mathfrak{m}_A})$.  Letting $p_1,p_2:\Spec A'\times_A\Spec A'\rightrightarrows \Spec A'$ be the projection maps, we can find an isomorphism $p_1^\ast D_{A'}\xrightarrow{\sim} p_2^\ast D_{A'}$ (which is the identity modulo $I$), and $[D_{A'}]=0$ if and only if this isomorphism can be chosen to give a descent datum.  But $\Spec A$ is affine and $I\otimes_{A/\mathfrak{m}_A}\ad D_{A/\mathfrak{m}_A}$ is a finite $A$-module, so $\H_{\et}^2(\Spec A,I\otimes_{A/\mathfrak{m}_A}\ad D_{A/\mathfrak{m}_A})=0$.  Since $H$-torsors are affine and descent is effective for affine morphisms, $D_A$ exists.

Now suppose that $D_A$ and $D_A'$ are two lifts of $D_{A/I}$.  There is a finite \'etale cover $A\rightarrow A'$ over which they become isomorphic, and a choice of isomorphism $f'$ over $A'$ yields a cocycle $[f']\in \H_{\et}^1(\Spec A,I\otimes_{A/\mathfrak{m}_A}\ad D_{A/\mathfrak{m}_A})$.  Then $[f']=0$ if and only if $f'$ can be modified to give an isomorphism $f:D_A\xrightarrow{\sim} D_A'$.  But $A$ is affine and $I\otimes_{A/\mathfrak{m}_A}\ad D_{A/\mathfrak{m}_A}$ is a finite $A$-module, so $\H_{\et}^2(\Spec A,I\otimes_{A/\mathfrak{m}_A}\ad D_{A/\mathfrak{m}_A})=0$ and $[f']=0$.  Therefore, $D_A$ is unique up to isomorphism.
\end{proof}

\begin{lemma}\label{isom-def}
Let $H$ be an affine algebraic group over $E$, and let $D_A$, $D_A'$ be $H$-torsors over $A$, and let $\overline{f}:D_{A/I}\rightarrow D_{A/I}'$ be an isomorphism.  Then there exists a lift $f:D_A\rightarrow D_A'$, and the set of such lifts is a torsor under a left action of $\H^0(A,I\otimes_{A/\mathfrak{m}_A}\ad D_{A/\mathfrak{m}_A}')$ and a right action of $\H^0(A,I\otimes_{A/\mathfrak{m}_A}\ad D_{A/\mathfrak{m}_A})$.
\end{lemma}
\begin{proof}


Since $D_{A/I}$ and $D_{A/I}'$ are isomorphic, and there is a unique lift of $D_{A/I}$ to an $H$-torsor over $A$, up to isomorphism, $D_A\cong D_A'$.  Therefore, we may fix an identification of $D_A$ and $D_A'$, so that the existence of a lift of the isomorphism $\overline{f}$ becomes a question of the existence of a lift of an automorphism of an $H$-torsor.  But such a lift exists, by Lemma~\ref{auts-def}.

If $f_1,f_2:D_A\rightrightarrows D_A'$ are two isomorphisms lifting $\overline f$, then $f_1\circ f_2^{-1}$ is an automorphism of $D_A'$ which is the identity modulo $I$, and is therefore an element of $I\otimes_{A/\mathfrak{m}_A}\ad D_A'$.  On the other hand, given an automorphism $g$ of $D_A'$ which is the identity modulo $I$, $g\circ f_2$ is a lift of $\overline f$.

Similarly, $f_2^{-1}\circ f_1$ is an automorphism of $D_A$ which is the identity modulo $I$.  On the other hand, given an automorphism $h$ of $D_A$ which is the identity modulo $I$, $f_2\circ h$ is a lift of $\overline f$.
\end{proof}

\begin{lemma}\label{lift-tau}
Let $D_A$ be a $\Res_{E\otimes L_0/E}G$-torsor over $A$, and let $\tau_0$ be a semi-linear action of $\Gal_{L/K}$ on $D_{A/I}$.  That is, we have a set of isomorphisms $\tau_0(g):g^\ast D_{A/I}\rightarrow D_{A/I}$ such that $\tau_0(g_1g_2)=\tau_0(g_2)\circ g_2^\ast\tau_0(g_1)$.  Then there is a semi-linear action $\tau$ of $\Gal_{L/K}$ lifting $\tau_0$, and it is unique up to isomorphism.
\end{lemma}
\begin{proof}
For every $g\in\Gal_{L/K}$, we can lift $\tau_0(g)$ to an isomorphism $\tau_1(g):g^\ast D_A\rightarrow D_A$, and we wish to show that we can choose the $\{\tau_1(g)\}_{g\in\Gal_{L/K}}$ such that they provide an action of $\Gal_{L/K}$.  The assignment $(g_1,g_2)\mapsto c(g_1,g_2):=\tau_1(g_2)\circ g_2^\ast\tau_1(g_1)\circ \tau_1(g_1g_2)^{-1}$ is a $2$-cocycle of $\Gal_{L/K}$ valued in $\H^0(A,I\otimes_{A/\mathfrak{m}_A}\ad D_{A/\mathfrak{m}_A})$.  But 
$$\H^2(\Gal_{L/K},\H^0(A,I\otimes_{A/\mathfrak{m}_A}\ad D_{A/\mathfrak{m}_A}))=0$$
because $\Gal_{L/K}$ is a finite group and $\H^0(A,I\otimes_{A/\mathfrak{m}_A}\ad D_{A/\mathfrak{m}_A})$ is a characteristic $0$ vector space.  Therefore, there is some $1$-cochain $c'$ such that $c=d(c')$.  If we define $\tau(g)=c'(g)^{-1}\circ\tau_1(g)$, then $\tau(g_2)\circ g_2^\ast\tau(g_1)=\tau(g_1g_2)$, as desired.

Let $\tau,\tau'$ be two semi-linear actions of $\Gal_{L/K}$ on $D_A$ lifting $\tau_0$.  Then the assignment $g\mapsto c(g):=\tau(g)\circ\tau'(g)^{-1}$ is a $1$-cocycle of $\Gal_{L/K}$, again valued in $\H^0(A,I\otimes_{A/\mathfrak{m}_A}\ad D_{A/\mathfrak{m}_A})$, because 
\begin{eqnarray*}
c(gh) &=& \tau(gh)\circ\tau'(gh)^{-1} 	\\
&=& \tau(h)\circ h^\ast\tau(g) \circ h^\ast\tau'(g)^{-1}\circ \tau'(h)^{-1}	\\
&=&  \tau(h)\circ h^\ast c(g) \circ \tau(h)^{-1}\circ\tau(h)\circ \tau'(h)^{-1}\\
&=& h\cdot c(g) + c(h)
\end{eqnarray*}
where we have switched from multiplicative to additive notation in the last line.
But $\H^1(\Gal_{L/K},\H^0(A,I\otimes_{A/\mathfrak{m}_A}\ad D_{A/\mathfrak{m}_A}))=0$, again because $\Gal_{L/K}$ is a finite group and $\H^0(A,I\otimes_{A/\mathfrak{m}_A}\ad D_{A/\mathfrak{m}_A})$ is a characteristic $0$ vector space, so 
$$c(g)=g\cdot m - m = \tau_0(g)g^\ast m\tau_0(g)^{-1}\circ m^{-1}$$ 
for some $m\in I\otimes_{A/\mathfrak{m}_A}\ad D_{A/\mathfrak{m}_A}$.  Then $m$ is an infinitesimal automorphism of $D_A$ carrying $\tau$ to $\tau'$.
\end{proof}

\begin{lemma}\label{N-def}
Let $D_A$ be a $\Res_{E\otimes L_0/E}G$-torsor over $A$ equipped with a semi-linear action $\tau$ of $\Gal_{L/K}$.  Suppose we are given $N_0\in\ad D_{A/I}$ such that $\underline\Ad(\tau(g))(N_0)=N_0$ for all $g\in\Gal_{L/K}$.  Then there exists $N\in\ad D_A$ lifting $N_0$ such that $\underline\Ad(\tau(g))(N)=N$ for all $g\in \Gal_{L/K}$, and the set of such lifts is a torsor under $I\otimes_{A/\mathfrak{m}_A}(\ad D_{A/\mathfrak{m}_A})^{\Gal_{L/K}}$.
\end{lemma}
\begin{proof}
Suppose that $N$ and $N'$ are two $\Gal_{L/K}$-fixed lifts of $N_0$.  Then 
\[	N-N'\in I\otimes_{A/\mathfrak{m}_A}\ad D_{A/\mathfrak{m}_A}	\]
and
\[	\underline\Ad(\tau(g))(N-N')=N-N'	\]
as desired.

We now address the existence of $N$.  Choose some lift $\widetilde N\in \ad D_A$, and define 
\[	N:=\frac{1}{|\Gal_{L/K}|}\sum_{g\in\Gal_{L/K}}\underline\Ad(\tau(g))(\widetilde N)	\]
Then $N$ lifts $N_0$, and $\underline\Ad(\tau(g))(N)=N$ for all $g\in\Gal_{L/K}$.
\end{proof}

The ``averaging'' technique used here is ubiquitous in the theory of representations of finite groups.

\begin{lemma}\label{phi-def}
Let $D_A$ be a $\Res_{E\otimes L_0/E}G$-torsor over $A$ equipped with a semi-linear action $\tau$ of $\Gal_{L/K}$.  Suppose there is an isomorphism $\Phi_0:\varphi^\ast D_{A/I}\rightarrow D_{A/I}$ such that $\tau(g)\circ g^\ast\Phi_0 = \Phi_0\circ \varphi^\ast\tau(g)$ as isomorphisms $\varphi^\ast g^\ast D_{A/I}\rightarrow D_{A/I}$.  Then there exists an isomorphism $\Phi:\varphi^\ast D_A\rightarrow D_A$ lifting $\Phi_0$ such that $\tau(g) \circ g^\ast\Phi= \Phi\circ \varphi^\ast\tau(g)$, and the set of such lifts is a torsor under a right action of $\H^0(A,I\otimes_{A/\mathfrak{m}_A}\ad D_{A/\mathfrak{m}_A})^{\Gal_{L/K}}$. 
\end{lemma}
\begin{proof}
Suppose first that there are two isomorphisms $\Phi,\Phi':\varphi^\ast D_A\rightrightarrows D_A$ with the desired property.  Then $\Phi\circ\Phi'^{-1}$ is an infinitesimal automorphism of $D_A$ such that for all $g\in\Gal_{L/K}$,
\begin{align*}
\tau(g)\circ g^\ast(\Phi\circ\Phi'^{-1}) &= \Phi\circ \varphi^\ast\tau(g)\circ g^\ast\Phi'^{-1} = (\Phi\circ \Phi'^{-1})\circ \tau(g)
\end{align*}
Thus, $\Phi\circ\Phi'^{-1}\in I\otimes_{A/\mathfrak{m}_A}(\ad D_{A/\mathfrak{m}_A})^{\Gal_{L/K}}$.

We now address the existence of a $\Gal_{L/K}$-fixed lift $\Phi$.  By Lemma~\ref{isom-def}, the set of all lifts $\Phi:\varphi^\ast D_A\rightarrow D_A$ of $\Phi_0:\varphi^\ast D_{A/I}\rightarrow D_{A/I}$ is non-empty, so we choose some lift $\widetilde\Phi$ and again ``average'' it under the action of $\Gal_{L/K}$.  More precisely, for any $g\in\Gal_{L/K}$, $\tau(g) \circ g^\ast\widetilde\Phi\circ \varphi^\ast\tau(g)^{-1}\circ \widetilde\Phi^{-1}$ is an isomorphism $D_A\rightarrow D_A$ which is trivial modulo $I$, so it is an element of $I\otimes_{A/\mathfrak{m}_A}\ad D_{A/\mathfrak{m}_A}$.  Viewing $I\otimes_{A/\mathfrak{m}_A}\ad D_{A/\mathfrak{m}_A}$ additively, we define
$$\Phi:=\left(\frac{1}{|\Gal_{L/K}|}\sum_{g\in\Gal_{L/K}}\tau(g) \circ g^\ast\widetilde\Phi\circ \varphi^\ast\tau(g)^{-1}\circ \widetilde\Phi^{-1}\right)\circ \widetilde\Phi$$
Then for any $h\in\Gal_{L/K}$, 
\begin{align*}
\tau(h)\circ h^\ast\Phi &= \tau(h)\circ \left(\frac{1}{|\Gal_{L/K}|}\sum_{g\in\Gal_{L/K}}\!\!\!\!h^\ast\tau(g) \circ h^\ast g^\ast\widetilde\Phi\circ h^\ast\varphi^\ast\tau(g)^{-1}\circ h^\ast\widetilde\Phi^{-1}\!\right)\!\circ h^\ast\widetilde\Phi	\\
&= \left(\frac{1}{|\Gal_{L/K}|}\sum_{g\in\Gal_{L/K}}\!\!\!\!\tau(gh) \circ (gh)^\ast\widetilde\Phi\circ h^\ast\varphi^\ast\tau(g)^{-1}\circ h^\ast\widetilde\Phi^{-1}\!\right)\!\circ h^\ast\widetilde\Phi	\\
&= \left(\frac{1}{|\Gal_{L/K}|}\sum_{g\in\Gal_{L/K}}\!\!\!\!\tau(gh) \circ (gh)^\ast\widetilde\Phi\circ \varphi^\ast\tau(gh)^{-1}\circ \varphi^\ast\tau(h)\circ h^\ast\widetilde\Phi^{-1}\!\right)\!\circ h^\ast\widetilde\Phi	\\
&= \Phi\circ \varphi^\ast\tau(h)\circ h^\ast\widetilde\Phi^{-1}\circ h^\ast\widetilde\Phi \\
&= \Phi\circ \varphi^\ast\tau(h)
\end{align*}
as desired.
\end{proof}

\begin{remark}
The reader may be concerned that we have asserted the equality of two isomorphisms of $\Res_{E\otimes L_0/E}G$-torsors, one an isomorphism $\varphi^\ast g^\ast D_A\rightarrow D_A$, the other an isomorphism $g^\ast\varphi^\ast D_A\rightarrow D_A$.  But although $\Gal_{L/K}$ may very well be non-abelian, its action on $A\otimes_{\Q_p}L_0$ factors through an abelian quotient and commutes with the action of $\varphi$.  Therefore, $\varphi^\ast g^\ast D_A$ and $g^\ast\varphi^\ast D_A$ are canonically identified.
\end{remark}

Now that we have shown that we can lift $D_{A/I}$ and $\tau_0$ over $A$, uniquely up to isomorphism, and that we can lift $N_0$ and $\Phi_0$ compatibly with $\tau$, we are in a position to prove Proposition~\ref{def-phi-n-tau}.
\begin{proof}[Proof of Proposition~\ref{def-phi-n-tau}]
Let $D_{A/I}$ be a $(\varphi,N,\Gal_{L/K})$-module over $A/I$.  Lemma~\ref{torsor-def} implies that we can lift the underlying $\Res_{E\otimes L_0/E}G$-torsor to a torsor $D_A$ over $A$, and Lemma~\ref{lift-tau} implies that we can lift the action of $\Gal_{L/K}$ to a semi-linear action $\tau$ of $\Gal_{L/K}$ on $D_A$, in both cases uniquely up to isomorphism.  In addition, Lemma~\ref{N-def} implies that we can lift $N_0$ to $N\in \ad D_A$ and Lemma~\ref{phi-def} implies that we can lift $\Phi_0$ to $\Phi:\varphi^\ast D_A\rightarrow D_A$, such that $N$ and $\Phi$ are $\Gal_{L/K}$-fixed.

Then $D_A$, together with $\tau$, $N$, and $\Phi$, is a $(\varphi,N,\Gal_{L/K})$-module if and only if $N=p\underline\Ad(\Phi) (N)$.  Define $h=N-p\underline\Ad(\Phi)(N)\in I\otimes_{A/\mathfrak{m}_A}\ad D_{A/\mathfrak{m}_A}^{\Gal_{L/K}}$.  If $\H^2(D_{A/\mathfrak{m}_A})=0$, then there exist $f,g\in I\otimes_{A/\mathfrak{m}_A}\ad D_{A/\mathfrak{m}_A}^{\Gal_{L/K}}$ such that $h=\ad_{N_0}(f)+(p\underline\Ad(\Phi_0)-1)(g)$.  Then we claim that if we define $\widetilde N:=N+g$ and $\widetilde\Phi:=f^{-1}\circ\Phi$, then $\tilde N=p\underline\Ad(\widetilde\Phi)(\widetilde N)$.  Note that we are going back and forth between the ``additive'' and ``multiplicative'' interpretations of $I\otimes_{A/\mathfrak{m}_A}D_{A/\mathfrak{m}_A}^{\Gal_{L/K}}$.  But $\widetilde N=p\underline\Ad(\widetilde\Phi)(\widetilde N)$ holds after pushing out $D_A$ by any representation $G\rightarrow \GL(V)$ over $E$, by the construction in the proof of~\cite[Proposition 3.1.2]{kisin}, so it holds in $\ad D_A$.

Suppose that $D_A$ and $D_A'$ are two lifts of $D_{A/I}$ as $(\varphi,N,\Gal_{L/K})$-modules.  We may assume that the underlying torsors and actions of $\Gal_{L/K}$ are identified.  Then $g:=N-N'$ and $f:=\Phi\circ \Phi'^{-1}$ are elements of $I\otimes_{A/\mathfrak{m}_A}D_{A/\mathfrak{m}_A}^{\Gal_{L/K}}$, and 
\[	\ad_{N_0}(f)+(p\underline\Ad(\Phi_0) -1)(g)=0	\]
because this holds after pushing out $D_A$ by any representation $G\rightarrow \GL(V)$ over $E$, by the construction in the proof of~\cite[Proposition 3.1.2]{kisin}.  Therefore, $(f,g)$ represents a class in $I\otimes_{A/\mathfrak{m}_A}\H^1(D_{A/\mathfrak{m}_A})$.

Now $D_A$ and $D_A'$ are isomorphic if and only if there is some $\Gal_{L/K}$-invariant automorphism $u$ of the underlying torsor of $D_A$ which is the identity modulo $I$, and carries $N$ to $N'$ and $\Phi$ to $\Phi'$.  That is, $u\in I\otimes_{A/\mathfrak{m}_A}D_{A/\mathfrak{m}_A}^{\Gal_{L/K}}$, and $\Ad(u)(N)=N'$ and $u\circ\Phi=\Phi'\circ\varphi^\ast u$.  But $\Ad(u)(N)=N'$ if and only if $N'-N=\ad_{N_0}(u)$, and $u\circ\Phi=\Phi'\circ\varphi^\ast u$ if and only if $\Phi\circ\Phi^{-1}=u^{-1}\circ \underline\Ad(\Phi')(u)$.  Thus, $D_A$ and $D_A'$ are isomorphic if and only if $(N-N',\Phi\circ\Phi^{-1})$ is in the image of $d^0$.
\end{proof}

\begin{remark}\label{def-thy-framed}
The proof of Proposition~\ref{def-phi-n-tau} also shows that if we fix a particular lift $D_A$ of the underlying $\Res_{E\otimes L_0/E}G$-torsors and a particular lift $\tau$ of $\tau_0$, then the space of lifts $(\Phi,N)$ of $(\Phi_0,N_0)$ to $D_A$ compatible with $\tau$ such that $N=p\underline\Ad(\Phi)(N)$ is either empty or a torsor under 
$$\ker\left((\ad D_{A/\mathfrak{m}_A}\otimes_{A/\mathfrak{m}_A}I)\oplus (\ad D_{A/\mathfrak{m}_A}\otimes_{A/\mathfrak{m}_A}I)\rightarrow(\ad D_{A/\mathfrak{m}_A}\otimes_{A/\mathfrak{m}_A}I)\right)$$
This differs from the statement of Proposition~\ref{def-phi-n-tau} in that we are considering the space of lifts, not the space of lifts up to isomorphism.  We will use this observation in Section~\ref{reducedness} to compute the dimension of a cover of a particular moduli stack.
\end{remark}

We now turn to the question of deforming \emph{filtered} $(\varphi,N,\Gal_{L/K})$-modules.  

\begin{lemma}
Let $A$ be a henselian local $E$-algebra with maximal ideal $\mathfrak{m}_A$, and let $I\subset A$ be an ideal with $I\mathfrak{m}_A=0$.  Let $H$ be a reductive algebraic group over $E$, and let $D_A$ be an $H$-torsor over $A$, and suppose that the reduction $D_{A/I}$ of $D_A$ modulo $I$ is equipped with a $\otimes$-filtration $\mathcal{F}_0^\bullet$.  Then there is a $\otimes$-filtration on $D_A$ lifting it, and the space of such lifts is a torsor under $(\ad D_{A/\mathfrak{m}_A}/\Fil^0\ad D_{A/\mathfrak{m}_A})\otimes_{A/\mathfrak{m}_A}I$.
\end{lemma}
\begin{proof}
Suppose first that there are two $\otimes$-filtrations, $\mathcal{F}^\bullet$ and ${\mathcal{F}'}^\bullet$, lifting $\mathcal{F}_0^\bullet$.  Let $\lambda_0:\Gm\rightarrow\Aut_H(D_{A/I})$ be a cocharacter splitting $\mathcal{F}_0^\bullet$, let $\lambda,\lambda':\Gm\rightrightarrows\Aut_H(D_A)$ be cocharacters splitting $\mathcal{F}^\bullet$ and ${\mathcal{F}'}^\bullet$, respectively, and let $P,P'\subset \Aut_H(D_A)$ be the corresponding parabolics.  We may arrange that $\lambda$ and $\lambda'$ reduce to $\lambda_0$ modulo $I$.  By the local constancy of the type of a filtration, $\lambda$ and $\lambda'$ are conjugate by an infinitesimal automorphism of $D_A$ (since they are equal modulo $I$), and $\mathcal{F}^\bullet={\mathcal{F}'}^\bullet$ if and only if $\lambda$ and $\lambda'$ are $\Fil^0\ad D_{A/\mathfrak{m}_A}\otimes_{A/\mathfrak{m}_A}I$-conjugate.
\end{proof}

We define the groupoid $\mathfrak{Mod}_{F,\varphi,N,\tau}$ of $G$-valued filtered $(\varphi,N,\Gal_{L/K})$-modules of $p$-adic Hodge type $\mathbf{v}$ and Galois type $\tau$ to be the groupoid on the category of $E$-algebras whose fiber over $A$ is the category of $G$-torsors $D_A$ over $A\otimes L_0$ equipped with $\Phi$, $N$, and $\tau$ making $D_A$ into an object of $\mathfrak{Mod}_{\varphi,N,\tau}$, and such that the $G$-torsor $(D_A)_L^{\Gal_{L/K}}$ over $A\otimes K$ is equipped with a $\otimes$-filtration of type $\mathbf{v}$.  The deformation theory of an object $D_A$ of $\mathfrak{Mod}_{F,\varphi,N,\tau}(A)$ is controlled by the total complex $C_F^\bullet(D_A)$ of the double complex
\begin{equation*}
\xymatrix{ 
(\ad D_A)^{\Gal_{L/K}} \ar[r]\ar[d] &  (\ad D_A)^{\Gal_{L/K}}\oplus(\ad D_A)^{\Gal_{L/K}}\ar[r] & (\ad D_A)^{\Gal_{L/K}} \\
(\ad D_{A,L}/\!\Fil^0\!\ad D_{A,L})^{\Gal_{L/K}} & &
}
\end{equation*}
where the top row is the complex $C^\bullet(D_A)$.  We let $\H_F^\bullet(D_A)$ denote the cohomology of $C_F^\bullet(D_A)$.  

Note that since $(\Res_{E\otimes K/E}G)/P_\mathbf{v}$ is smooth and the $\otimes$-filtration on $(D_A)_L^{\Gal_{L/K}}$ does not interact with the $(\varphi,N,\Gal_{L/K})$-module structure, any obstruction to deforming $D_{A/I}$ as a filtered $(\varphi,N,\Gal_{L/K})$-module comes from an obstruction to deforming it as a $(\varphi,N,\Gal_{L/K})$-module. 

More precisely, we have the following result, following~\cite[Lemma 3.2.1]{kisin}:
\begin{prop}\label{def-phi-n-tau-filt}
The natural morphism of groupoids $\mathfrak{Mod}_{F,\varphi,N,\tau}\rightarrow\mathfrak{Mod}_{\varphi,N,\tau}$ given by forgetting the $\otimes$-filtration is formally smooth.  Furthermore, let $A$ be a henselian local ring with maximal ideal $\mathfrak{m}_A$ and an ideal $I\subset A$ with $I\mathfrak{m}_A=0$.  Let $D_{A/I}$ be an object of $\mathfrak{Mod}_{F,\varphi,N,\tau}(A/I)$ and set $D_{A/\mathfrak{m}_A}=D_{A/I}\otimes_{A/I}A/\mathfrak{m}_A$.  Then
\begin{enumerate}
\item	If $\H_F^2(D_{A/\mathfrak{m}_A})=0$ then there exists $D_A$ in $\mathfrak{Mod}_{F,\varphi,N,\tau}(A)$ lifting $D_{A/I}$.
\item	The set of isomorphism classes of liftings of $D_{A/I}$ over $A$ is either empty or a torsor under the cohomology group $\H_F^1(D_{A/\mathfrak{m}_A})\otimes_{A/\mathfrak{m}_A}I$.
\end{enumerate}
\end{prop}
\begin{proof}
The assertion about formal smoothness follows from the smoothness of the quotient $\Res_{E\otimes K/E}G/P_\mathbf{v}$.  Similarly, if $\H_F^2(D_{A/\mathfrak{m}_A})=0$, then $\H^2(D_{A/\mathfrak{m}_A})=0$ and a lift $D_A$ of $D_{A/I}$ (as $G$-valued $(\varphi,N,\Gal_{L/K})$-modules) exists.  But $\Res_{E\otimes K/E}G/P_\mathbf{v}$ is smooth, so we can lift the $\otimes$-filtration on $(D_{A/I})_L^{\Gal_{L/K}}$ to a $\otimes$-filtration on $(D_A)_L^{\Gal_{L/K}}$.

For the last claim, we note that two lifts, $D_A$ and $D_A'$, of $D_{A/I}$ (as filtered $(\varphi,N,\Gal_{L/K})$-modules) if and only if the infinitesimal automorphism $u$ of the underlying torsor of $D_A$ can be chosen to carry the $\otimes$-filtrations to each other.  But this is the case if and only if its image in $\left(\ad D_{A/\mathfrak{m}_A,L}^{\Gal_{L/K}}/\Fil^0\ad D_{A/\mathfrak{m}_A,L}^{\Gal_{L/K}}\right)\otimes_{A/\mathfrak{m}_A}I$ is trivial.
\end{proof}

\section{Associated cocharacters}\label{associated}

Let $G$ be a connected reductive group over a field $k$, and let $N$ be a nilpotent element of $\mathfrak{g}:=\Lie(G)$.  That is, for any finite dimensional representation $\rho:G\rightarrow \GL(V)$, the pushforward $\rho_\ast N$ is a nilpotent element of $\mathfrak{gl}(V)$.  If $G$ is semisimple, this is equivalent to the condition that $\ad_N:\mathfrak{g}\rightarrow\mathfrak{g}$ be a nilpotent operator.  We briefly review the theory of ``associated cocharacters'' for $N$; we refer the reader to~\cite{jantzen} for further details and proofs, particularly section $5$.  We are only interested in the case when the characteristic of $k$ is $0$, but we review the theory for positive characteristic as well.  We assume the ground field is algebraically closed.

Let $L$ be a Levi factor of a parabolic subgroup of $G$, and suppose that 
$$N\in \mathfrak{l}:=\Lie(L)$$
\begin{definition}
$N$ is \emph{distinguished} in $\mathfrak{l}$ if every torus contained in $Z_L(N)$ is contained in the center of $L$.
\end{definition}

Intuitively, if $N$ is distinguished in $\mathfrak{l}$, $L$ should be the ``smallest'' Levi whose Lie algebra intersects the orbit of $N$ (under conjugation) nontrivially.  For example, if $G=\GL_3$, then 
\[	N=\left(\begin{smallmatrix}0&1&0\\0&0&0\\0&0&0\end{smallmatrix}\right)	\]
is not distinguished in $\mathfrak{gl}_3$, but it is distinguished in the Lie algebra of 
\[	L=\left\{\left(\begin{smallmatrix}\ast&\ast&0\\ \ast&\ast&0\\0&0&\ast\end{smallmatrix}\right)\right\}	\]
Indeed,
\[	Z_G(N)=\left\{\left(\begin{smallmatrix}x&\ast&\ast\\ 0&x&0\\0&\ast&\ast\end{smallmatrix}\right)\right\}	\]
which certainly contains a non-central torus of $\GL_3$, whereas
\[	Z_G(N)\cap L=\left\{\left(\begin{smallmatrix}x&\ast&0\\ 0&x&0\\0&0&\ast\end{smallmatrix}\right)\right\}	\]

Every nilpotent element $N\in\mathfrak{g}$ is distinguished in some Levi subgroup of $G$.  In fact, 
\begin{lemma}[{\cite[4.6]{jantzen}}]
Let $N\in\mathfrak{g}$ be nilpotent and let $T$ be a maximal torus in $Z_G(N)^\circ$.  Then the centralizer $L$ of $T$ in $G$ is a Levi subgroup, and $N$ is a distinguished nilpotent element of $\mathfrak{l}$.
\end{lemma}

\begin{definition}
A cocharacter $\lambda:\mathbf{G}_m\rightarrow G$ is said to be \emph{associated} to $N$ if $\Ad(\lambda(t))(N)=t^2N$ and there is a Levi subgroup $L\subset G$ such that $N$ is distinguished nilpotent in $\mathfrak{l}$ and $\lambda$ factors through the derived group $\mathcal{D}L$ of $L$.
\end{definition}

\begin{lemma}[{\cite[5.3]{jantzen}}]
\begin{enumerate}
\item	If the characteristic of $k$ is good for $G$, then cocharacters associated to $N$ exist.
\item	Any two cocharacters associated to $N$ are conjugate under $Z_G(N)^\circ$.
\end{enumerate}
\end{lemma}
Since characteristic $0$ is a good characteristic for all $G$, associated cocharacters will exist for us.  From now on, we assume that the characteristic of $k$ is good for $G$.

\begin{prop}[{\cite[5.5]{jantzen}}]\label{assoc-jacobson-morozov}
Suppose the characteristic is $0$.  Let $N\in\mathfrak{g}$ be a nilpotent element of the Lie algebra of $G$.  Then $\lambda\mapsto d\lambda(1)$ is a bijection from the set of cocharacters associated to $N$ to the set of $X\in[N,\mathfrak{g}]$ such that $[X,N]=2N$.
\end{prop}

\begin{remark}
Suppose we are working over an algebraically closed field of characteristic $0$.  The theorem of Jacobson-Morozov states that if $N\in\mathfrak{g}$ is nilpotent and non-zero, there exist $H,Y\in \mathfrak{g}$ such that $(N,H,Y)$ form an $\mathfrak{sl}_2$-triple inside $\mathfrak{g}$, and they are unique up to conjugation by $Z_{\mathfrak{g}}(N)$.  Then since $\SL_2$ is simply connected, we can exponentiate the map $\mathfrak{sl}_2\rightarrow \mathfrak{g}$ to get a map $\SL_2\rightarrow G$.  Composing with the standard diagonal torus $\mathbf{G}_m\rightarrow \SL_2$ turns out to yield an associated cocharacter $\lambda:\mathbf{G}_m\rightarrow G$.
\end{remark}

\begin{remark}
McNinch has relaxed the requirement that the ground field be algebraically closed.  He has shown~\cite[Theorem 26]{mcninch} that over a perfect ground field $F$ of characteristic good for $G$, if $N\in\mathfrak{g}(F)$ is nilpotent and non-zero, there is an $F$-rational cocharacter associated to $N$.  We will not need this here, however.
\end{remark}

Given an associated cocharacter $\lambda$ of $N$, we may construct the associated parabolic $P_G(\lambda)=U_G(\lambda)\rtimes Z_G(\lambda)$, where $U_G(\lambda)$ is the unipotent radical of $P_G(\lambda)$ and $Z_G(\lambda)=P_G(\lambda)/U_G(\lambda)$ is reductive.

Note that if $\lambda$ is associated to $N$, then $N$ lies in the weight $2$ part of the $\lambda$-grading on the Lie algebra $\mathfrak{p}$ of $P_G(\lambda)$, while the Lie algebra of $Z_G(\lambda)$ is by definition the weight $0$ part.  Thus, although we have two Levi subgroups of $G$ floating around, $N$ is not contained in the Lie algebra of $Z_G(\lambda)$, let alone distinguished there.

\begin{prop}[{\cite[5.9,5.10,5.11]{jantzen}}]
\begin{enumerate}
\item	The associated parabolic $P_G(\lambda)$ depends only on $N$, not on the choice of associated cocharacter.
\item	We have $Z_G(N)\subset P_G(\lambda)$.  In particular, $Z_G(N)=Z_P(N)$.
\item	$Z_G(N)=\left(U_G(\lambda)\cap Z_G(N)\right)\rtimes\left(Z_G(\lambda)\cap Z_G(N)\right)$
\item	$Z_G(\lambda)\cap Z_G(N)$ is reductive.
\end{enumerate}
\end{prop}

\begin{prop}\label{disconn}
Let $G$ be a connected reductive group over an algebraically closed field of characteristic $0$.  Let $G'$ be a possibly disconnected reductive subgroup of $G$, and suppose $N\in\Lie G'$.  Let $\lambda:\mathbf{G}_m\rightarrow (G')^\circ$ be an associated cocharacter of $N$.  Then $Z_{G'}(N)=\left(U_{G'}(\lambda)\cap Z_{G'}(N)\right)\rtimes\left(Z_{G'}(\lambda)\cap Z_{G'}(N)\right)$ and $U_{G'}(\lambda)\cap Z_{G'}(N)$ is connected.
\end{prop}
\begin{proof}
We first claim that $\lambda:\mathbf{G}_m\rightarrow (G')^\circ\rightarrow G$ is associated to $N$ as a cocharacter of $G$.  Since we are in characteristic $0$, we may use Proposition~\ref{assoc-jacobson-morozov} to pass freely between associated cocharacters and Jacobson-Morozov triples.  More precisely, $d\lambda(1)$ satisfies $[d\lambda(1),N]=2N$ whether we view $d\lambda(1)$ as an element of $\Lie G'$ or of $\Lie G$, and if $d\lambda(1)\in[N,\mathfrak{g}']$, then $d\lambda(1)\in[N,\mathfrak{g}]$ as well.  Therefore, $\lambda$ is associated to $N$ as a cocharacter of $G$.

We now consider the structure of $Z_{G'}(N)$ and $U_{G'}(\lambda)\cap Z_{G'}(N)$.  We know that $Z_G(N)\subset P_G(\lambda)$, and in fact, $Z_G(N)$ is the semi-direct product of $U_G(\lambda)\cap Z_G(N)$ and $Z_G(\lambda)\cap Z_G(N)$.  It follows that $Z_{G'}(N)\subset P_{G'}(\lambda)$ and in fact, $Z_{G'}(N)$ is the semi-direct product of $U_{G'}(\lambda)\cap Z_{G'}(N)$ and $Z_{G'}(\lambda)\cap Z_{G'}(N)$.

For the second assertion, we observe that $\lambda$ normalizes
$U_{G'}(\lambda)\cap Z_{G'}(N)$ (by the normality of $U_{G'}(\lambda)$ in $P_{G'}(\lambda)$ and the definition of an associated cocharacter), so everything in $U_{G'}(\lambda)\cap Z_{G'}(N)$ is connected by a copy of $\mathbf{A}^1$ to the identity.  So $U_{G'}(\lambda)\cap Z_{G'}(N)$ is connected.
\end{proof}

We record a few results about families of nilpotent elements of $\mathfrak{g}$.  We let $\mathcal{N}$ denote the space of nilpotent elements of $\mathfrak{g}$.  That is, for any $k$-algebra $A$, $\mathcal{N}(A)$ is the set of elements $N\in\mathfrak{g}(A)$ such that for every representation $\sigma:G\rightarrow \GL_n$, the characteristic polynomial of $\sigma_\ast(N)$ is $T^d$.

\begin{lemma}\label{fam-section}
Let $N\in\mathcal{N}(k)$.  There is an fppf neighborhood $U\rightarrow G\cdot N$ of $N$ and a section $s:U\rightarrow G_U$ such that the $U$-pullback of the universal nilpotent element of $\mathfrak{g}$ over the orbit $G\cdot N$ is of the form $\Ad(s(U))(N)$.  If the characteristic of $k$ is $0$, $U$ may be taken to be an \'etale neighborhood.
\end{lemma}
\begin{proof}
We have a morphism $G\rightarrow G\cdot N$, given by $g\mapsto \Ad(g)N$.  This is the structure morphism of a $Z_G(N)$-bundle on $G\cdot N$.  Fppf-locally on $G\cdot N$ (or \'etale-locally in characteristic $0$, since in that case $Z_G(N)$ is smooth), this structure morphism admits a section, which by definition has the desired property.
\end{proof}

\begin{cor}\label{fam-cent}
Let $S$ be a reduced scheme over a characteristic $0$ field $k$, and let $N\in \mathcal{N}(S)$ be a family of nilpotent elements such that for every geometric point $\overline{s}$ of $S$, the conjugacy class of $N_{\overline{s}}$ in $\Lie G$ is constant.  Then the centralizer $Z_{G_S}(N)\subset G_S$ is smooth over $S$.
\end{cor}
\begin{proof}
The nilpotent family $N$ is some morphism $f:S\rightarrow \mathcal{N}$, and the constancy of the conjugacy classes and reducedness of $S$ imply that $f$ factors through some orbit $G\cdot N_0$.  Thus, we may assume that $S=G\cdot N_0$ and $N$ is the universal nilpotent element.  For any point $x\in G\cdot N_0$, there is some \'etale neighborhood $U\rightarrow G\cdot N_0$ of $x$ and a section $s:U\rightarrow G_U$ such that the restriction of the universal nilpotent element to $U$ is of the form $\Ad(s(U))N$.  Therefore, $Z_{G_S}(N)|_U=s(U)Z_G(N)_Us(U)^{-1}$, which is visibly smooth over $U$.
\end{proof}

\begin{cor}\label{fam-assoc}
Let $S$ be a reduced scheme over a characteristic $0$ field $k$, and let $N\in \mathcal{N}(S)$ be a family of nilpotent elements such that for every geometric point $\overline{s}$ of $S$, the conjugacy class of $N_{\overline{s}}$ in $\Lie G$ is constant.  Then \'etale-locally on $S$, there is a family of cocharacters $\lambda:(\Gm)_S\rightarrow G_S$ such that for each $\overline{s}$, $\lambda_{\overline{s}}$ is an associated cocharacter for $N_{\overline{s}}$.
\end{cor}
\begin{proof}
As before, we reduce to the case of a family of nilpotent elements over $U$ of the form $\Ad(s(U))N_0$ for some section $s:U\rightarrow G_U$ and some $N_0\in\mathcal{N}$.  Then if $\lambda_0:\Gm\rightarrow G$ is an associated cocharacter for $N_0$, we define $\lambda:(\Gm)_U\rightarrow G_U$ to be $s(U)\lambda_0s(U)^{-1}$.
\end{proof}

\section{Reducedness of $X_{\varphi,N,\tau}$}\label{reducedness}

We will use the theory of associated cocharacters to study the structure of a cover $X_{\varphi,N,\tau}$ of $\mathfrak{Mod}_{\phi,N,\tau}$ in the case when $L/K$ is totally ramified.  In that case, $L_0=K_0$ with $f:=[L_0:\Q_p]$, and $\tau$ is a linear representation of $\Gal_{L/K}$.  We will first show that this cover is generically smooth, and the method of proof will let us calculate its dimension.  We will then be able to see that this cover is actually a local complete intersection.  Since a local complete intersection which is generically reduced is reduced everywhere, we conclude that our cover is actually reduced.

Let $E/\Q_p$ be a finite extension, and let $D_E$ be a $\Res_{E\otimes L_0/E}G$-torsor over $E$ equipped with a choice of trivializing section, i.e., a copy of $\Res_{E\otimes L_0/E}G$.  For any $E$-algebra $A$, let $\Rep_A\Gal_{L/K}$ denote the set of representations of $\Gal_{L/K}$ on $\Res_{E\otimes L_0/E}G(A)=G(A\otimes L_0)$.

Let $X_{\varphi,N,\tau}$ denote the functor on the category of $E$-algebras whose $A$-points are triples 
\[	(\Phi,N,\tau)\in (\Res_{E\otimes L_0/E}G)(A)\times (\Res_{E\otimes L_0/E}\mathfrak{g})(A)\times\Rep_A\Gal_{L/K}	\]
which satisfy $N=p\Ad(\Phi)(\varphi(N))$, $\tau(g)\circ\Phi=\Phi\circ\tau(g)$, and $N={\Ad}(\tau(g))(N)$ for all $g\in\Gal_{L/K}$.  Here $\varphi$ denotes the Frobenius on the coefficients.

Similarly, let $X_{N,\tau}$ denote the functor on the category of $E$-algebras parametrizing pairs 
\[	(N,\tau)\in (\Res_{E\otimes L_0/E}\mathfrak{g})(A)\times\Rep_A\Gal_{L/K}	\]
such that $N={\Ad}(\tau(g))(N)$ for all $g\in\Gal_{L/K}$.  There is a natural map $X_{\varphi,N,\tau}\rightarrow X_{N,\tau}$ given by forgetting $\Phi$.

There is a third functor $X_\tau$ on the category of $E$-algebras whose $A$-points are representations $\tau:\Gal_{L/K}\rightarrow G(A\otimes L_0)$, and there is a forgetful map $X_{N,\tau}\rightarrow X_\tau$.

All three of these functors are visibly representable by finite-type affine schemes over $E$, which we also denote by $X_{\varphi,N,\tau}$, $X_{N,\tau}$, and $X_\tau$.  Moreover, there is a left action of $\Res_{E\otimes L_0/E}G$ on $X_{\varphi,N,\tau}$ coming from changing the choice of trivializing section.  Explicitly,
\[	a\cdot (\Phi,N,\{\tau(g)\}_{g\in\Gal_{L/K}}) = (a\Phi\varphi(a)^{-1},\Ad(a)(N),\{a\tau(g)a^{-1}\}_{g\in\Gal_{L/K}})	\]
The quotient of $X_{\varphi,N,\tau}$ by this action---``forgetting the framing''---is $\mathfrak{Mod}_{\varphi,N,\tau}$.

\begin{remark}
We have assumed that $L/K$ is totally ramified, so $\Gal_{L/K}$ acts trivially on the coefficients.  Thus, when we write ``$\Ad(\tau(g))(N)$'', we literally mean the adjoint action of $\Res_{E\otimes L_0/E}G$ on its own Lie algebra, not a twisted adjoint action.
\end{remark}

Consider the coherent sheaf $\mathcal{H}^2$ on $X_{\varphi,N,\tau}$ given by the cokernel of 
\[	(\ad D_A)^{\Gal_{L/K}}\oplus(\ad D_A)^{\Gal_{L/K}}\xrightarrow{(p\Phi-1)\oplus \ad_N} (\ad D_A)^{\Gal_{L/K}}	\]
At any closed point $x\in X_{\varphi,N,\tau}$, the specialization $\mathcal{H}^2(x)$ is $\H^2(D_{A\otimes k(x)}$ from Proposition~\ref{def-phi-n-tau} which controls the obstruction theory of the corresponding $(\varphi,N,\Gal_{L/K})$-module.  Therefore, the locus in $X_{\varphi,N,\tau}$ where $\mathcal H^2$ vanishes is open, and to show that $X_{\varphi,N,\tau}$ is generically smooth it suffices to show that this locus is dense.

\begin{prop}
There is a dense open subscheme of $X_{\varphi,N,\tau}$ where $\mathcal H^2$ vanishes.
\end{prop}

Before we begin, we remind the reader that $\Ad$ refers to a literal adjoint action of an algebraic group on its Lie algebra, while $\underline\Ad$ refers to a Frobenius-semilinear action.

\begin{proof}
We begin by extending scalars of $X_{\varphi,N,\tau}$ from $E$ to $\overline E$.  Then the $\overline{E}$-points of $X_{\varphi,N,\tau}$ correspond to triples of $f$-tuples $\underline{\Phi}=(\Phi_1,\ldots,\Phi_f)$, $\underline{N}=(N_1,\ldots,N_f)$, and $\underline{\tau}=(\tau_1,\ldots,\tau_f)$ with $\Phi_i\in G(\overline{E})$, $N_i\in\mathfrak{g}(\overline{E})$, and $\tau_i:\Gal_{L/K}\rightarrow G(\overline{E})$ a representation, which are required to satisfy
\begin{eqnarray*}
N_i&=&p\Ad(\Phi_i)(N_{i+1})	\\
N_i&=&\Ad(\tau_i(g))(N_i)	\\
\tau_i(g)\circ \Phi_i&=&\Phi_i\circ\tau_{i+1}(g)
\end{eqnarray*}
for all $i$ and all $g\in\Gal_{L/K}$. 

It suffices to show that $\mathcal H^2$ vanishes on a dense open subset of each non-empty fiber of $X_{\varphi,N,\tau}\rightarrow X_\tau$.  Moreover, the condition ``$\mathcal H^2$ vanishes at $y\in X_{\varphi,N,\tau}$'' is invariant under the action of $\Res_{E\otimes L_0/E}G$ on $X_{\varphi,N,\tau}$.  The compatibilities between $\underline{\Phi}$, $\underline{N}$, and $\underline{\tau}$ imply that if $k/\overline{E}$ is an extension of fields and the fiber over $z\in X_\tau(k)$ is non-empty, the representation $\underline{\tau}$ corresponding to $z$ has the property that the Frobenius-conjugates of the $\tau_i$ are $G(k)$-conjugate.  Thus, letting $\underline{a}=(1,\Phi_1,\Phi_1\Phi_2,\ldots)$ and replacing $(\underline\Phi,\underline N,\underline\tau)$ with $\underline{a}\cdot(\underline\Phi,\underline N,\underline\tau)$, we may assume that $\underline{\tau}=(\tau,\ldots,\tau)$ for some representation $\tau:\Gal_{L/K}\rightarrow G(k)$.

Let $X_{\varphi,N}$ denote the fiber of $X_{\varphi,N,\tau}$ over the point corresponding to $\underline{\tau}$, so that $X_{\varphi,N}$ parametrizes pairs of $f$-tuples $\underline{\Phi}$ and $\underline{N}$ such that $\Phi_i\in Z_G(\tau)$, $N_i\in \Lie Z_G(\tau)$, and $N_i=p\Ad(\Phi_i)(N_{i+1})$ for all $i$.  Let $X_N$ denote the fiber of $X_{N,\tau}$ over the point corresponding to $\underline\tau$, so that $X_N$ parametrizes $f$-tuples $\underline{N}$ with $N_i\in \Lie Z_G(\tau)$.  There is a forgetful map $X_{\varphi,N}\rightarrow X_N$, and to show that $\mathcal H^2$ vanishes generically on $X_{\varphi,N}$, it suffices to show that it vanishes on a dense subset of each non-empty fiber.  Since $\mathcal H^2$ vanishes on a Zariski-open set, it in fact suffices to find a point on each connected component of each non-empty fiber.  Note that although $Z_G(\tau)$ is reductive by~\cite[Theorem 2.2]{humphreys}, it will not generally be connected.

The compatibility between $\underline\Phi$ and $\underline N$ implies that either $N_i=0$ for all $i$ or $N_i\neq 0$ for all $i$.  We first treat the case where $N_i\neq 0$ for all $i$.  For each $N_i$, choose an associated cocharacter $\lambda_i:\mathbf{G}_m\rightarrow Z_G(\tau)^\circ$.  Then if $(\underline\Phi,\underline N)$ corresponds to a $k'$-point of $X_{\varphi,N}$ for some extension $k'/k$, we have
\[	\Ad(\lambda_i(p^{-1/2}))(N_i)= p^{-1}N_i = \Ad(\Phi_i)(N_{i+1})	\]
by the compatibility between $\underline{\Phi}$ and $\underline{N}$.  We see that if $\underline{N}$ corresponds to a point $y\in X_N$ with non-empty fiber, the Frobenius conjugates of the $N_i$ are $Z_G(\tau)(\overline{E})$-conjugate.  Letting $\underline{b}=(\Phi_1^{-1}\lambda_1(p^{-1/2}),\ldots,\Phi_f^{-1}\lambda_f(p^{-1/2}))$ and replacing $(\underline\Phi,\underline N,\underline\tau)$ with $\underline{b}\cdot(\underline\Phi,\underline N,\underline\tau)$, we may assume that $\underline{N}=(N,\ldots,N)$ for some $N\in \Lie Z_G(\tau)$.

The fiber of $X_{\varphi,N}\rightarrow X_N$ over $\underline{N}$ is a coset in $Z_G(\tau)^{\times f}$ of $\left(Z_G(N)\cap Z_G(\tau)\right)^{\times f}$.  This will generally be disconnected, even if $Z_G(\tau)$ is connected.  We will find a point on each component of the fiber over $\underline{N}$ where $\mathcal H^2$ vanishes.

Choose an associated cocharacter $\lambda:\mathbf{G}_m\rightarrow Z_G(\tau)^\circ$ for $N$, and let $\Phi_0$ denote $\lambda(p^{-1/2})$.  Let $\underline{\Phi}_0=(\Phi_0,\ldots,\Phi_0)$.  We need to analyze the maps
\[	p\underline\Ad(\underline{\Phi})_0-1:\ad D_{k'}\rightarrow \ad D_{k'}	\]
and
\[	\ad_{\underline{N}}:\ad D_{k'}\rightarrow \ad D_{k'}	\]
Here $\ad D_{k'}$ is $\mathfrak{g}_{k'}^{\times f}$, and $\underline\Ad(\underline{\Phi})$ acts by
\[	(X_1,\ldots,X_f)\mapsto (\Ad(\Phi_1)(X_2),\ldots,\Ad(\Phi_f)(X_1))	\]
and $\ad_{\underline{N}}$ acts by
\[	(X_1,\ldots,X_f)\mapsto (\ad_{N_1}(X_1),\ldots,\ad_{N_f}(X_f))	\]

Each factor $\mathfrak{g}_{k'}$ is graded by $\lambda$; $p\underline\Ad(\underline{\Phi}_0)-1$ is a semi-simple endomorphism of $\ad D_{k'}$ since it is the difference of commuting semi-simple operators.  Further, $p\underline\Ad(\underline{\Phi}_0)-1$ acts invertibly on $\mathfrak{g}_{k'}$, except on the weight-$2$ eigenspaces of the factors, where it is $0$. Therefore, the cokernel of $p\underline\Ad(\underline{\Phi}_0)-1$ is the direct sum of the weight $2$ eigenspaces of the factors.  But by the representation theory of $\mathfrak{sl}_2$, the weight $2$ part of $\mathfrak{g}_{k'}$ is in the image of $\ad_N$.  Thus, $\mathcal H^2=0$ at the point corresponding to $(\underline{\Phi}_0,\underline{N},\underline{\tau})$.

It remains to find points where $\mathcal H^2$ vanishes on any other components of the fiber of $X_{\varphi,N}\rightarrow X_{N}$ over the point corresponding to $(\underline N,\underline\tau)$.  This fiber is a torsor under the action of $\left(Z_G(\tau)\cap Z_G(N)\right)^{\times f}$, the centralizer of $\underline N$ in $\left(Z_G(\tau)\right)^{\times f}$, so connected components of the fiber correspond to connected components of $\left(Z_G(\tau)\cap Z_G(N)\right)^{\times f}$.  By Proposition~\ref{disconn}, the disconnectedness of $Z_G(\tau)\cap Z_G(N)$ is entirely accounted for by the disconnectedness of $Z_G(\tau)\cap Z_G(\lambda)\cap Z_G(N)$.  For any component of $\left(Z_G(\tau)\cap Z_G(N)\right)^{\times f}$, we may therefore choose a representative $\underline{c}:=(c_1,\ldots,c_f)$ with the $c_i\in Z_G(\tau)\cap Z_G(\lambda)\cap Z_G(N)$.  

Furthermore, by Lemma~\ref{finrep}, whose statement and proof we postpone, there is a finite-order point on the component of $c_1\cdots c_f$ of $Z_G(\tau)\cap Z_G(\lambda)\cap Z_G(N)$.  By adjusting $c_f$ by a point on the connected component of the identity $\left(Z_G(\tau)\cap Z_G(\lambda)\cap Z_G(N)\right)^\circ$, we may replace $\underline{c}$ by another point (on the same connected component) with $c_1\cdots c_f$ finite order.

We put $\underline{\Phi}:=\underline{\Phi}_0\cdot \underline{c}$, and we claim that the $k'$-linear endomorphism $p\underline\Ad(\underline{\Phi})-1$ on $\mathfrak{g}^{\times f}$ is semi-simple.  It suffices to show that the endomorphism $p\underline\Ad(\underline{\Phi})$ is semi-simple.  The $f$th iterate of this map is 
\[	(X_1,\ldots,X_f)\mapsto \left(p^f\Ad(c_1c_2\cdots c_f)\circ\Ad(\Phi_0^f)(X_1),\ldots,  \right)	\]
But $\Ad(\Phi_0^f)$ and $\Ad(c_1c_2\cdots c_f)$ are commuting semi-simple operators, because $\Phi_0=\lambda(p^{-1/2})$ and the $c_i$ were chosen to centralize $\lambda$ so the $f$th iterate of $p\Ad(\underline{\Phi})$ is semi-simple.  Since we are working in characteristic $0$, this implies that $p\Ad(\underline{\Phi})$ is semi-simple as well.

Let us consider the kernel of $p\underline\Ad(\underline{\Phi})-1$.  The operators $\Ad(\Phi_0)$ and $\Ad(c_1c_2\cdots c_f)$ can be simultaneously diagonalized, because they are commuting semi-simple operators, so to compute the kernel of $p\underline\Ad(\underline{\Phi})-1)$, it suffices to compute the kernel of the restriction of $p\underline\Ad(\underline{\Phi})-1$ to each simultaneous eigenspace.  So suppose $\underline{X}:=(X_1,\ldots,X_f)\in \ker(p\underline\Ad(\underline{\Phi})-1)$, and suppose that the $X_i$ are eigenvectors for $\Ad(\Phi_0)$.  Then $p\Ad(c_i)\circ\Ad(\Phi_0)(X_{i+1})=X_i$ for all $i$, and iterating, we have
\[	p^{f}\Ad(c_ic_{i+1}\cdots c_{i-1})\circ\Ad(\Phi_0^f)(X_i))=X_i	\]
for all $i$.  The indices are taken modulo $f$.  But $c_ic_{i+1}\cdots c_{i-1}$ is a finite order operator, say of order $n$.  Iterating the application of $(p\Ad(\underline\Phi))^f$ $n$ times, we have
\[	p^{nf}\Ad(\Phi_0^{nf})(X_i) = X_i	\]
Therefore, $X_i$ lives in the weight $2$ eigenspace of $\Phi_0$ for all $i$.

We have now seen that $p\underline\Ad(\underline\Phi)-1$ is a semi-simple endomorphism of $\ad D_{k'}$ whose kernel is the direct sum of the weight $2$ eigenspaces of the factors, while the image of $\ad_{\underline{N}}$ includes the weight $2$ -eigenspaces.  Thus, 
\[	(p\underline\Ad(\underline\Phi)-1)+\ad_{\underline N}:\ad D_{k'}\oplus \ad D_{k'}\rightarrow\ad D_{k'}	\]
is surjective, and $\mathcal H^2$ vanishes at the point corresponding to $(\underline \Phi,\underline N,\underline\tau)$.

Now suppose that the $N_i=0$ for all $i$.  Then the fiber of $X_{\varphi,N}$ over $\underline N$ is $Z_G(\tau)^{\times f}$, so we need to find a point on every connected component of $Z_G(\tau)^{\times f}$ where $\H^2$ vanishes, i.e., to find some $\underline\Phi$ such that $p\underline\Ad(\underline\Phi)-1$ acts invertibly on $\ad D_{k'}$. 

Suppose that $\underline\Phi$ does not have this property, i.e., that there is some $\underline{N}'\neq 0$ such that $(p\underline\Ad(\underline\Phi)-1)(\underline N')=0$.  We will find some $\underline\Phi'$ on the same connected component of $Z_G(\tau)^{\times f}$ such that $p\underline\Ad(\underline\Phi')-1$ acts invertibly on $\ad D_{k'}$.

Note that $\underline N'$ is nilpotent, so by the argument above there is some $\underline{b}\in Z_G(\tau)$ such that $\underline{b}\cdot (\underline\Phi,\underline N')$ has $\underline{b}\cdot\underline{N}'=(N',\ldots,N')$.  Therefore, $\underline{b}\cdot(\underline\Phi,\underline N',\underline\tau)$ is a $(\varphi,N,\Gal_{L/K})$-module with $\underline{b}\cdot\underline{N}'=(N',\ldots,N')$ non-zero.  In other words, if $\lambda':\Gm\rightarrow Z_G(\tau)^\circ$ is a cocharacter associated to $N'$, then $\underline{b}\cdot\underline\Phi \in(\lambda'(p^{-1/2}))_i\left(Z_G(\tau)\cap Z_G(N)\right)^{\times f}$.

As before, there exists 
\[	\underline{c}=(c_1,\ldots,c_f)\in \left(Z_G(\tau)\cap Z_G(\lambda')\cap Z_G(N)\right)^{\times f}	\]
with $c_1\cdots c_f$ finite order such that $\underline{b}\cdot\underline\Phi$ is on the same connected component of $(\lambda'(p^{-1/2}))_i\left(Z_G(\tau)\cap Z_G(N)\right)^{\times f}$ as $\left(\lambda'(p^{-1/2})c_1,\ldots,\lambda'(p^{-1/2})c_f\right)$.  Therefore, $\underline{\Phi}$ is on the same connected component of $Z_G(\tau)^{\times f}$ as $\underline{b}^{-1}\cdot\left(\lambda'(p^{-1/2})c_1,\ldots,\lambda'(p^{-1/2})c_f\right)$.  Since $\Gm$ is connected, this implies that $\underline{\Phi}$ is on the same connected component of $Z_G(\tau)^{\times f}$ as $\underline{b}\cdot\left(\lambda'(t)c_1,\ldots,\lambda'(t)c_f\right)$, for all $t$.

Now $\lambda':\Gm\rightarrow Z_G(\tau)^\circ$ induces a grading of $\ad D_{k'}$ (by grading each factor).  Then for any $t_0\in k'$ such that $\lambda'(t_0)$ does not have $\zeta/p$ as an eigenvalue for any $p$th power root of unity $\zeta$, $p\underline\Ad\left(\lambda'(t_0),\ldots,\lambda'(t_0)\right)-1$ acts invertibly on $\ad D_{k'}$.  We claim that $p\underline\Ad\left(\lambda'(t_0)c_1,\ldots,\lambda'(t_0)c_f\right)-1$ acts invertibly on $\ad D_{k'}$.

Indeed, if $\Ad\left(\lambda'(t_0)c_1,\ldots,\lambda'(t_0)c_f\right)$ has eigenvalue $1/p$, then the $f$th iterate, which is $k'\otimes_{\Q_p}K_0$-linear and acts by
\[	(X_1,\ldots,X_f)\mapsto \left(\Ad(c_1\cdots c_{f}\lambda'(t_0)^f)(X_1),\ldots, \Ad(c_f\cdots c_{f-1}\lambda'(t_0)^f)(X_f)\right)	\]
has eigenvalue $1/p^f$.  Suppose that $c_1\cdots c_f$ has order $n$.  Then the $fn$th iterate acts by
\[	(X_1,\ldots,X_f)\mapsto \left(\Ad(\lambda'(t_0)^{fn})(X_1),\ldots, \Ad(\lambda'(t_0)^{fn})(X_f)\right)	\]
and has eigenvalue $1/p^{fn}$, contradicting our hypothesis on $t_0$.

To summarize, $\underline\Phi$ is on the same connected component of $Z_G(\tau)^{\times f}$ as $\underline\Phi'$, where $\underline\Phi'=\underline{b}^{-1}\cdot(\lambda'(t_0)c_i)_i$, and $p\underline\Ad(\lambda'(t_0)c_i)_i-1$ acts invertibly on $\ad D_{k'}$.  Therefore, $p\underline\Ad(\underline\Phi')-1$ acts invertibly on $\ad D_{k'}$, and we are done.
\end{proof}

The proof of Lemma~\ref{finrep} is due to Peter McNamara:
\begin{lemma}\label{finrep}
Let $H$ be a (possibly disconnected) reductive group over an algebraically closed field.  On each connected component of $H$, there is a point of finite order.
\end{lemma}
\begin{proof}
Choose a component $gH^\circ$.  To produce a finite order point on $gH^\circ$, it suffices to produce a finite order point on the component of $g$ in the center $Z_{Z_H(g)}$ of the centralizer $Z_H(g)$.  But $Z_{Z_H(g)}$ is commutative, so we have a decomposition $Z_{Z_H(g)}=M\times U$, where $M$ consists of semisimple elements of $Z_{Z_H(g)}$ and $U$ consists of unipotent elements of $Z_{Z_H(g)}$, by~\cite[Theorem 15.5]{humphreys-lag}.  Since $U$ is a unipotent group, it is connected, and so it suffices to produce a finite order point on each component of $M$.  Now the connected component of the identity $M^0\subset M$ is a torus, and we claim that the exact sequence 
\[	0\rightarrow M^0\rightarrow M\rightarrow M/M^0\rightarrow 0	\]
is split.  Since $M/M^0$ is abelian, we may assume it is cyclic.  Let $x\in M/M^0$ be a generator, and let $n$ be its order.  Choose any lift $\widetilde{x}\in M$ of $x$.  If $n\widetilde{x}\in M^0$ is the identity, we are done.  Otherwise, note that multiplication $n:M^0\rightarrow M^0$ is surjective; if $y\in M^0$ is in the preimage of $n\widetilde{x}$ under multiplication by $n$, then $y^{-1}\widetilde{x}\in M$ is a lift of $x$ such that $ny^{-1}\widetilde{x}$ is the identity, and we have our desired splitting.
\end{proof}

We are now in a position to deduce the first part of Theorem~\ref{x-phi-n-tau-sings}:
\begin{cor}
$X_{\varphi,N,\tau}$ is reduced and locally a complete intersection.  Each irreducible component has dimension $\dim \Res_{E\otimes L_0/E}G$.
\end{cor}
\begin{proof}
We again extend scalars on $X_{\varphi,N,\tau}$ from $E$ to $\overline{E}$.
We then choose an $\overline{E}$-point $x\in X_{\varphi,N,\tau}$ where $\mathcal H^2$ vanishes.  The tangent space at this point is given by $\ker\left(\ad D_{\overline{E}}\oplus \ad D_{\overline{E}}\rightarrow \ad D_{\overline{E}}\right)$ by Remark~\ref{def-thy-framed} (since we are working with a moduli space of framed objects), and since $\mathcal H^2$ vanishes, the tangent space has dimension $\dim_{\overline E}\ad D_{\overline{E}} = [L_0:\Q_p]\dim G$.  Since this implies that there is a Zariski open dense subspace of $X_{\varphi,N,\tau}$ with dimension $[L_0:\Q_p]\dim G$ and $X_{\varphi,N,\tau}$ is a scheme of finite type over a field, we see that $X_{\varphi,N,\tau}$ is equidimensional of dimension $[L_0:\Q_p]\dim G$.  Indeed, since the dimension of a scheme is insensitive to nilpotents, we may pass to the underlying reduced subscheme of $X_{\varphi,N,\tau}$ and compute the dimension of each irreducible component.  The dimension of a variety is equal to the transcendence degree of its function field, and since each irreducible component has a Zariski open dense subspace of dimension $[L_0:\Q_p]\dim G$, we are done.

Next, observe that $X_\tau$ is the disjoint union of smooth schemes of the form 
\[	\left(\Res_{E\otimes L_0/E}G\right)/\left(Z_{\Res_{E\otimes L_0/E}G}(\tau_0)\right)	\]
where $\tau_0:\Gal_{L/K}\rightarrow G(\overline E)$ is a representation of $\Gal_{L/K}$.

Over a diagonal point $\tau=(\tau_0,\ldots,\tau_0)\in X_\tau$, the fiber $X_{\varphi,N}$ is defined by the relations $N_i=p\Ad(\Phi_i)(N_{i+1})$, where $N_i=\Ad(\tau_0(g))(N_i)$ and $\tau_0(g)=\Ad(\Phi_i)(\tau_{0}(g))$. Thus, $\Phi_i$ and $N_i$ are required to live in $Z_G(\tau_0)$ and its Lie algebra, respectively.  Then the condition
\[	N_i=p\Ad(\Phi_i)(N_{i+1})	\]
gives us $\dim Z_G(\tau_0)$ equations, so the fiber $X_{\varphi,N}$ is cut out of the smooth $2[L_0:\Q_p]\dim Z_G(\tau_0)$-dimensional space $Z_G(\tau_0)^{\times f}\times \ad Z_G(\tau_0)^{\times f}$ by $[L_0:\Q_p]\dim Z_G(\tau_0)$ equations.

Now for any $\tau'\in X_\tau$ on the same component as $\tau$, there is some \'etale neighborhood $U$ of $\tau'$ such that the quotient map $\Res_{E\otimes L_0/E}G \rightarrow \Res_{E\otimes L_0/E}G/Z_{\Res_{E\otimes L_0/E}G}(\tau)$ admits a section $g\in \Res_{E\otimes L_0/E}G(U)$, since $Z_{\Res_{E\otimes L_0/E}G}(\tau)$ is smooth.  Therefore, the $U$-pullback $X_{\varphi,N,\tau}|_U$  of $X_{\varphi,N,\tau}$ is isomorphic to $U\times X_{\varphi,N}$ and has dimension $[L_0:\Q_p]\dim G$.  But $U\times X_{\varphi,N}$ is cut out of the smooth $[L_0:\Q_p](\dim G-\dim Z_G(\tau_0))+2[L_0:\Q_p]\dim Z_G(\tau_0)$-dimensional space $U\times Z_G(\tau_0)^{\times f}\times \ad Z_G(\tau_0)^{\times f}$ by $[L_0:\Q_p]\dim Z_G(\tau_0)$ equations, so it is locally a complete intersection.

Since being locally a complete intersection can be checked \'etale-locally, it follows that $X_{\varphi,N,\tau}$ is locally a complete intersection.  Furthermore, schemes which are local complete intersections are Cohen--Macaulay by~\cite[Theorem 21.3]{matsumura}.  Cohen--Macaulay schemes which are generically reduced are reduced everywhere, since they have no embedded points, by~\cite[Theorem 17.3]{matsumura}, so we are done.
\end{proof}

Thus far, we have considered moduli spaces of $(\varphi,N,\Gal_{L/K})$-modules, rather than moduli spaces of filtered $(\varphi,N,\Gal_{L/K})$-modules.  We now add a filtration to our set-up and complete the proof of Theorem~\ref{x-phi-n-tau-sings}.  

Fix a conjugacy class $[\mathbf{v}]$ of cocharacters $\mathbf{v}:\Gm\rightarrow (\Res_{E\otimes K/E}G)_{\overline{E}}$ with a representative defined over $E$.  Let $P_\mathbf{v}$ denote the parabolic $P_{\Res_{E\otimes K/E}G}(\mathbf{v})\subset \Res_{E\otimes K/E}G$ for some such representative $\mathbf{v}$.  Then $\Res_{E\otimes K/E}G/P_\mathbf{v}$ represents the moduli problem on $E$-algebras which is defined on $A$-points by
\[	A\mapsto \{\otimes\text{-filtrations of type }\mathbf{v}\text{ on }(\Res_{E\otimes K/E}G)_A\}	\]
Indeed, given an $A$-point of $\Res_{E\otimes K/E}G/P_\mathbf{v}$ where $A$ is an $E$-algebra, we obtain a family $\mathcal{P}\rightarrow \Spec (A\otimes K)$ of parabolic subgroups, such that for every geometric point $x$ of $\Spec (A\otimes K)$, $\mathcal{P}_x$ is conjugate to $P_{\Res_{E\otimes K/E}G}(\mathbf{v})$.  \'Etale-locally on $\Spec (A\otimes K)$, $\mathcal{P}=gP_{\mathbf{v}}g^{-1}$ for some $g\in G(A\otimes K)$ since the morphism $\Res_{E\otimes K/E}G\rightarrow\Res_{E\otimes K/E}G/P_\mathbf{v}$ is smooth.  Thus, \'etale-locally on $\Spec (A\otimes E)$, we get a $\otimes$-filtration on $(\Res_{E\otimes K/E}G)_A$, by taking the $\otimes$-filtration associated to the $\otimes$-grading induced by $g\lambda g^{-1}$.  Since $g$ is defined up to translation by an element of $P_{\mathbf{v}}(A)$ (since parabolic subgroups are their own normalizers) and such elements preserve the $\otimes$-filtration, we in fact get a $\otimes$-filtration on $(\Res_{E\otimes K/E}G)_A$ over all of $\Spec A$.

Thus, $X_{\varphi,N,\tau}\times \Res_{E\otimes K/E}G/P_\mathbf{v}$ is the moduli space of framed filtered $(\varphi,N,\tau)$-modules valued in $G$; it is locally a complete intersection because $X_{\varphi,N,\tau}$ is and $\Res_{E\otimes K/E}G/P_\mathbf{v}$ is smooth.

\section{Galois deformation rings}\label{section-def-rings}

Fix a continuous potentially semi-stable representation $\rho:\Gal_K\rightarrow G(E)$, with Galois type $\tau$ and $p$-adic Hodge $\mathbf{v}$, and assume that $\rho$ becomes semi-stable over a finite, Galois, totally ramified extension $L/K$.  We wish to study potentially semi-stable lifts $\tilde\rho:\Gal_K\rightarrow G(R)$, where $R$ is a $\Q_p$-finite artin local ring with residue field $E$.  More precisely, we consider the deformation functor $\Def_\rho^\square$ whose $R$-points are
\begin{equation*}
\begin{split}
\Def_\rho^{\square}(R):=\{\tilde\rho:\Gal_K\rightarrow G(R)| \tilde\rho\text{ is a lift of }\rho\}
\end{split}
\end{equation*}
Following~\cite{kisin}, we will show that $\Def_\rho^\square$ is pro-represented by a complete local noetherian $\Q_p$-algebra $R_\rho^\square$ which is reduced and equidimensional, by relating it to $\mathfrak{Mod}_{F,\varphi,N,\tau}$.  This will prove Theorem~\ref{galois-def-rings}.

We also define the deformation groupoid $\Def_\rho^{\tau,\mathbf{v}}$ whose $R$-points are
\begin{equation*}
\begin{split}
\Def_\rho^{\tau,\mathbf{v}}(R):=\{\widetilde\rho:\Gal_K\rightarrow G(R)| \widetilde\rho\otimes_RE\cong\rho\text{ and }\widetilde\rho\text{ is potentially semi-stable}\\ \text{with Galois type }\tau\text{ and }p\text{-adic Hodge type }\mathbf{v}\}
\end{split}
\end{equation*}
There is a natural morphism of groupoids $\Def_\rho^{\square,\tau,\mathbf{v}}\rightarrow\Def_\rho^{\tau,\mathbf{v}}$ given by ``forgetting the basis''.  There is also an associated functor $|\Def_\rho^{\tau,\mathbf{v}}|$ whose $R$-points are defined by
\[	|\Def_\rho^{\tau,\mathbf{v}}|(R):=\Def_\rho^{\tau,\mathbf{v}}(R)/\sim	\]

Then $\Def_\rho^{\square,\tau,\mathbf{v}}\rightarrow\Def_\rho^{\tau,\mathbf{v}}$ is formally smooth in the sense that for every square-zero thickening $R\rightarrow R/I$, 
given a deformation $\rho':\Gal_K\rightarrow G(R)$ of $\rho$ and a lift $\widetilde\rho:\Gal_K\rightarrow G(R/I)$ of $\rho$ such that $\rho'\otimes_RR/I\cong\widetilde\rho$, there is a lift $\widetilde\rho':\Gal_K\rightarrow G(R)$ of $\rho$ such that $\widetilde\rho'\cong\widetilde\rho$.  

Moreover, the fibers of the map $|\Def_\rho^{\square,\tau,\mathbf{v}}|(E[\varepsilon])\rightarrow|\Def_\rho^{\tau,\mathbf{v}}|(E[\varepsilon])$ are torsors under $\ad G/(\ad G)^{\Gal_K}$.  More precisely, $g\in\ad G$ (an element of $G(E[\varepsilon])$ which is the identity modulo $\varepsilon$) acts by conjugation on representations $\rho:\Gal_K\rightarrow G(E[\varepsilon])$, and $g\in\ad G$ acts trivially if and only if $g\in(\ad G)^{\Gal_K}$.  Thus,
\begin{equation}\label{tgt-space-pst}
\dim_E|\Def_\rho^{\square,\tau,\mathbf{v}}|(E[\varepsilon]) = \dim_E|\Def_\rho^{\tau,\mathbf{v}}|(E[\varepsilon]) + \dim_E\ad G - \dim_E(\ad G)^{\Gal_K}
\end{equation}

Fix a faithful representation $\sigma:G\rightarrow \GL_n$.  Then $\sigma\circ\rho$ is potentially semi-stable, with Galois type $\sigma\circ\tau$ and $p$-adic Hodge type the (geometric) conjugacy class of $\sigma\circ\mathbf{v}$.  More precisely, $\tau$ is a homomorphism $\Gal_{L/K}=I_{L/K}\rightarrow\Aut(\D_{\st}^L(\rho))$, and by $\sigma\circ\tau$ we really mean the homomorphism $I_{L/K}\rightarrow GL_n(E\otimes_{\Q_p}L)$ induced by pushing out $\D_{\st}^L(\rho)$ along $\sigma$.  To interpret $\sigma\circ\mathbf{v}$, we choose a representative cocharacter for the conjugacy class $\mathbf{v}$, compose with $\sigma$, and consider the corresponding conjugacy class of cocharacters $\Gm\rightarrow\Res_{E\otimes L/E}\GL_n$.  Then we define two deformation functors:
\begin{equation*}
\begin{split}
\Def_{\sigma\circ\rho}^\square(R):=\{\tilde\rho':\Gal_K\rightarrow\GL_n(R)| \tilde\rho'\text{ is a lift of }\sigma\circ\rho\}
\end{split}
\end{equation*}
\begin{equation*}
\begin{split}
\Def_{\sigma\circ\rho}^{\square,\tau,\mathbf{v}}(R):=\{\tilde\rho':\Gal_K\rightarrow \GL_n(R)| \tilde\rho'\text{ is a potentially semi-table lift of }\sigma\circ\rho \\ 
\text{ with Galois type }\sigma\circ\tau\text{ and }p\text{-adic Hodge type }\sigma\circ\mathbf{v}\}
\end{split}
\end{equation*}
The pro-representability of $\Def_{\sigma\circ\rho}^\square$ follows from Schlessinger's criterion; this is discussed in~\cite{mazur}, for example.  The pro-representability of $\Def_{\sigma\circ\rho}^{\square,\tau,\mathbf{v}}$ is a deeper fact, and the proof relies on integral $p$-adic Hodge theory.

Choose an $\mathscr{O}_E$-model for $\sigma\circ\rho$, and let $V_\F$ be the reduction modulo $\pi$, where $\pi$ is a uniformizer of $\mathscr{O}_E$.  Then there is a complete local noetherian $W(k_E)$-algebra $R_{V_\F}^{\square}$ which prorepresents framed deformations of $V_\F$ (on the category of artinian $W(k_E)$-algebras).  Here $k_E$ is the residue field of $\mathscr{O}_E$.  This again follows from Schlessinger's criterion.  

We can associate to $R_{V_\F}^{\square}$ its generic fiber, which will be a (non-quasi-compact) quasi-Stein rigid analytic space $X_{V_\F}^\square$ over $W(k_E)[1/p]$.  Then $\sigma\circ\rho$ corresponds to a point of $X_{V_\F}^\square$, and for every $x\in X_{V_\F}^\square$, the complete local ring at $x$ is the framed deformation ring of the corresponding characteristic $0$ Galois representation $V_x$ by~\cite[Proposition 2.3.5]{kisin-bt} or the discussion in~\cite[\textsection6]{conrad-char0}.

We define two closed subspaces of $X_{V_\F}^\square$.  The first, which we denote $X_{V_\F,G}^\square$, is the subspace consisting of Galois representations such that the image of $\Gal_K$ in $\GL_n(R)$ is contained in $G(R)\subset\GL_n(R)$.  The second, which we denote $X_{V_\F,\st}^{\square,\tau,\mathbf{v}}$, is the subspace consisting of potentially semi-stable representations with Galois type $\sigma\circ\tau$ and $p$-adic Hodge type $[\sigma\circ\mathbf{v}]$.  The existence of $X_{V_\F,\st}^{\square,\tau,\mathbf{v}}$ follows from~\cite[Theorem 2.7.6]{kisin}.

Then the point $x_0$ corresponding to $\sigma\circ\rho$ lies in $X_{V_\F,G}^\square\cap X_{V_\F,\st}^{\square,\tau,\mathbf{v}}$, by construction.  Furthermore, the complete local ring of $X_{V_\F,G}^\square\cap X_{V_\F,\st}^{\square,\tau,\mathbf{v}}$ at $x_0$ (which is noetherian) pro-represents $\Def_\rho^{\square,\tau,\mathbf{v}}$.  The reader may object that the conjugacy class of $\sigma\circ\mathbf{v}$ may ``glue together'' several different conjugacy classes of cocharacters $\Gm\rightarrow \Res_{E\otimes_{\Q_p}L/E}G$.  This is quite true.  However, $p$-adic Hodge types are locally constant, by the discussion in~\textsection\ref{tann-filt}, and so the entire connected component of $X_{V_\F,G}^\square\cap X_{V_\F,\st}^{\square,\tau,\mathbf{v}}$ containing $x_0$ has $p$-adic Hodge type $\mathbf{v}$.

Let $\Sp(A)\subset X_{V_\F,G}^\square\cap X_{V_\F,\st}^{\square,\tau,\mathbf{v}}$ be a connected affinoid subdomain containing $x_0$.  By Appendix~\ref{tann-fam-pst}, $\Spec A$ carries a $G$-valued family of filtered $(\varphi,N,\Gal_{L/K})$-modules, and therefore defines an $A$-valued point of the groupoid $\mathfrak{Mod}_{F,\varphi,N,\tau}$.  Identifying $A$ with the groupoid on $E$-algebras it represents, we view this as a morphism of groupoids $A\rightarrow \mathfrak{Mod}_{F,\varphi,N,\tau}$.

\begin{prop}\label{gal-rep-form-sm}
For every maximal ideal $\mathfrak{m}$ of $A$, corresponding to a Galois representation $\rho:\Gal_K\rightarrow G(E')$, the morphism of groupoids $\widehat A_\mathfrak{m}\rightarrow \mathfrak{Mod}_{F,\varphi,N,\tau}$ is formally smooth.
\end{prop}
\begin{proof}
The proof of~\cite[Proposition 3.3.1]{kisin} carries over verbatim here.
\end{proof}

\begin{cor}
For any $x\in X_{V_\F,G}^\square\cap X_{V_\F,\st}^{\square,\tau,\mathbf{v}}$, let $\widehat A_x$ be the completed local ring of $X_{V_\F,G}^\square\cap X_{V_\F,\st}^{\square,\tau,\mathbf{v}}$ at $x$.  Then there is an object $D_x\in\mathfrak{Mod}_{F,\varphi,N,\tau}(\widehat A_x)$.  If $U\subset \Spec(\widehat A_x)$ is the complement of the support of $\H^2(D_x)$, then $U$ is dense in $\Spec(\widehat A_x)$.
\end{cor}
\begin{proof}
After an \'etale extension $\widehat A_x\rightarrow A'$ corresponding to a finite extension of the residue field (to split $D_x$), $D_{A'}:=D_x\otimes_{\widehat A_x}A'$ is induced by a morphism 
\[	\Spec A'\rightarrow X_{\varphi,N,\tau}\times \Res_{E\otimes L_0/E}G/P_\mathbf{v}	\]
Furthermore, $U_{A'}$ is the complement of the support of $\H^2(D_{A'})$.  Then the proof of~\cite[Proposition 3.1.6]{kisin} carries over verbatim, and we see that the support of $\H^2(D_{A'})$ is nowhere dense in $\Spec A'$.  It follows that $U$ is dense in $\Spec(\widehat A_x)$.
\end{proof}
It follows that there is a formally smooth dense open subscheme of $\Spec A$ where $\H^2(D_A)=0$.  

\begin{cor}
For any $x\in X_{V_\F,G}^\square\cap X_{V_\F,\st}^{\square,\tau,\mathbf{v}}$, let $\widehat A_x$ be the completed local ring of $X_{V_\F,G}^\square\cap X_{V_\F,\st}^{\square,\tau,\mathbf{v}}$ at $x$.  Then $\widehat A_x$ is a complete intersection.
\end{cor}
\begin{proof}
The morphism $X_{\varphi,N,\tau}\times \Res_{E\otimes K/E}G/P_\mathbf{v}\rightarrow \mathfrak{Mod}_{F,\varphi,N,\tau}$ is smooth, so the fiber product with $\Spec A_x\rightarrow \mathfrak{Mod}_{F,\varphi,N,\tau}$ is an affine scheme $\Spec A'$ smooth over $A_x$ and formally smooth over $X_{\varphi,N,\tau}\times \Res_{E\otimes K/E}G/P_\mathbf{v}$.  It suffices to show that $A'$ is locally a complete intersection.

Let $y$ be a point of $\Spec A'$ and let ${A'}_y^\wedge$ be the complete local ring at $y$.  Then the morphism $\Spec {A'}_y^\wedge\rightarrow X_{\varphi,N,\tau}\times \Res_{E\otimes K/E}G/P_\mathbf{v}$ is induced by a local ring homomorphism $B\rightarrow {A'}_y^\wedge$, where $B$ is the completed stalk at a point of $X_{\varphi,N,\tau}\times \Res_{E\otimes K/E}G/P_\mathbf{v}$.  But $B\rightarrow {A'}_y^\wedge$ is formally smooth by Proposition~\ref{gal-rep-form-sm}, so ${A'}_y^\wedge$ is a formal power series ring over $B$.  Since $B$ is complete intersection, $\widehat{A'}_y$ is as well.
\end{proof}

Then as in~\cite[Theorem 3.3.4]{kisin}, we prove the following: 
\begin{prop}
$\Spec A$ is equi-dimensional of dimension 
\[	\dim_EG+\dim_E(\Res_{E\otimes K/E}G)/P_\mathbf{v}	\]
\end{prop}
\begin{proof}
By the discussion of~\cite{conrad-irred} before Lemma 2.2.3, the dimension of local rings of $\Spec A$ is constant on irreducible components.  Thus, it suffices to show that $\Spec A$ contains a Zariski dense subspace of dimension $\dim_EG+\dim_E(\Res_{E\otimes K/E}G)/P_\mathbf{v}$.

There is a formally smooth dense open subscheme $U\subset\Spec A$ where $\H^2(D_A)$ vanishes.  To compute the dimension of $\Spec A$, we choose a closed point $x\in U$, with residue field $E'$, corresponding to the maximal ideal $\mathfrak{m}\subset A$.  Let $\rho_x$ be the representation $\rho_x:\Gal_K\rightarrow G(E')$ corresponding to $x$, and let $D_x:=\D_{\st}^L(\rho_x)$.  Since $A$ is formally smooth at $x$, to compute the dimension of $A$, it suffices to compute the tangent space at $x$.  But by equation~\eqref{tgt-space-pst},
\[	\dim_{E'}|\Def_{\rho_x}^{\square,\tau,\mathbf{v}}|(E'[\varepsilon]) = \dim_{E'}|\Def_{\rho_x}^{\tau,\mathbf{v}}|(E'[\varepsilon]) + \dim_E G - \dim_{E'}(\ad \rho_x)^{\Gal_K}	\]
where $\ad \rho_x$ is the induced Galois representation $\ad\rho_x:\Gal_K\rightarrow \ad G$.  Now by Proposition~\ref{def-reps}, 
\[	\dim_{E'}|\Def_{\rho_x}^{\tau,\mathbf{v}}|(E'[\varepsilon]) = \dim_{E'}\Ext^1(D_x,D_x)	\]
where $\Ext^1$ means extensions in the category of filtered $(\varphi,N,\Gal_{L/K})$-modules.  Because $\H^2(D_x)=0$ by assumption, we can actually compute $\dim_{E'}\Ext^1(D_x,D_x)$ to be
\begin{eqnarray*}
\dim_{E'}\Ext^1(D_x,D_x) &=& \dim_{E'}\H_F^1(D_x) \\
&=& \dim_{E'}((\ad D_x)_K/\Fil^0(\ad D_x)_K) + \dim_{E'}H_F^0(D_x)
\end{eqnarray*}
This follows from Proposition~\ref{def-phi-n-tau-filt}, since $\H_F^2(D_x)=\H^2(D_x)=0$.  In addition, we have $\dim_{E'}H_F^0(D_x) = \dim_{E'}(\ad \rho_x)^{\Gal_K}$, since both spaces are the infinitesimal automorphisms of $\rho_x$, so in the end we find that
\begin{equation*}
\begin{split}
\dim_{E'}|\Def_{\rho_x}^{\square,\tau,\mathbf{v}}|(E'[\varepsilon]) &= \dim_E G + \dim_{E'}((\ad D_x)_K/\Fil^0(\ad D_x)_K) \\
&= \dim_EG+\dim_E(\Res_{E\otimes K/E}G)/P_\mathbf{v}
\end{split}
\end{equation*}
as desired, since $(\ad D_x)_K$ is the tangent space of the (smooth) group $\Res_{E\otimes K/E}G$ and $\Fil^0(\ad D_x)_K$ is the tangent space of the (smooth) group $P_\mathbf{v}$.
\end{proof}

Again as in~\cite[Theorem 3.3.8]{kisin}, the crystalline analogue follows by similar arguments:
\begin{prop}
Let $\rho$ be a potentially crystalline representation $\rho:\Gal_K\rightarrow G(E)$ with Galois type $\tau$ and $p$-adic Hodge type $\mathbf{v}$.  Then the deformation problem 
\begin{multline*}
\Def_{\rho,{\rm{cr}}}^{\square,\tau,\mathbf{v}}(R):=\{\tilde\rho:\Gal_K\rightarrow G(R)| \tilde\rho\text{ is a potentially crystalline lift of }\rho\\ 
\text{ with Galois type }\tau\text{ and }p\text{-adic Hodge type }\mathbf{v}\}
\end{multline*}
is pro-representable by a complete local noetherian ring $R_{\rho,{\rm{cr}}}^{\square,\tau,\mathbf{v}}$ which is formally smooth of dimension 
\[	\dim_EG+\dim_E(\Res_{E\otimes K/E}G)/P_\mathbf{v}	\]
\end{prop}
In fact, the arguments are easier because $X_{\varphi,\tau}$ is actually smooth of dimension $\dim \Res_{E\otimes L_0/E}G$, not merely generically smooth.

\section{Explicit calculations}\label{section-calcs}

We wish to study the irreducible components of $X_{\varphi,N,\tau}$.  For simplicity, we restrict ourselves to the case when $L=K$ and $\tau$ is trivial.  In addition, suppose temporarily that $K_0=\Q_p$. 

Let $N_0\in\mathcal{N}$ be a non-zero nilpotent element of $\mathfrak{g}$, and let $\mathcal{O}_{N_0}\subset\mathcal{N}$ be its $G$-orbit, which is locally closed in $\mathcal{N}$.  Then if we consider the fiber square
\[\begin{CD}
X_{\varphi,N}|_{\mathcal{O}_{N_0}}@>>>X_{\varphi,N}	\\
@VVV					@VVV	\\
\mathcal{O}_{N_0}@>>>\mathcal{N}
\end{CD}	\]
the left vertical arrow is smooth.  Our technique for studying the irreducible components of $X_{\varphi,N,\tau}$ relies on studying the closure of $X_{\varphi,N}|_{\mathcal{O}_{N_0}}$ inside $\mathcal{N}\times G$.  To do this, we will use Springer resolutions of closures of nilpotent orbits.

Let $G$ be a connected reductive group, and let $N_0\in\mathfrak{g}$ be nilpotent and non-zero.  Then we can find a cocharacter $\lambda:\mathbf{G}_m\rightarrow G$ associated to $N_0$, and $\lambda$ defines a grading on $\mathfrak{g}$.  Associated cocharacters are not unique, but they are defined up to conjugacy by $Z_G(N)^{\circ}\subset P:=P_G(\lambda)$; it follows that the associated filtration on $\mathfrak{g}$ depends only on $N_0$.  

The Lie algebra $\mathfrak{p}$ of $P$ is naturally identified with $\mathfrak{g}_{\geq 0}$, and carries a filtration, which is preserved by the conjugation action of $P$.  Given $P':=gPg^{-1}$, the Lie algebra $\mathfrak{p}'$ carries the conjugate filtration.

We will be interested in $\overline{G\cdot N_0}$, the closure of the orbit of $N_0$ under the adjoint action of $G$ on $\mathfrak{g}$.  In general, this will be singular.  However, we have
\begin{prop}[{\cite[8.3.1]{weyman}}]
There is a natural morphism $G\times^P\mathfrak{g}_{\geq 2}\rightarrow \overline{G\cdot N_0}$ given by $(g,N)\mapsto \Ad(g)(N)$, and this is a resolution of singularities.
\end{prop}

We rewrite the quotient $G\times^P\mathfrak{g}_{\geq 2}$ in a more convenient form.  Consider the morphism 
\begin{eqnarray*}
G\times \mathfrak{g}_{\geq 2} &\rightarrow& G/P\times \mathfrak{g}	\\
(g,X)&\mapsto& (gP, \Ad(g)(X))
\end{eqnarray*}
It is $P$-equivariant, so descends to a morphism $G\times^P\mathfrak{g}_{\geq 2}\rightarrow G/P\times \mathfrak{g}$.  The image is $\{(gP,X)| X\in \Ad(g)(\mathfrak{g}_{\geq 2})\}$.  In other words, the image parametrizes pairs $(P',X)$, where $P'$ is a parabolic conjugate to $P$, and $X\in \mathfrak{p}_{\geq 2}'$.

Moreover, if $(g,X)$ and $(g',X')$ have the same image in $G/P\times\mathfrak{g}$, then there exists $p\in P$ such that $g'=gp$, and 
$$\Ad(g)(X)=\Ad(g')(X')=\Ad(g)\Ad(p)(X')$$
This implies that $X=\Ad(p)(X')$, so $(g,X)\sim (g',X')$.  In other words, the map $G\times^P\mathfrak{g}_{\geq 2}\rightarrow G/P\times \mathfrak{g}$ is an isomorphism onto its image.

\subsection{$\GL_2$}

We study the geometric structure of $X_{\varphi,N}$ more closely when our group $G$ is $\GL_2$ and $\tau$ is trivial.

Fix an unramified extension $K_0$ over $\Q_p$ of degree $f$.  After extending scalars on $X_{\varphi,N}$ from $E$ to $\overline{E}=\overline{\Q}_p$, we are considering the subscheme of $\GL_2^{\times f}\times \mathfrak{gl}_2^{\times f}$ of $f$-tuples $\underline{\Phi}:=(\Phi_1,\ldots ,\Phi_f)$ and $\underline{N}:=(N_1,\ldots ,N_f)$ satisfying 
$$N_i=p\Ad(\Phi_i)(N_{i+1})$$
for all $i$ (here the indices are taken modulo $f$).

There are two irreducible components, $X_{\reg}$ and $X_0$, corresponding to the regular nilpotent orbit in $\mathfrak{gl}_2$ and the orbit $\underline{N}=(0,\ldots,0)$, respectively.  Their intersection $X_{\reg,0}$ is the subscheme of $\GL_2^{\times f}$ of $\underline{\Phi}$ such that $\det(1-p\underline\Ad\underline{\Phi})=0$, where we consider $1-p\underline\Ad\underline{\Phi}$ as an operator on the $\overline{\Q}_p$-vector space $\mathfrak{gl}_2^{\times f}$ acting via
\[	\underline\Ad\underline{\Phi}(\underline{X})=(\Ad(\Phi_1)(X_2),\ldots,\Ad(\Phi_f)(X_0))	\]

More precisely, Corollary~\ref{fam-cent} implies that the locus in $X_{\varphi,N}$ where $\underline{N}\neq 0$ is a smooth open subscheme of dimension $\dim\GL_2\cdot f$ (since it is a $Z_G(\underline N_{\reg})$-torsor), and the locus where $\underline{N}=0$ is a smooth closed subscheme of dimension $f\cdot\dim\GL_2$.  We let $X_{\reg}$ denote the closure of the former inside $X_{\varphi,N}$ and we let $X_0$ denote the latter.  We will show that $X_{\reg}$ and $X_0$ are smooth irreducible components of $X_{\varphi,N}$, and their intersection $X_{\reg,0}$ is smooth as well, and characterized as the subscheme of $\GL_2^{\times f}$ such that $\det(1-p\underline\Ad\underline\Phi)=0$.

We do the last part first.



\begin{prop}
$X_{\reg,0}\subset \GL_2^{\times f}|_{\underline N=0}$ is defined scheme-theoretically by the equation $\det(1-p\underline\Ad\underline\Phi)=0$.
\end{prop}
\begin{proof}
If $(\underline\Phi_0,\underline N_0)$ corresponds to a geometric point of the open subscheme $X_{0}|_{\underline N\neq 0}\subset X_{\reg}$, then $\underline N_0$ is an element of the kernel of $1-p\underline\Ad\underline\Phi_0$, so $\det(1-p\underline\Ad\underline\Phi_0)=0$.  Since $X_{0}|_{\underline N\neq 0}$ is smooth (and in particular reduced), the equation $\det(1-p\underline\Ad\underline\Phi)$ vanishes on $X_{\reg}$.  Thus $X_{\reg,0}$ is contained in the subscheme defined by $\det(1-p\underline\Ad\underline\Phi)=0$.

Conversely, suppose that $\underline\Phi_0$ corresponds to a geometric point of $\GL_2^{\times f}=X_0$ with $\det(1-p\underline\Ad\underline\Phi_0)$.  Then there is some non-zero $N\in \mathfrak{gl}_2^{\times f}$ such that $(1-p\underline\Ad\underline\Phi_0)(N)=0$.  We define a morphism $\A^1\rightarrow X_{\varphi,N}$ via $t\mapsto (\underline\Phi_0,tN)$.   For $t\neq 0$, this morphism lands in the $\underline N\neq 0$ locus of $X_{\varphi,N}$.  Therefore, $(\underline\Phi_0,0)\in X_{\reg}$.

Thus, $X_{\reg,0}$ is a closed subscheme of $\{\det(1-p\underline\Ad\underline\Phi)=0\}\subset X_0$ with the same geometric points.  But by the next proposition, $\{\det(1-p\underline\Ad\underline\Phi)=0\}$ is smooth, and in particular, reduced, so it is equal to $X_{\reg,0}$.
\end{proof}

\begin{prop}
The subscheme of $\GL_2^{\times f}$ defined by $\det(1-p\underline\Ad\underline\Phi)=0$ is smooth.
\end{prop}
\begin{proof}
In order to study the locus in $\GL_2^{\times f}$ where $\det(1-p\underline\Ad\underline\Phi)=0$, we will compute the characteristic polynomial of $\underline\Phi\in\GL_2^{\times f}$ acting on $\mathfrak{gl}_2^{\times f}$.

For $\underline{X}:=(X_1,\ldots,X_f)\in \mathfrak{gl}_2^{\times f}$, $\underline\Ad\underline{\Phi}$ acts by 
$$\underline\Ad\underline{\Phi}(\underline{X})=(\Ad(\Phi_1)(X_2),\ldots,\Ad(\Phi_f)(X_1))$$
As a matrix, this is
$$\begin{pmatrix}
0 & \Ad(\Phi_1) & 0 & \cdots & 0	\\
0 & 0 & \Ad(\Phi_2) & & 0		\\
\vdots &  & \ddots & \ddots & \vdots	\\
0 & 0 & \cdots & 0 & \Ad(\Phi_{f-1})	\\
\Ad(\Phi_f) & 0 & \cdots & 0 & 0
\end{pmatrix}$$
Here the ``entries'' are actually $4\times 4$ matrices, and we view $\lambda$ and $\Ad(\Phi_i)$ as operators on $\mathfrak{gl}_2$.
Thus, to compute the characteristic polynomial of $p\underline\Ad\underline{\Phi}$, we need to compute the determinant of
$$\begin{pmatrix}
\lambda & -p\Ad(\Phi_1) & 0 & \cdots & 0	\\
0 & \lambda & -p\Ad(\Phi_2) & & 0		\\
\vdots &  & \ddots & \ddots & \vdots	\\
0 & 0 & \cdots & \lambda & -p\Ad(\Phi_{f-1})	\\
-p\Ad(\Phi_f) & 0 & \cdots & 0 & \lambda
\end{pmatrix}$$

By row reduction, this is the same as the determinant of
$$\begin{pmatrix}
\lambda & -p\Ad(\Phi_1) & 0 & \cdots & 0	\\
0 & \lambda & -p\Ad(\Phi_2) & & 0		\\
\vdots &  & \ddots & \ddots & \vdots	\\
0 & 0 & \cdots & \lambda & -p\Ad(\Phi_{f-1})	\\
0 & 0 & \cdots & 0 & \lambda-p^f\Ad(\Phi_f\cdot\Phi_1\cdots\Phi_{f-1})/\lambda^{f-1}
\end{pmatrix}$$
which is $\det(\lambda^f-p^f\Ad(\Phi_f\cdot\Phi_1\cdots\Phi_{f-1}))$.

We are interested in the subscheme of $\GL_2^{\times f}$ where $\det(1-p^f\Ad(\Phi_1\cdots\Phi_{f}))=0$.  Letting $\Nm\underline{\Phi}$ denote the product $\Phi_1\cdots\Phi_f$, the equation of this subscheme can be computed to be
$$p^f(\Tr\Nm\underline{\Phi})=(p^f+1)^2\det\Nm\underline{\Phi}$$
This follows from a brute force computation that the characteristic polynomial of the adjoint action of $\Phi\in\GL_2$ on $\mathfrak{gl}_2$is
$$\lambda^4-\frac{(\Tr\Phi)^2}{\det\Phi}\lambda^3 + 2(\frac{(\Tr\Phi)^2}{\det\Phi}-1)\lambda^2-\frac{(\Tr\Phi)^2}{\det\Phi}\lambda + 1$$
Thus, we are interested in the zero-locus of
\begin{equation*}
\begin{split}
1-p^f\frac{(\Tr\Nm\underline\Phi)^2}{\det\Nm\underline\Phi}+2p^{2f}&\left(\frac{(\Tr\Nm\underline\Phi)^2}{\det\Nm\underline\Phi}-1\right)-p^{3f}\frac{(\Tr\Nm\underline\Phi)^2}{\det\Nm\underline\Phi}+p^{4f} \\
&= (1-p^{2f})^2-p^f\frac{(\Tr\Nm\underline\Phi)^2}{\det\Nm\underline\Phi}(1-p^f)^2	\\
&= (1-p^f)^2\left[(1+p^f)^2-p^f\frac{(\Tr\Nm\underline\Phi)^2}{\det\Nm\underline\Phi}\right]
\end{split}
\end{equation*}
But then a simple computation shows that the equation
$$p^f(\Tr\Phi)^2=(p^f+1)^2\det\Phi$$
defines a smooth subscheme of $\GL_2$.  For if $\Phi=(\begin{smallmatrix}a&b\\c&d\end{smallmatrix})$, then the Jacobian of this equation is
$$\begin{pmatrix}2p^fa+2p^fd-(p^f+1)^2d\\ (p^f+1)^2c\\ (p^f+1)^2b \\ 2p^fa+2p^fd-(p^f+1)^2a\end{pmatrix}$$
Vanishing would force $a=d$ and $b=c=0$, which in turn would force $4p^f=(p^f+1)^2$, implying $(p^f-1)^2=0$, implying $p^f=1$, which is impossible.  Since the multiplication map $\GL_2^{\times f}\rightarrow \GL_2$ is smooth, this shows that $X_{\reg,0}$ is smooth.
\end{proof}

Next we claim that $X_{\reg}$ is smooth.  This can be done via a simple tangent space calculation: we know that $X_{\varphi,N}$ is equidimensional of dimension $4f$, and we know that $X_{\reg}$ contains an irreducible dense open smooth piece of dimension $4f$, by definition, so it is enough to show that the tangent space at every point of $X_{\reg}$ has dimension $4f$.

\begin{lemma}
Fix $\underline\Phi\in\GL_2^{\times f}(k)$ for some extension $k/\overline{E}$.  Then the space of elements $\underline{N}\in\mathfrak{gl}_2^{\times f}(k)$ such that $\underline N=p\underline\Ad\underline\Phi(N)$ is a $k$-vector space of dimension at most $1$.
\end{lemma}
\begin{proof}
The space of such $\underline{N}$ is certainly a $k$-vector space (under the diagonal action of $k$ on $\mathcal{N}(k) = \mathfrak{gl}_2(k)^{\times f}$), so if there are no such $\underline N$, we are done.  Suppose there is some $\underline N=(N_1,\ldots,N_f)$ such that $\underline N=p\underline\Ad\underline\Phi(N)$.  Then $N_i = p\Ad(\Phi_i)(N_{i+1})$, so the $N_i$ are determined by $N_1$.  We may assume by conjugation that $N_1=\left(\begin{smallmatrix}0&1\\0&0\end{smallmatrix}\right)$.  We also have $N_1=p^f\Ad(\Nm\underline\Phi)(N_1)$, so if $\lambda:\Gm\rightarrow \GL_2$ is the cocharacter $t\mapsto \left(\begin{smallmatrix}t&0\\0&t^{-1}\end{smallmatrix}\right)$ associated to $N_1$, then $\Nm\underline\Phi$ is of the form $\left(\begin{smallmatrix}p^{-f/2}&0\\0&p^{f/2}\end{smallmatrix}\right)\left(\begin{smallmatrix}a&0\\0&a\end{smallmatrix}\right)\left(\begin{smallmatrix}1&b\\0&1\end{smallmatrix}\right)$.  It suffices to show that the space of $N\in \mathfrak{gl}_2(k)$ such that $N=p^f\Ad(\Nm\underline\Phi)(N)$ is $1$-dimensional.

Now $\mathfrak{gl}_2(k)$ is graded by the action of $\lambda$, with weight spaces of weights $-2$, $0$, and $2$, generated by $\left(\begin{smallmatrix}0&0\\1&0\end{smallmatrix}\right)$, $\{\left(\begin{smallmatrix}1&0\\0&1\end{smallmatrix}\right),\left(\begin{smallmatrix}1&0\\0&-1\end{smallmatrix}\right)\}$, and $\left(\begin{smallmatrix}0&1\\0&0\end{smallmatrix}\right)$, respectively.  Conjugation by $\left(\begin{smallmatrix}1&b\\0&1\end{smallmatrix}\right)$ acts trivially on $\mathfrak{gl}_2/\Fil^{\geq 0}\mathfrak{gl}_2$ and $\Fil^{\geq 0}\mathfrak{gl}_2/\Fil^{\geq 2}\mathfrak{gl}_2$, and conjugation by $\left(\begin{smallmatrix}p^{-f/2}&0\\0&p^{f/2}\end{smallmatrix}\right)$ acts by multiplication by $p^{f}$ and $1$ on these spaces, respectively.  Therefore, if $N_1'\in \mathfrak{gl}_2(\overline{E})$ satisfies $N_1'=p^f\Ad(\Nm\underline\Phi)(N_1')$, the image of $N_1'$ is $0$ in each of these quotients.  Therefore, $N_1'$ is a multiple of $\left(\begin{smallmatrix}0&1\\0&0\end{smallmatrix}\right)$, as desired.
\end{proof}


Away from $X_{\reg,0}$, $X_{\reg}$ is smooth.  This can be seen by considering the morphism $X_{\varphi,N}\rightarrow X_N$ restricted to the regular nilpotent orbit $U$ of $\GL_2^{\times f}$.  For any point $\underline N\in U$, there is an \'etale neighborhood $V$ and a section $s:V\rightarrow \GL_2^{\times f}$ such that the nilpotent matrix over $V$ is of the form $\Ad(s(V))(\underline N)$, by Lemma~\ref{fam-section}.  If $\underline{N}=(N_0,\ldots,N_0)$, this implies that 
\[	X_{\varphi,N}|_V=s(V)\left((\Phi_0,\ldots,\Phi_0)Z_{\GL_2}(N_0)^{\times f}\right)\varphi(s(V))^{-1}	\]
where $N_0 = p\Phi_0N_0\Phi_0^{-1}$.  But every $\underline{N}\in U$ is $\GL_2^{\times f}$-conjugate to $(N_0,\ldots,N_0)$ for some regular nilpotent $N_0\in\mathfrak{gl}_2$.  Thus, we have an \'etale-local description of $X_{\varphi,N}|_U$, showing it is smooth.

On the other hand, at a geometric point of $X_{\reg}$ corresponding to $(\Phi_0,0)$, the tangent space is the space of pairs $(\underline\Phi_0+\varepsilon\underline\Phi_1,\varepsilon \underline N_1)$ with $\underline\Phi_0+\underline\varepsilon\Phi_1\in X_{\reg,0}$ and $\underline N_1$ satisfying $\underline N_1=p\underline \Phi_0\cdot\underline N_1$.  For we have seen that the equation $\det(1-p\Ad\underline\Phi)$ vanishes on $X_{\reg}$, so $\det(1-p\Ad(\underline\Phi_0+\varepsilon\underline\Phi_1))=0$, so $\underline\Phi_0+\varepsilon\underline\Phi_1$ is an $\overline{E}[\varepsilon]$-point of $X_{\reg,0}$.

Since $X_{\reg,0}$ is a smooth divisor of $\GL_2^{\times f}$ and the space of $\underline N_1$ compatible with $\underline \Phi_0$ is $1$-dimensional, this has the correct dimension.

In short, we have shown the following:
\begin{thm}
The space $X_{\varphi,N}$ is the union of two smooth schemes of dimension $4f$, whose intersection is smooth of dimension $4f-1$.
\end{thm}

\subsection{Regular nilpotent orbits in $\GL_n$}

Let $G=\GL_n$, and assume for the sake of simplicity that $K_0=\Q_p$.  The regular nilpotent orbit $\mathcal{O}_{\reg}$ in $\mathfrak{g}$ is the orbit of $N_{\reg}$, i.e., the nilpotent element with all ones on the superdiagonal.  Let $\lambda_{\reg}:\G_m\rightarrow G$ be the cocharacter $\diag(t^{n-1},t^{n-3},\ldots,t^{1-n})$; $\lambda_{\reg}$ is associated to $N_{\reg}$.  The conjugation action of $\lambda_{\reg}$ induces a grading on $\mathfrak{g}$, with $N_{\reg}$ in weight 2; $\mathfrak{g}\cong\oplus_{i=1-n}^{n-1} \mathfrak{g}_{2i}$, and the graded pieces are on the diagonals.  The parabolic $P_{\reg}:=P_G(\lambda_{\reg})$ is the standard upper triangular Borel.

We wish to study the closure $X_{\reg}$ of $X_{\varphi,N}|_{\mathcal{O}_{\reg}}$ inside $X_{\varphi,N}$.  To do this, we first extend scalars from $E$ to $\overline{E}$, and we define an auxiliary moduli problem $\widetilde X_{\reg}$.  The resolution $G\times^{P_{\reg}}\mathfrak{g}_{\geq2}$ of the closure $\overline{\mathcal{O}}_{\reg}$ of $\mathcal{O}_{\reg}$ carries a universal parabolic $\mathcal{P}$, along with the filtered Lie algebra $\mathfrak{P}\supset\mathfrak{P}_{\geq 2}\supset\cdots\supset\mathfrak{P}_{2(n-1)}$ of $\mathcal{P}$.  More precisely, $\mathcal{P}$ is a parabolic subgroup scheme $\mathcal{P}\subset G\times G/P_{\reg}$, such that for every parabolic subgroup scheme $P\rightarrow S$ with $P_{\overline s}$ conjugate to $P_{\reg}$, there is a unique morphism $f:S\rightarrow G/P_{\reg}$ such that $P\cong f^\ast \mathcal{P}$.  Then we define
\[	\widetilde X_{\reg}(A):=\{(\Phi,N)\in (\mathcal{P}\times_{G/P_{\reg}}\mathfrak{P}_{\geq 2})(A)\mid (1-p\Ad(\Phi))|_{\mathfrak{p}_{\geq2}/\mathfrak{p}_{\geq4}}=0, (1-p\Ad(\Phi))(N)=0\}	\]
In other words, $\Phi$ is an $A$-point of a family of parabolics $P$ and $N$ is an $A$-point of the Lie algebra ${\mathfrak{p}}$ of $P$, and we impose certain linear algebraic conditions on $\Phi$ and $N$.  There is a natural morphism $\widetilde X_{\reg}\rightarrow X_{\varphi,N}$ given by forgetting the parabolic, as well as a natural morphism $\widetilde X_{\reg}\rightarrow \mathcal{N}$ given by forgetting both the parabolic and $\Phi$.

\begin{prop}
$\widetilde X_{\reg}$ is smooth, as is the fiber $\widetilde X_{\reg}|_{N=0}$ over $N=0$.
\end{prop}
\begin{proof}
We use the functorial criterion for smoothness.  Let $A$ be an $\overline E$-algebra, let $I\subset A$ be an ideal such that $I^2=0$, and let $(\Phi_0,N_0,P_0)$ be an $A/I$-point of $\widetilde X_{\reg}$.  We wish to lift $(\Phi_0, N_0,P_0)$ to an $A$-point of $\widetilde X_{\reg}$.

First of all, $G/P_{\reg}$ is smooth, so we can lift $P_0$ to an $A$-point $P$ of $(G/P_{\reg})$.  We claim that the space of $\Phi$ in $P$ such that $1-p\Ad(\Phi)$ kills ${\mathfrak{p}}_{\geq 2}/\mathfrak{p}_{\geq 4}$ is smooth over $\Spec A$.  To see this, we can work locally on $\Spec A$.  Since the quotient $G\rightarrow G/P_{\reg}$ admits sections Zariski-locally, we may assume that there is some $g_P\in G(A)$ such that $P=g_PP_{\reg}g_P^{-1}$.  Therefore, we may assume that $P=P_{\reg}$.  But the space of $\Phi$ such that $(1-p\Ad(\Phi))|_{\mathfrak{p}_{\geq2}/\mathfrak{p}_{\geq4}}=0$ is a torsor under the subgroup $Z_G(\mathfrak{p}_{\geq2}/\mathfrak{p}_{\geq4})\subset P_{\reg}$
which acts trivially on $\mathfrak{p}_{\geq2}/\mathfrak{p}_{\geq4}$, so it is smooth.

Thus, we can lift ${\Phi}_0$ to an $A$-point $\Phi$ of $P$ such that $1-p\Ad(\Phi)$ kills $\mathfrak{p}_{\geq 2}/\mathfrak{p}_{4}$, so $\widetilde X_{\reg}|_{N=0}$ is smooth.  It remains to lift $N_0$ to an $A$-point of the kernel of $1-p\Ad(\Phi)$.  But the kernel of $1-p\Ad(\Phi)$ on $\mathfrak{p}_{\geq 2}$ is a rank $n-1$ vector bundle on $\Spec A$, so we can lift $N_0$.
\end{proof}

\begin{prop}
The morphism $\widetilde X_{\reg}\rightarrow X_{\varphi,N}$ is an isomorphism onto $X_{\reg}$.
\end{prop}
\begin{proof}
We first show that $\widetilde X_{\reg}\rightarrow X_{\varphi,N}$ is a monomorphism.  So suppose $(\Phi,N)$ is an $A$-valued pair such that $N=p\Ad(\Phi)(N)$; we need to show that there is at most one $P$ such that $\Phi\in P$, $N\in \mathfrak{p}_{\geq 2}$, and $1-p\Ad(\Phi)$ kills $\mathfrak{p}_{\geq 2}/\mathfrak{p}_{\geq4}$.  For this, it suffices to show that $\Phi$ and $N$ determine $\mathfrak{p}$ as a Lie subalgebra of $\mathfrak{gl}_n$, together with its filtration.  Further, it suffices to check this on geometric points of $A$, so we may assume that $A$ is an algebraically closed field of characteristic $0$.  

So suppose there is some such $P$, and let $\lambda_P:\G_m\rightarrow G$ be a cocharacter such that $P=P_G(\lambda_P)$.  Then we can uniquely write $\Phi=zu$ with $z\in Z_G(\lambda_P)$ and $u\in U_G(\lambda_P)$.  Now $U_G(\lambda_P)$ acts (via the adjoint action) as the identity on each quotient $\mathfrak{g}_{\geq i}/\mathfrak{g}_{\geq i+1}$, and $\Ad(\lambda_P(t_0))$ acts on $\mathfrak{g}_{\geq 2}/\mathfrak{g}_{\geq4}$ by multiplication by $1/p$ if and only if $t_0=p^{-1/2}$.  Thus, $1-p\Ad(\Phi)$ kills $\mathfrak{g}_{\geq 2}/\mathfrak{g}_{4}$ if and only if $z=\lambda_P(p^{-1/2})z'$, where $z'\in Z_G(\lambda_P)$ acts as the identity on $\mathfrak{g}_{\geq 2}/\mathfrak{g}_{\geq4}$.  Since $Z_G(\lambda_P)$ is conjugate to the standard diagonal torus in $\GL_n$ (since $\lambda_P$ is conjugate to $\lambda_{\reg}$), we see easily that $1-p\Ad(\Phi)$ kills $\mathfrak{g}_{\geq 2}/\mathfrak{g}_{\geq4}$ if and only if $z'\in Z_G$.

But now we see that $\Ad(\Phi)$ acts on $\mathfrak{g}_{\geq 2i}/\mathfrak{g}_{\geq 2i+2}$ as multiplication by $p^{-i}$.  It follows that $\ker(1-p^{-i}\Ad(\Phi))$ is a subspace of $\mathfrak{g}_{\geq 2i}$, linearly disjoint from $\mathfrak{g}_{\geq 2i+2}$.  Thus, if we can show that $\mathfrak{g}_{2i}=\ker(1-p^{-i}\Ad(\Phi))+\mathfrak{g}_{2i+2}$, we will be done.  But this follows because $1-p^{-i}\Ad(\Phi):\mathfrak{g}_{\geq 2i}\rightarrow \mathfrak{g}_{\geq 2i}$ descends to the zero map on $\mathfrak{g}_{\geq 2i}/\mathfrak{g}_{\geq 2i+2}$.  This implies that the image of $1-p^{-i}\Ad(\Phi)$ lies in $\mathfrak{g}_{\geq 2i+2}$, so $\ker(1-p^{-i}\Ad(\Phi))$ has dimension at least $\dim \mathfrak{g}_{\geq 2i}/\mathfrak{g}_{\geq 2i+2}$; since it is linearly disjoint from $\mathfrak{g}_{\geq -2i+2}$, it has dimension exactly $\dim \mathfrak{g}_{\geq 2i}-\dim\mathfrak{g}_{\geq 2i+2}$, and $\mathfrak{g}_{2i}=\ker(1-p^{-i}\Ad(\Phi))+\mathfrak{g}_{2i+2}$.

Next, we show that the image of $\widetilde X_{\reg}$ is contained in $X_{\reg}$. 
Let $(\Phi,N)$ correspond to a geometric point of $X_{\varphi,N}$ in the image of $\widetilde X_{\reg}$.  Then we may assume that $\Phi\in P_{\reg}$ and $N\in \mathfrak{p}_{\geq 2}$, and $1-p\Ad(\Phi)$ kills $\mathfrak{p}_{\geq 2}/\mathfrak{p}_{\geq4}$.  We have seen that $1-p\Ad(\Phi)$ kills an $n-1$-dimensional subspace of $\mathfrak{p}_{\geq 2}$ containing $N$ and intersecting $\mathfrak{p}_{\geq4}$ trivially.  But any such subspace contains a regular nilpotent element $N'$, and $(\Phi,N+t(N'-N))$ defines an $\mathbf{A}^1$-point of $X_{\varphi,N}$ connecting $(\Phi,N)$ to $X_{\varphi,N}|_{\mathcal{O}_{\reg}}$.

Finally, we show that $\widetilde X_{\reg}\rightarrow X_{\varphi,N}$ is proper. Let $R$ be a discrete valuation ring over $\overline E$, and let $f:\Spec R\rightarrow X_{\reg}$ be a morphism such that the generic point $\eta$ of $\Spec R$ maps to $X_{\varphi,N}|_{\mathcal{O}_{\reg}}$.  This induces a family $(\Phi,N)$ of $(\varphi,N)$-modules over $R$.  Forgetting $\Phi$ yields a morphism $\eta\rightarrow \mathcal{O}_{\reg}$, and therefore an $R$-point $P$ of $G/P_{\reg}$, since $G/P_{\reg}$ is proper.  Since $N_\eta\in (\mathfrak{p}_{\eta})_{\geq 2}$ and this is a closed condition, we have $N\in\mathfrak{p}_{\geq 2}$.  Further, since $1-p\Ad(\Phi_\eta)$ kills $(\mathfrak{p}_{\eta})_{\geq 2}/(\mathfrak{p}_{\eta})_{\geq4}$, $1-p\Ad(\Phi)$ kills $\mathfrak{p}_{\geq 2}/\mathfrak{p}_{\geq4}$.  Thus, the image of $\widetilde X_{\reg}$ includes all of $X_{\reg}$, and the morphism $\widetilde X_{\reg}\rightarrow X_{\varphi,N}$ is proper.

We now know that $\widetilde X_{\reg}\rightarrow X_{\varphi,N}$ is a proper monomorphism, so it is a closed immersion.  Furthermore, the geometric points of its image are exactly those of $X_{\reg}$; since both $\widetilde X_{\reg}$ and $X_{\reg}$ are reduced, this shows that $\widetilde X_{\reg}\rightarrow X_{\varphi,N}$ is an isomorphism onto $X_{\reg}$.
\end{proof}

Combining these two results, we see that $X_{\reg}$ is smooth, and has a nice moduli description.

\subsection{The subregular nilpotent orbit of $\GL_3$}\label{subreg}
Let $G=\GL_3$, and assume again that $K_0=\Q_p$.  There are three geometric conjugacy classes in $\mathcal{N}_{\overline E}$, namely the orbits $\mathcal{O}_{\reg}$, $\mathcal{O}_{\sub}$, and $\{0\}$ of $N_{\reg}:=\left(\begin{smallmatrix}0&1&0\\0&0&1\\0&0&0\end{smallmatrix}\right)$, $N_{\sub}:=\left(\begin{smallmatrix}0&1&0\\0&0&0\\0&0&0\end{smallmatrix}\right)$, and $N_0:=\left(\begin{smallmatrix}0&0&0\\0&0&0\\0&0&0\end{smallmatrix}\right)$, respectively.  We have seen that the closure $X_{\reg}$ of $X_{\varphi,N}|_{\mathcal{O}_{\reg}}$ is smooth; it is connected, so it is irreducible.  At the other extreme, $X_0:=X_{\varphi,N}|_{N=0}$ is evidently smooth and irreducible.
We now treat the structure of the closure $(X_{\sub})_{\overline E}$ of $(X_{\varphi,N})_{\overline E}|_{\mathcal{O}_{\sub}}$ inside $(X_{\varphi,N})_{\overline E}$ and show that it is singular.  Going forward, we extend scalars on $X_{\varphi,N}$ from $E$ to $\overline E$ and suppress the subscript.  

The cocharacter $\lambda_{\sub}:\Gm\rightarrow\GL_3$ defined by $\lambda_{\sub}(t)=\left(\begin{smallmatrix}t&0&0\\0&t^{-1}&0\\0&0&1\end{smallmatrix}\right)$ is associated to $N_{\sub}$.  The Lie algebra $\mathfrak{gl}_3$ is graded by the action of $\lambda_{\sub}$, and the part which has weight at least $2$ is the $1$-dimensional subspace
\[	\mathfrak{g}_{\geq2}=\mathfrak{g}_2=\left\{\left(\begin{smallmatrix}0&\ast&0\\0&0&0\\0&0&0\end{smallmatrix}\right)\right\}	\]
which has weight exactly $2$.  
We also have
\[	\mathfrak{g}_{\geq0} = \left\{\left(\begin{smallmatrix}\ast&\ast&\ast\\0&\ast&0\\0&\ast&\ast\end{smallmatrix}\right)\right\}	\]
and
\[	\mathfrak{g}_{\geq1} = \left\{\left(\begin{smallmatrix}0&\ast&\ast\\0&0&0\\0&\ast&0\end{smallmatrix}\right)\right\}	\]
Then setting $P_{\sub}:=P_G(\lambda_{\sub})$, the resolution $G\times^{P_{\sub}}\mathfrak{g}_{2}$ of the closure $\overline{\mathcal{O}}_{\sub}$ of $\mathcal{O}_{\sub}$ carries a universal parabolic $\mathcal{P}$ and a line bundle corresponding to $\mathfrak{g}_{2}$.

We consider an auxiliary moduli problem $\widetilde X_{\sub}$:
\[	\widetilde X_{\sub}(A):=\{(\Phi,N)\in (P\times_{G/P_{\sub}}\mathfrak{p}_{2})(A)| (1-p\Ad(\Phi))|_{\mathfrak{p}_{2}}=0\}	\]
As before, there are natural morphisms $\widetilde X_{\sub}\rightarrow X_{\varphi,N}$ and $\widetilde X_{\sub}\rightarrow \mathcal{N}$.

\begin{prop}
$\widetilde X_{\sub}$ is smooth, as is the fiber $\widetilde X_{\sub}|_{N=0}$ over $N=0$.
\end{prop}
\begin{proof}
We use the functorial criterion for smoothness.  Let $A$ be an $\overline E$-algebra, let $I\subset A$ be an ideal such that $I^2=0$, and let $(\Phi_0,N_0,P_0)$ be an $A/I$-point of $\widetilde X_{\sub}$.  We wish to lift $(\Phi_0,N_0,P_0)$ to an $A$-point of $\widetilde X_{\sub}$.

First of all, $G/P_{\sub}$ is smooth, so we can lift $P_0$ to an $A$-point $P$ of $G/P_{\sub}$.  We claim that the space of $\Phi$ in $P$ such that $1-p\Ad(\Phi)$ kills $\mathfrak{p}_{2}$ is smooth over $\Spec A$.  To see this, we can work locally on $\Spec A$.  Since the quotient $G\rightarrow G/P_{\sub}$ admits sections Zariski-locally, we may assume that there is some $g_P\in G(A)$ such that $P=g_PP_{\sub}g_P^{-1}$.  Therefore, we may assume that $P=P_{\sub}$.  But the space of $\Phi$ such that $(1-p\Ad(\Phi))|_{\mathfrak{p}_{2}}=0$ is a torsor under the subgroup $Z_{P_{\sub}}(\mathfrak{p}_{2})\subset P_{\sub}$ which acts trivially on $\mathfrak{p}_2$, so it is smooth.

Thus, we can lift $\Phi_0$ to an $A$-point $\Phi$ of $P$ such that $1-p\Ad(\Phi)$ kills $\mathfrak{p}_{2}$, so $\widetilde X_{\sub}|_{N=0}$ is smooth.  It remains to lift $N_0$ to an $A$-point of the kernel of $1-p\Ad(\Phi)$.  But the kernel of $1-p\Ad(\Phi)$ on $\mathfrak{p}_{2}$ is a line bundle on $\Spec A$, so we can lift $N_0$.
\end{proof}

\begin{lemma}\label{P2-p2}
Let $A$ be a local ring.  If $P, P'\in (G/P_{\sub})(A)$ are parabolic subgroups of $G_A$ such that $\mathfrak{p}_2=\mathfrak{p}'_2$ as submodules of $\mathfrak{g}_A$, then $P=P'$.
\end{lemma}
\begin{proof}
After conjugating, we may assume that $P=P_{\sub}$, so that $\mathfrak{p}_2$ is generated by $N_{\sub}$.  Further, there is some $g\in G(A)$ such that $P'=gPg^{-1}$, and conjugation by $g$ defines an isomorphism of $\mathfrak{p}_2$ and $\mathfrak{p}'_2$.  Since $\mathfrak{p}_2=\mathfrak{p}'_2$ by assumption, $\Ad(g)(N_{\sub})=\alpha\cdot N_{\sub}$ for some $\alpha\in A^\times$, so
\[	g\in\{g\in G(A) | \Ad(g)(N_{\sub})=\alpha N_{\sub}\text{ for some }\alpha\in A^\times\}\subset P_{\sub}	\]
so $P'=P$.
\end{proof}

\begin{prop}
The natural morphism $\widetilde X_{\sub}\rightarrow X_{\varphi,N}$ is proper, and its image is $X_{\sub}$.  Moreover, it is an isomorphism above $X_{\sub}|_{\mathcal{O}_{\sub}}$.
\end{prop}
\begin{proof}
Suppose that an $\overline E$-point $(\Phi,N)\in\X_{\varphi,N}(\overline{E})$ is in the image of $\widetilde X_{\sub}$.  We need to show that $(\Phi,N)$ is in $X_{\sub}(\overline E)$.  By assumption, there is some parabolic $P$ conjugate to $P_{\sub}$ such that $\Phi\in P$ and $N\in\mathfrak{p}_{2}$.  If $N\neq 0$, then $(\Phi,N)$ is in $X_{\varphi,N}|_{\mathcal{O}_{\sub}}$.  If $N=0$, then we observe that $1-p\Ad(\Phi)$ acts trivially on $\mathfrak{p}_{2}$, so there is some non-zero $N'\in \mathfrak{p}_{2}$ such that $N'=p\Ad(\Phi)(N')$.  Then $(\Phi,tN')$ defines an $\mathbf{A}^1$-point of $\X_{\varphi,N}$ connecting $(\Phi,0)$ to $X_{\varphi,N}|_{\mathcal{O}_{\sub}}$.

Now we need to show that $\widetilde X_{\sub}\rightarrow X_{\varphi,N}$ is proper and surjects onto $X_{\sub}$.  Let $R$ be a discrete valuation ring, and let $\Spec R\rightarrow X_{\sub}$ be a morphism such that the generic point $\eta$ lands in $X_{\varphi,N}|_{\mathcal{O}_{\sub}}$.  This induces a family $(\Phi,N)$ of $(\varphi,N)$-modules over $R$.  Forgetting $\Phi$ yields a morphism $\eta\rightarrow \mathcal{O}_{\sub}$, and therefore an $R$-point $P$ of $G/P_{\sub}$, since $G/P_{\sub}$ is proper.  Since $N_\eta\in (\mathfrak{p}_{\eta})_{2}$ and this is a closed condition, we have $N\in\mathfrak{p}_{2}$.  Further, since $1-p\Ad(\Phi_\eta)$ kills $(\mathfrak{p}_{\eta})_{2}$, $1-p\Ad(\Phi)$ kills $\mathfrak{p}_{2}$.  Thus, the image of $\widetilde X_{\sub}$ includes all of $X_{\sub}$, and the morphism $\widetilde X_{\sub}\rightarrow X_{\varphi,N}$ is proper.

Finally, we need to show that $\widetilde X_{\sub}\rightarrow X_{\varphi,N}$ is an isomorphism when restricted to the preimage of $X_{\sub}|_{\mathcal{O}_{\sub}}$.  It suffices to show that it is a monomorphism, and since $\widetilde X_{\sub}\rightarrow X_{\varphi,N}$ is proper, this can be checked on geometric points of $X_{\sub}|_{\mathcal{O}_{\sub}}$.  Suppose the fiber over $(\Phi,N)\in X_{\sub}|_{\mathcal{O}_{\sub}}$ has more than one $\kappa(\overline{\eta})$-point, i.e., there are two parabolics $P,P'\in G/P_{\sub}$ such that $\Phi\in P\cap P'$ and $N\in \mathfrak{p}_2\cap \mathfrak{p}'_2$.  Then $N$ generates both $\mathfrak{p}_2$ and $\mathfrak{p}'_2$, so by Lemma~\ref{P2-p2}, $P=P'$.
\end{proof}

We see that $X_{\sub}$ is the image of a smooth connected variety, so $X_{\sub}$ is itself irreducible.

In contrast to the case of the regular nilpotent orbit, the morphism $\widetilde X_{\sub}\rightarrow X_{\sub}$ is not an isomorphism: if $\Phi=\left(\begin{smallmatrix}1&0&0\\0&p&0\\0&0&p^2\end{smallmatrix}\right)$, then $\left(\Phi,0\right)$ is a point of $X_{\sub}$.  However, the fiber of $\widetilde X_{\sub}\rightarrow X_{\varphi,N}$ over $\left(\Phi,0\right)$ contains distinct points corresponding to the parabolics $\left(\begin{smallmatrix}*&*&*\\0&*&0\\0&*&*\end{smallmatrix}\right)$ and $\left(\begin{smallmatrix}*&0&*\\ \ast&*&*\\0&0&*\end{smallmatrix}\right)$.

We claim more: 
\begin{thm}\label{xsub-sing}
$X_{\sub}$ is singular at every point $(\Phi,0)$ with more than one pre-image in $\widetilde X_{\sub}$.  
\end{thm}
\begin{proof}
For such a $\Phi$, there exist distinct parabolics $P,P'\subset G$  with $\Phi\in P\cap P'$ and $(1-p\Ad(\Phi))|_{\mathfrak{p}_2}=(1-p\Ad(\Phi))|_{\mathfrak{p}'_2}=0$.  There are natural maps of tangent spaces $T_{(\Phi,0,P)}\widetilde X_{\sub}\rightarrow T_{(\Phi,0)}X_{\sub}$ and $T_{(\Phi,0,P')}\widetilde X_{\sub}\rightarrow T_{(\Phi,0)}X_{\sub}$; we will study their kernels and images.

After conjugating, we may assume that $P=P_{\sub}$.  The tangent space of $\widetilde X_{\sub}$ at $(\Phi,0,P_{\sub})$ consists of deformations $(\widetilde\Phi,\widetilde N,\widetilde P)$ such that $\widetilde\Phi\in\widetilde P$, $\widetilde N\in\widetilde{\mathfrak{p}}_2$, and $1-p\Ad(\widetilde\Phi)$ kills $\widetilde{\mathfrak{p}}_2$, where $\widetilde{\mathfrak{p}}$ is the Lie algebra of $\widetilde P$ and $\widetilde{\mathfrak{p}}_2$ is its weight $2$ part.  The kernel of the morphism $T_{(\Phi,0,P)}\widetilde X_{\sub}\rightarrow T_{(\Phi,0)}X_{\sub}$ consists of deformations $\widetilde P$ of $P_{\sub}$ such that $(\Phi,0,\widetilde P)$ is an element of $T_{(\Phi,0,P)}\widetilde X_{\sub}$.  If there are two such deformations $\widetilde P_1,\widetilde P_2$, the weight $2$ parts of their Lie algebras are generated by $\Ad(1+\varepsilon g_i)(N_{\sub})$, respectively, where $i=1,2$ and $g_i\in\mathfrak{g}$.  But $\Ad(1+\varepsilon g_i)(N_{\sub})=N_{\sub}+\varepsilon[g_i,N_{\sub}]$ and $1-p\Ad(\widetilde\Phi)$ must kill both $N_{\sub}+\varepsilon[g_1,N_{\sub}]$ and $N_{\sub}+\varepsilon[g_2,N_{\sub}]$, so we must have $(1-p\Ad(\Phi))([g_1-g_2,N_{\sub}])=0$ (since $\varepsilon\Ad(\widetilde\Phi)=\varepsilon\Ad(\Phi)$ by assumption).

We claim there is at most a $2$-dimensional space of elements of $\ker(1-p\Ad(\Phi)):\mathfrak{g}\rightarrow\mathfrak{g}$ of the form $[g,N_{\sub}]$; combined with Lemma~\ref{P2-p2}, this implies that the kernel of $T_{(\Phi,0,P)}\widetilde X_{\sub}\rightarrow T_{(\Phi,0)}X_{\sub}$ is at most $1$-dimensional.  The image of $[N_{\sub},-]:\mathfrak{g}\rightarrow\mathfrak{g}$ lands in $\mathfrak{p}$, and $\ker(1-p\Ad(\Phi))$ generates a nilpotent subalgebra.  But $\mathfrak{p}_{\geq 1}$ is $3$-dimensional and $[\mathfrak{p}_{\geq1},\mathfrak{p}_{\geq1}]=\mathfrak{p}_2$, so if $1-p\Ad(\Phi)$ kills a $2$-dimensional subspace of $\mathfrak{p}_{\geq1}/\mathfrak{p}_2$, it cannot kill $\mathfrak{p}_2$.

If $\ker(1-p\Ad(\Phi))|_{\mathfrak{p}}$ is $1$-dimensional, as in the example above with $\Phi=\left(\begin{smallmatrix}1&0&0\\0&p&0\\0&0&p^2\end{smallmatrix}\right)$, this shows that the map $T_{(\Phi,0,P_{\sub})}\widetilde X_{\sub}\rightarrow T_{(\Phi,0)}X_{\sub}$ is injective.  But the images of $T_{(\Phi,0,P)}\widetilde X_{\sub}$ and $T_{(\Phi,0,P')}\widetilde X_{\sub}$ in $T_{(\Phi,0)}X_{\sub}$ are distinct (as $(\Phi,\varepsilon N_{\sub})\neq (\Phi,\varepsilon N')$, where $N'$ is the subregular nilpotent element generating $\mathfrak{p}'_2$), so together they generate a subspace of $T_{(\Phi,0)}X_{\sub}$ of dimension strictly greater than $9$, and $(\Phi,0)$ is a singular point of $X_{\sub}$.  

On the other hand, suppose that $\ker(1-p\Ad(\Phi))\cap \rm{im}([N_{\sub},-])$ is $2$-dimensional.  This happens, for example, if $\Phi=\left(\begin{smallmatrix}1&0&0\\0&p&0\\0&0&p\end{smallmatrix}\right)$.  Then the image of $T_{(\Phi,0,P_{\sub})}\widetilde X_{\sub}$ in $T_{(\Phi,0)}X_{\sub}$ is $8$-dimensional; to show that $(\Phi,0)$ corresponds to a singular point of $X_{\sub}$, we need to show that not all deformations $(\widetilde\Phi,0)$ coming from $T_{(\Phi,0,P_{\sub})}\widetilde X_{\sub}$ (which is a $7$-dimensional subspace) also come from $T_{(\Phi,0,P')}\widetilde X_{\sub}$.  Indeed, if two $8$-dimensional subspaces of $T_{(\Phi,0)}X_{\sub}$ intersect in a subspace of dimension at most $6$, they generate a subspace of dimension at least $10$.

We first note that $\ker(1-p\Ad(\Phi))$ generates a nilpotent subalgebra of $\mathfrak{g}$, and by~\cite[Theorem 2.2]{lmt}, any nilpotent subalgebra is contained in a Borel.  But $1-p\Ad(\Phi)$ cannot kill the entire nilpotent part of a Borel subalgebra, which is $3$-dimensional, so $\ker(1-p\Ad(\Phi))$ is at most $2$-dimensional, which implies that $\ker(1-p\Ad(\Phi))=\ker(1-p\Ad(\Phi))\cap \rm{im}([N_{\sub},-])\subset\mathfrak{p}$.  

Now we can compute explicitly.  If $(1-p\Ad(\Phi))(N_{\sub})=0$, then $\Phi$ is of the form $\left(\begin{smallmatrix}a&\ast&\ast\\0&pa&0\\0&\ast&\ast\end{smallmatrix}\right)$.  We may assume that $\widetilde P=P_{\sub}$; then $\widetilde\Phi$ has the same form.  If $1-p\Ad(\Phi)$ kills an additional element $N'\in\mathfrak{p}$, then either $N'=\left(\begin{smallmatrix}0&0&\ast\\0&0&0\\0&0&0\end{smallmatrix}\right)$ and $\Phi$ is of the form $\left(\begin{smallmatrix}a&\ast&\ast\\0&pa&0\\0&0&pa\end{smallmatrix}\right)$, or $N'=\left(\begin{smallmatrix}0&0&0\\0&0&0\\0&\ast&0\end{smallmatrix}\right)$ and $\Phi$ is of the form $\left(\begin{smallmatrix}a&\ast&0\\0&pa&0\\0&\ast&a\end{smallmatrix}\right)$.  

Assume that $N'=\left(\begin{smallmatrix}0&0&1\\0&0&0\\0&0&0\end{smallmatrix}\right)$; the argument in the second case will be similar.  Then if $\widetilde\Phi$ comes from $T_{(\Phi,0,P')}\widetilde X_{\sub}$, $1-p\Ad(\widetilde\Phi)$ kills an element of the form $N'+\varepsilon N''$, where $N''=[g,N']$ for some $g\in\mathfrak{g}$.  We have
\[	
\begin{split}
N'+\varepsilon N''=p\Ad(\widetilde\Phi)(N'+\varepsilon N'') = p\Ad(\widetilde\Phi)(N')+p\varepsilon\Ad(\widetilde\Phi)(N'') = \Ad(\widetilde\Phi\Phi^{-1})(N') + p\varepsilon\Ad(\Phi)(N'') \\
= N' + \varepsilon[\widetilde\Phi\Phi^{-1}-1,N'] + p\varepsilon\Ad(\Phi)(N'')	
\end{split}
\]
In other words, $N''$ is such that $(1-p\Ad(\Phi))(N'')\in [N',\mathfrak{p}]$.
By computation, $[N',\mathfrak{p}]$ is the space 
\[	\left(\begin{smallmatrix}0&\ast&\ast\\0&0&0\\0&0&0\end{smallmatrix}\right)=\ker(1-p\Ad(\Phi))\subset\mathfrak{p}_{\geq1}	\]
On the other hand, $\Ad(\Phi)$ acts on $\mathfrak{g}_{\geq-2}/\mathfrak{g}_{\geq-1}$ by multiplication by $p$, so if $N''$ exists, it lies in $\mathfrak{g}_{\geq-1}$.  Further, the action of $\Ad(\Phi)$ on $\mathfrak{g}_{\geq-1}/\mathfrak{g}_{\geq0}$ has two eigenspaces, one with eigenvalue $1$ and one with eigenvalue $p$, so we are looking for $N''$ in $\mathfrak{g}_{\geq0}=\mathfrak{p}$.  Since $\Ad(\Phi)$ acts trivially on $\mathfrak{p}/\mathfrak{p}_{\geq1}$, $N''$ must lie in $\mathfrak{p}_{\geq1}$.  But $\Ad(\Phi)$ acts diagonalizably on $\mathfrak{p}_{\geq1}/\mathfrak{p}_2$ with eigenvalues $1$ and $p^{-1}$, so if $(1-p\Ad(\Phi))(N'')\in [N',\mathfrak{p}]\subset \ker(1-p\Ad(\Phi))$, then $N''$ is already in the kernel of $1-p\Ad(\Phi)$.

To summarize, if $\widetilde\Phi$ comes from both $T_{(\Phi,0,P_{\sub})}$ and $T_{(\Phi,0,P')}\widetilde X_{\sub}$, then $1-p\Ad(\widetilde\Phi)$ kills both $N_{\sub}$ and $N'+\varepsilon N''$, where $(1-p\Ad(\Phi))(N'')=0$.  But then $1-p\Ad(\widetilde\Phi)$ kills both $N_{\sub}$ and $N'$, so $\widetilde\Phi$ is of the form $\left(\begin{smallmatrix}a&\ast&\ast\\0&pa&0\\0&0&pa\end{smallmatrix}\right)$.  Since there are plainly choices of $\widetilde\Phi$ which do not lie in this space, $(\Phi,0)$ is a singular point of $X_{\sub}$.
\end{proof}
We conclude by remarking on the singular points of $X_{\sub}$ we have constructed.
Part of the singular locus of $X_{\sub}$ is $X_{\sub}\cap X_{\reg}\cap X_0$.  More precisely, suppose we have a pair $(\Phi,N)$ such that $(1-p\Ad(\Phi))(N)=0$ and $N$ is regular nilpotent.  After conjugating, we may assume that $N=\left(\begin{smallmatrix}0&1&0\\0&0&1\\0&0&0\end{smallmatrix}\right)$.  Then $(\Phi,0)$ is a singular point of $X_{\sub}$, since if $\Phi=\left(\begin{smallmatrix}a&b&c\\0&pa&pb\\0&0&p^2a\end{smallmatrix}\right)$, then $1-p\Ad(\Phi)$ kills the subregular nilpotent elements $\left(\begin{smallmatrix}0&1&ba^{-1}(1-p)^{-1}\\0&0&0\\0&0&0\end{smallmatrix}\right)$ and $\left(\begin{smallmatrix}0&0&ba^{-1}(p-1)^{-1}\\0&0&1\\0&0&0\end{smallmatrix}\right)$, which have distinct parabolics attached to them.

However, there are other singular points in $X_{\sub}$.  For example, let $\Phi=\left(\begin{smallmatrix}1&0&0\\0&p&0\\0&0&p\end{smallmatrix}\right)$.  The kernel of $1-p\Ad(\Phi):\mathfrak{gl}_3\rightarrow\mathfrak{gl}_3$ is the $2$-dimensional vector space $\left\{\left(\begin{smallmatrix}0&\ast&\ast\\0&0&0\\0&\ast&0\end{smallmatrix}\right)\right\}$, which consists entirely of subregular nilpotent elements.  The parabolic attached to $\left(\begin{smallmatrix}0&1&0\\0&0&0\\0&0&0\end{smallmatrix}\right)$ is $P_{\sub}$, while the parabolic attached to $\left(\begin{smallmatrix}0&0&1\\0&0&0\\0&0&0\end{smallmatrix}\right)$ is $\left\{\left(\begin{smallmatrix}\ast&\ast&\ast\\0&\ast&\ast\\0&0&\ast\end{smallmatrix}\right)\right\}$, so $(\Phi,0)$ is a singular point of $X_{\sub}$.  If $(\Phi,0)$ were a point of $X_{\reg}$, then $\ker(1-p\Ad(\Phi))$ would generate the nilpotent part of the Lie algebra of a Borel of $G$, which is $3$-dimensional.  But the Lie bracket of any two elements of $\ker(1-p\Ad(\Phi))$ is trivial, so $(\Phi,0)$ does not lie in $X_{\reg}$.

This dichotomy corresponds to the dichotomy in the proof of Theorem~\ref{xsub-sing}.  The subregular elements we wrote down for $\Phi$ such that $(\Phi,0)\in X_{\sub}\cap X_{\reg}\cap X_0$ have the property that $[N,\mathfrak{g}]\cap\ker(1-p\Ad(\Phi))$ is $1$-dimensional.  However, if $\ker(1-p\Ad(\Phi))$ is $2$-dimensional but consists entirely of subregular elements (so that $(\Phi,0)$ is a singular point of $X_{\sub}$ which does not lie in $X_{\reg}$), then the fiber of $\widetilde X_{\sub}\rightarrow X_{\sub}$ is isomorphic to $\P^1$ and so we can deform the parabolics attached to $\Phi$.

\appendix

\section{Tannakian formalism}\label{tannakian}

The theory of Tannakian categories enables us to study algebraic groups over fields in terms of their faithful representations.  The theory is developed in detail in~\cite{dm} and~\cite{saavedra}; we recall some of the basics here, and work out a number of useful examples.

\subsection{Fiber functors}

Let $k$ be a field, and let $A$ be a $k$-algebra.  Let $G$ be an affine $k$-group scheme.

\begin{definition}
A \emph{fiber functor} $\omega:\Rep_k(G)\rightarrow \Proj_A$ is a functor from the category of $k$-linear finite-dimensional representations of $G$ to the category of finite projective $A$-modules such that
\begin{enumerate}
\item	$\omega$ is $k$-linear, exact, and faithful
\item\label{tensor-def}	$\omega$ is a tensor functor, that is, $\omega(V_1\otimes_kV_2)=\omega(V_1)\otimes_A\omega(V_2)$
\item	If $\mathbf{1}$ denotes the trivial representation of $G$, then $\eta(\mathbf{1})$ is the trivial $A$-module of rank $1$.
\end{enumerate}
\end{definition}
We have suppressed some compatibilities in our definitions, in particular on the isomorphism in part~\ref{tensor-def}.  We refer the reader to~\cite[\textsection 1]{dm} for a full discussion of tensor categories and tensor functors.

Given a fiber functor $\omega:\Rep_k(G)\rightarrow \Proj_A$ and an $A$-algebra $A'$, there is a natural fiber functor $\omega':\Rep_k(G)\rightarrow \Proj_{A'}$ given by composing $\omega$ with the natural base extension functor $\varphi_{A'}:\Proj_A\rightarrow\Proj_{A'}$ sending $M$ to $M\otimes_AA'$.

\begin{definition}
Let $\omega,\eta:\Rep_k(G)\rightrightarrows \Proj_A$ be fiber functors.  Then $\underline\Hom^\otimes(\omega,\eta)$ is the functor on $A$-algebras given by
$$\underline\Hom^\otimes(\omega,\eta)(A'):=\Hom^\otimes(\varphi_{A'}\circ\omega,\varphi_{A'}\circ\eta)$$
Here $\Hom^\otimes$ refers to natural transformations of functors which preserve tensor products.
\end{definition}

\begin{thm}[{\cite[Prop. 2.8]{dm}}]\label{g-auts}
Let $\omega:\Rep_k(G)\rightarrow \Vect_k$ be the natural forgetful functor from the category of $k$-linear finite-dimensional representations of $G$ to the category of finite-dimensional $k$-vector spaces.  Then the natural morphism of functors on $k$-algebras $G\rightarrow \underline\Aut^\otimes(\omega)$ is an isomorphism.
\end{thm}
\begin{remark}
Deligne and Milne actually prove more than this; they show that given an abstract neutral $k$-linear Tannakian category $C$ with fiber functor $\omega:C\rightarrow \Vect_k$, the functor $\underline\Aut^\otimes(\omega)$ is representable by an affine $k$-group scheme $G$.  However, we will not need this level of generality.
\end{remark}

\begin{definition}
A (right) \emph{$G$-torsor} over an affine $k$-scheme $\Spec A$ is an affine morphism $X\rightarrow \Spec A$ which is faithfully flat over $A$, together with an action $X\times G_A\rightarrow X$ so that the morphism $X\times G_A\rightarrow X\times X$ defined by $(x,g)\mapsto (x,x\cdot g)$ is an isomorphism.
\end{definition}
\begin{remark}
In fact, the assumption that $X$ is affine follows by fpqc descent from the other properties, plus the seemingly milder hypothesis that the morphism $X\rightarrow\Spec A$ is fpqc.
\end{remark}
\begin{remark}
Suppose that $G$ is smooth.  Then if $\Spec A'\rightarrow \Spec A$ is an fpqc base change which trivializes $X$, we see that $X_{A'}\rightarrow \Spec A'$ is smooth.  Smoothness descends along fpqc morphisms, so $X\rightarrow \Spec A$ is smooth as well.  It follows that $X$ can actually be trivialized by an \'etale surjective base change on $A$.
\end{remark}
\begin{remark}
Suppose that $G$ is an affine algebraic group.  Then if $\Spec A'\rightarrow \Spec A$ is an fpqc base change which trivializes $X$, we see that $X_{A'}\rightarrow \Spec A'$ is fppf.  Being fppf descends along fpqc morphisms, so $X\rightarrow \Spec A$ is fppf as well.  It follows that $X$ can actually be trivialized by an fppf base change on $A$.
\end{remark}

\begin{thm}[{\cite[Thm. 3.2]{dm}\label{g-x-homs}}]
Let $\omega:\Rep_k(G)\rightarrow \Vect_k$ be the natural forgetful functor.
\begin{enumerate}
\item	For any fiber functor $\eta:\Rep_k(G)\rightarrow \Proj_A$, $\underline\Hom^\otimes(\varphi_A\circ\omega,\eta)$ is representable by an affine scheme faithfully flat over $\Spec A$; it is therefore a $G$-torsor.
\item	The functor $\eta\rightsquigarrow\underline\Hom^\otimes(\varphi_A\circ\omega,\eta)$ is an equivalence between the category of fiber functors $\eta:\Rep_k(G)\rightarrow\Proj_A$ and the category of $G$-torsors over $\Spec A$.  The quasi-inverse assigns to any $G$-torsor $X$ over $A$ the functor $\eta$ sending any $\rho:G\rightarrow \GL(V)$ to the $M\in\Proj_A$ associated to the push-out of $X$ over $A$.
\end{enumerate}
\end{thm}

\begin{cor}
Let $\eta:\Rep_k(G)\rightarrow \Proj_A$ be a fiber functor, corresponding to a $G$-torsor $X\rightarrow \Spec A$.  Then the functor $\underline\Aut^\otimes(\eta)$ is representable by the $A$-group scheme $\Aut_G(X)$.  This is a form of $G_A$.
\end{cor}

Let $G'$ be another affine $k$-group scheme, and let $f:G\rightarrow G'$ be a homomorphism of $k$-group schemes.  Then there is a push-out construction in the style of ``associated bundles''.  Namely, the space $X\times_{\Spec A}G'_A$ carries a right action of $G_A$, via 
$$(x,g')\cdot g = (x\cdot g, f(g^{-1})g')$$
where $x,g',g$ are $A'$-points of $X,G',G$, respectively.  Then we define 
$$X'=(X\times_{\Spec A}G'_A)/G$$
The existence of this quotient must be justified.  We claim it is sufficient to construct $X'_{A'}$, where $A\rightarrow A'$ is an fpqc morphism.  There is a descent datum on $(X\times_{\Spec A}G'_A)_{A'}$ because it is the base change of an $A$-scheme, and the action of $G$ respects this descent datum.  Therefore, if $X_{A'}$ exists, it is equipped with a descent datum.  But $X'_{A'}\rightarrow\Spec A'$ is affine and descent is effective in the affine case, so the existence of $X'_{A'}$ implies the existence of $X'$.

Now take $A\rightarrow A'$ to be an fpqc morphism which splits $X$.  Then 
$$X_{A'}\times_{\Spec A'}G'_{A'} = G_{A'}\times_{\Spec A'}G'_{A'}$$
and the quotient by $G_{A'}$ is visibly $G'_{A'}$.

We can also see this on the level of fiber functors as follows.  Suppose that $X$ corresponds to the fiber functor $\eta:\Rep_k(G)\rightarrow \Proj_A$.  We may define a fiber functor $\eta':\Rep_k(G')\rightarrow \Proj_A$ by taking
\[	\eta'(\rho) = \eta(\rho\circ f)	\]
for every representation $\rho:G'\rightarrow \GL(V)$.

These constructions can readily be checked to be equivalent.  In particular, given a representation $\rho:G\rightarrow \GL_n$, the push-out bundle is the $\GL_n$-torsor associated to the vector bundle $\omega(\rho)$.

We will be interested in the these constructions in the case when $G$ is a linear algebraic $k$-group.  By~\cite[Prop. 2.20]{dm}, this is the case if and only if $\Rep_k(G)$ has a tensor generator $V$.  That is, every object of $\Rep_k(G)$ is isomorphic to a subquotient of some direct sum of tensor powers of $V$ and $V^\ast$.  In fact, if $G$ is algebraic, then any faithful representation of $G$ is a tensor generator of $\Rep_k(G)$.

\subsection{Examples}

We give a number of examples which are relevant to $p$-adic Hodge theory.  As in the previous section, we let $k$ be a field, $G$ be an affine $k$-group, and $A$ be a $k$-algebra.  As our primary interest will be in the case where $A$ is a $k$-affionid algebra, we do not assume $A$ is finite type.  Several of these example rely on results proved in~\cite{bellovin}.

\subsubsection{$\Gm$}
Let $G=\Gm$, and let $\rho:G\rightarrow \GL(V)$ be a representation of $\Gm$.  Then $V$ decomposes as $V=\oplus_{n\in\Z} V_n$, where $\Gm$ acts on $V_n$ via multiplication by $t\mapsto t^n$.  These decompositions are exact and tensor-compatible, in the sense that if
\[	0\rightarrow V'\rightarrow V\rightarrow V''\rightarrow 0	\]
is an exact sequence of representations of $\Gm$, then
\[	0\rightarrow V'_n\rightarrow V_n\rightarrow V''_n\rightarrow 0	\]
is exact for all $n$, and 
\[	(V\otimes_kV')_n=\bigoplus_{p+q=n}V_p\otimes_kV'_q	\]
Thus, $\Rep_k(G)$ is isomorphic (as a rigid tensor category) to the category of graded vector spaces.  It is generated by the $1$-dimensional representation on which $\Gm$ acts by scaling.

\subsubsection{Gradings}
Let $X\rightarrow \Spec A$ be a $G$-torsor, corresponding to a fiber functor $\eta:\Rep_k(G)\rightarrow \Proj_A$.  A \emph{$\otimes$-grading} of $\eta$ is the specification of a grading $\eta(V)=\oplus_{n\in\Z}\eta(V)_n$ of vector bundles on each $\eta(V)$ such that
\begin{enumerate}
\item	the specified gradings are functorial in $V$
\item	the specified grading are tensor-compatible, in the sense that
\[	\eta(V\otimes_kV')_n=\bigoplus_{p+q=n}(\eta(V)_p\otimes\eta(V')_q)	\]
\item	$\eta(\mathbf{1})_0=\eta(\mathbf{1})$
\end{enumerate}
Equivalently, a  $\otimes$-grading of $\eta$ is a factorization of $\eta$ through the category of graded vector bundles on $\Spec A$.

Then for any $A$-algebra $A'$ and any point $t\in\Gm(A')$, we define a natural transformation from $\varphi_{A'}\circ\eta$ to itself, via the family of homomorphisms
\[	\bigoplus_{n\in\Z}t^n:\oplus_{n\in\Z}\left(\varphi_{A'}\circ\eta(V)\right)_n\rightarrow \bigoplus_{n\in\Z}\left(\varphi_{A'}\circ\eta(V)\right)_n	\]
This is clearly functorial in $V$, and it is exact and tensor-compatible.  Thus, we have defined a homomorphism $\Gm(A')\rightarrow \Aut^\otimes(\varphi_{A'}\circ\eta)=\Aut_G(X)(A')$.  But it is clearly functorial in $A'$, so we get a homomorphism of $A$-group schemes $\Gm\rightarrow \Aut_G(X)$.

\begin{example}
Let $G$ be a linear algebraic group over $\Q_p$, and let $\rho:\Gal_K\rightarrow G(\Q_p)$ be a continuous representation such that the composition $\sigma\circ\rho:\Gal_K\rightarrow\GL(V)$ is Hodge-Tate for every representation $\sigma:G\rightarrow\GL(V)$.  Then $\eta:V\mapsto \D_{\HT}^K(\sigma\circ\rho)$ is a fiber functor $\eta:\Rep_{\Q_p}\rightarrow \Vect_K$ equipped with a $\otimes$-grading.  Thus, we get a $G_K$-torsor $\D_{\HT}^K(\rho)$ over $\Spec K$, together with a cocharacter $\Gm\rightarrow\Aut_G(\D_{\HT}^K(\rho))$.
\end{example}

\subsubsection{Filtered vector bundles}\label{tann-filt}
Let $X\rightarrow \Spec A$ be a $G$-torsor, corresponding to a fiber functor $\eta:\Rep_k(G)\rightarrow \Proj_A$.  A \emph{$\otimes$-filtration} of $\eta$ is the specification of a decreasing filtration $\mathcal{F}^\bullet(\eta(V))$ of vector sub-bundles on each $\eta(V)$ such that
\begin{enumerate}
\item	the specified filtrations are functorial in $V$
\item	the specified filtrations are tensor-compatible, in the sense that
$$\mathcal{F}^n\eta(V\otimes_kV')=\sum_{p+q=n}\mathcal{F}^p\eta(V)\otimes\mathcal{F}^q\eta(V')\subset V\otimes V'$$
\item	$\mathcal{F}^n(\eta(\mathbf{1}))=\eta(\mathbf{1})$ if $n\leq 0$ and $\mathcal{F}^n(\eta(\mathbf{1}))=0$ if $n\geq 1$
\item	the associated functor from $\Rep_k(G)$ to the category of graded projective $A$-modules is exact.
\end{enumerate}
Equivalently, a $\otimes$-filtration of $\eta$ is the same as a factorization of $\eta$ through the category of filtered vector bundles over $\Spec A$.

We define two auxiliary subfunctors of $\underline\Aut^\otimes(\eta)$.
\begin{itemize}
\item	$P_{\mathcal{F}}=\underline\Aut_{\mathcal{F}}^\otimes(\eta)$ is the functor on $A$-algebras such that
\begin{equation*}
\begin{split}
\underline\Aut_{\mathcal{F}}^\otimes(\eta)(A') = \{\lambda\in\Aut_G(\eta)(A')| \lambda(\mathcal{F}^n\eta(V))&\subset\mathcal{F}^n\eta(V)\text{ for all }\\V\in\Rep_k(G)\text{ and }n\in\Z\}
\end{split}
\end{equation*}
\item	$U_{\mathcal{F}}=\underline\Aut_{\mathcal{F}}^{\otimes!}(\eta)$ is the functor on $A$-algebras such that
\begin{equation*}
\begin{split}
\underline\Aut_{\mathcal{F}}^\otimes(\eta)(A') = \{\lambda\in\Aut_G(\eta)(A')| (\lambda-id)(\mathcal{F}^n\eta(V))&\subset\mathcal{F}^{n+1}\eta(V)\text{ for all }\\V\in\Rep_k(G)\text{ and }n\in\Z\}
\end{split}
\end{equation*}
\end{itemize}
By~\cite[Chapter IV, 2.1.4.1]{saavedra}, these functors are both representable by closed subgroup schemes of $\Aut_G(X)$, and they are smooth if $G$ is.

Given a $\otimes$-grading of $\eta$, we may construct a $\otimes$-filtration of $\eta$, by setting
$$\mathcal{F}^n\eta(V)=\oplus_{n'\geq n}\eta(V)_{n'}$$
We say that a $\otimes$-filtration is \emph{splittable} if it arises in this way, and we say that $\otimes$-filtration is \emph{locally splittable} if it arises in this way, fpqc-locally on $\Spec A$.

In fact, if $k$ is a characteristic $0$ field or $G$ is a reductive group, then every $\otimes$-filtration is Zariski-locally splittable.  This is a theorem of Deligne, proved in~\cite[Chapter IV, 2.4]{saavedra}.

Now assume that $G$ is connected reductive.  Then we have the following results.
\begin{thm}[{\cite[Chapter IV, 2.2.5]{saavedra}}]
\begin{enumerate}
\item	$P_{\mathcal{F}}$ is represented by a parabolic subgroup of $\Aut_G(X)$ and $U_{\mathcal{F}}$ is represented by the unipotent radical of $P_{\mathcal{F}}$.  The Lie algebra $\Lie\Aut_G(X)$ itself has a filtration (via the filtration on $\eta$ applied to the adjoint representation of $G$), and $\Lie P_{\mathcal{F}}=\mathcal{F}^0\Lie\Aut_G(X)$ and $\Lie U_{\mathcal{F}}=\mathcal{F}^1\Lie\Aut_G(X)$.
\item	Let $\mu:\Gm\rightarrow\Aut_G(X)$ be a cocharacter corresponding to a splitting of the filtration.  Then $P_{\mathcal{F}}=P_{\Aut_G(X)}(\mu)$ and $Z_{\Aut_G(X)}(\mu)$ is a Levi subgroup of $P_{\mathcal{F}}$.
\item	If $\mu,\mu':\Gm\rightarrow \Aut_G(X)$ are two cocharacters splitting the filtration, they are conjugate by a unique $A$-point of $U_{\mathcal{F}}$.  More precisely, the functor of splittings of $\mathcal{F}$ is a $U_{\mathcal{F}}$-torsor, and as $\Spec A$ is affine and $U_{\mathcal{F}}$ is unipotent, it is a trivial $U_{\mathcal{F}}$-torsor.
\end{enumerate}
\end{thm}

The following lemma on conjugacy of cocharacters is well-known, but we provide a proof for the convenience of the reader.
\begin{lemma}\label{alg-conj}
Let $F$ be a separably closed field, let $G$ be a connected reductive group over $F$, and let $\lambda:\Gm\rightarrow G$ be a cocharacter defined over an separably closed extension $K/F$.  Then the $G(K)$-orbit of $\lambda$ (under conjugation) contains a cocharacter of $G$ defined over $F$, and all such cocharacters are $G(F)$-conjugate.
\end{lemma}
\begin{proof}
For any separably closed field $F$, the set of $G(F)$-orbits of cocharacters is naturally identified with $G(F)\backslash \Hom_F(\Gm,G)$.  Moreover, for any maximal $F$-torus $T\subset G$, there is a natural map ${\rm{X}}_\ast(T)\rightarrow G(F)\backslash \Hom_F(\Gm,G)$, where ${\rm{X}}_\ast(T)$ is the set of cocharacters of $T$.  Since $F$ is separably closed, all maximal $F$-tori are split and $G(F)$-conjugate; since the image of any homomorphism $\lambda:\Gm\rightarrow G$ is contained in some maximal torus, this map is surjective.  

On the other hand, if $\lambda,\lambda':\Gm\rightarrow T$ are two $G(F)$-conjugate cocharacters, then the image of $\lambda'$ is also contained in the maximal torus $gTg^{-1}$ for some $g\in G(F)$.  Thus, $T$ and $gTg^{-1}$ are maximal tori in $Z_G(\lambda')$, so they are $Z_G(\lambda')(F)$-conjugate, and $\lambda$ and $\lambda'$ are $N_G(T)(F)$-conjugate.  It follows that we have a bijection $N_G(T)(F)\backslash{\rm{X}}_\ast(T)\rightarrow G(F)\backslash \Hom_F(\Gm,G)$.  Since the quotient map $N_G(T)\rightarrow W_G(T)=N_G(T)/Z_G(T)$ is smooth and $F$ is separably closed, $W_G(T)(F)=N_G(T)(F)/Z_G(T)(F)$ and we have a bijection $W_G(T)(F)\backslash{\rm{X}}_\ast(T)\rightarrow G(F)\backslash \Hom_F(\Gm,G)$.

The left side is insensitive to change in $F$, because $W_G(T)$ is finite \'etale and the category of $F$-tori is anti-equivalent to the category of finite free $\Z$-modules (as $\Gal_F$ is trivial).  We therefore see that 
\[	G(F)\backslash \Hom_F(\Gm,G)\xleftarrow{\sim}W_G(T)(F)\backslash{\rm{X}}_\ast(T)\cong W_G(T)(K)\backslash{\rm{X}}_\ast(T)\xrightarrow{\sim} G(K)\backslash \Hom_K(\Gm,G)	\]
as desired.
\end{proof}

A \emph{type} is a conjugacy class of cocharacters $\mathbf{v}:\Gm\rightarrow G_{k^{\rm{sep}}}$.  If $X\rightarrow \Spec A$ is a $G$-torsor corresponding to a fiber functor $\eta$, then Lemma~\ref{alg-conj} shows that a $\otimes$-filtration on $\eta$ induces a well-defined type at every point $x\in\Spec A$.  We claim that the type is Zariski-locally constant on $\Spec A$.

To see this, we first prove the following lemma:
\begin{lemma}\label{hens-conj}
Suppose $A'$ is a strictly henselian local $F$-algebra, where $F$ is separably closed.  If $\lambda:\Gm\rightarrow G(A')$ is a cocharacter, then the $G(A')$-conjugacy class of $\lambda$ contains a cocharacter defined over $F$, and all such cocharacters are $G(F)$-conjugate.
\end{lemma}
\begin{proof}
As in the proof of Lemma~\ref{alg-conj}, we fix a maximal $F$-torus $T\subset G$ and consider the natural map ${\rm{X}}_\ast(T)\rightarrow G(A')\backslash \Hom_{A'}(\Gm,G_{A'})$.  For any cocharacter $\lambda:\Gm\rightarrow G_{A'}$, there is a (fiberwise) maximal torus of $G_{A'}$ containing the image of $\lambda$ since $A'$ is strictly henselian.  Indeed, the centralizer $Z_{G_{A'}}(\lambda)$ is a reductive $A'$-group scheme whose maximal $A'$-tori are maximal $A'$-tori of $G_{A'}$, and the existence of maximal $A'$-tori in $Z_{G_{A'}}(\lambda)$ follows from~\cite[Cor. 3.2.7]{conrad-luminy} (again using that $A'$ is strictly henselian).  Moreover, maximal tori of $G_{A'}$ are $G(A')$ conjugate by~\cite[Thm. 3.2.6]{conrad-luminy}, so we may assume that the image of $\lambda$ is contained in $T_{A'}$.

Again as in the proof of Lemma~\ref{alg-conj}, we see that we have a bijection \[	W_G(T_{A'})(A')\backslash {\rm{X}}_\ast(T_{A'})\xrightarrow{\sim} G(A')\backslash \Hom_{A'}(\Gm,G_{A'})	\]
We further have a natural surjection ${\rm{X}}_\ast(T)\twoheadrightarrow W_G(T_{A'})(A')\backslash {\rm{X}}_\ast(T_{A'})$; we claim that if $\lambda,\lambda':\Gm\rightrightarrows T$ are $W_G(T_{A'})(A')$-conjugate, they are $W_G(T)(F)$-conjugate.  Suppose $\lambda=g\lambda'g^{-1}$, where $g\in W_G(T_{A'})$, and let $\overline g$ denote the reduction of $g$ modulo the maximal ideal of $A'$.  Then $\overline g$ is defined over $F$, so $g\overline g^{-1}\in W_G(T_{A'})$ is residually trivial.  But $W_G(T_{A'})$ is finite \'etale over $A'$, so $g=\overline g$ and we are done.
\end{proof}

Choose a cocharacter $\lambda:\Gm\rightarrow\Aut_G(X)$ splitting $\eta$.  We may apply Lemma~\ref{hens-conj} to find an affine \'etale cover $\{U_i\}_{i=1}^n$ of $\Spec A$ (with each $U_i$ is connected) so that $\lambda|_{U_i}$ is defined over $k^{\rm{sep}}$ (or more precisely, the structure morphism $U_i\rightarrow\Spec k$ factors through a finite separable extension $k'/k$, and $\lambda|_{U_i}$ is defined over $k'$).  As \'etale morphisms are open, we are reduced to checking that for any point $s\in\Spec A$ in the image of $U_i$, the type of $\lambda$ at $\kappa(s)$ can be computed at $U_i\times_k\kappa(s)$. In other words, we need to know that types are insensitive to extensions $F/\kappa(s)$, where $F$ is a finite \'etale $\kappa(s)$-algebra (not necessarily a field).  But this is clear, since the type is constant on each component of $U_i\times_k\kappa(s)$ and we can compute the type at $\kappa(s)$ itself on any geometric point over it.

%
%

\subsubsection{Endomorphisms and nilpotent elements}\label{endos}
Let $X\rightarrow \Spec A$ be a $G$-torsor, corresponding to a fiber functor $\eta:\Rep_k(G)\rightarrow \Proj_A$.  Suppose that for each $V\in \Rep_k(G)$, $\eta(V)$ is equipped with an endomorphism $N_V$, and suppose further that these endomorphisms are exact and tensor-compatible, in the sense that $N_{V\otimes V'}=1_V\otimes N_{V'}+N_V\otimes 1_{V'}$.  Then $(1+\varepsilon N_V)_V$ is a family of exact and tensor-compatible automorphisms of $\eta(V)_{A[\varepsilon]/\varepsilon^2}$.  In other words, we have an $A[\varepsilon]/\varepsilon^2$-point of $\underline\Aut^\otimes(\eta)$, and therefore an element $N\in\Aut_G(X)(A[\varepsilon]/\varepsilon^2)=\Lie\Aut_G(X)$.

We can say more when the $\{N_V\}$ are all \emph{nilpotent} endomorphisms and $k$ has characteristic $0$.  Then for each $A$-algebra $A'$ and each representation $V$, there is an action of $\Ga(A')$ on $\eta(V)_{A'}$, where $t\in\Ga(A')$ acts via $\exp(t\cdot N_V)$ (note that because $N_V$ is nilpotent, there are no issues of convergence).  We therefore have homomorphisms
$$\Ga(A')\rightarrow\Aut^\otimes(\varphi_{A'}\circ\eta)=\Aut_G(X)(A')$$
These homomorphisms are functorial in $A'$, so we have a homomorphism of $A$-group schemes $\Ga\rightarrow \Aut_G(X)$.  This in turn induces a homomorphism of Lie algebras over $\Lie\Ga\rightarrow \Lie\Aut_G(X)$.  In particular, the image of the distinguished element $d/dt\in\Lie\Ga$ yields a distinguished element $N\in\Lie\Aut_G(X)$.

\subsubsection{Semi-linear automorphisms}\label{tann-semilinear}
Let $X\rightarrow \Spec k'\otimes_kA$ be a $G$-torsor, corresponding to a fiber functor $\eta:\Rep_k(G)\rightarrow \Proj_{k'\otimes_kA}$, where $k'/k$ is a finite cyclic extension, with $\Gal(k'/k)$ generated by $\varphi$.  Suppose that for each $V\in \Rep_k(G)$, $k'\otimes_k\eta(V)$ is equipped with a bijection $\Phi_V:k'\otimes_k\eta(V)\rightarrow k'\otimes_k\eta(V)$ which is $A$-linear but $k'\otimes_kA$-semi-linear over $\varphi$.  That is, $\Phi(av)=\varphi(a)\Phi(v)$ for $a\in k'\otimes_kA$, $v\in k'\otimes_k\eta(V)$.  Suppose further that the $\Phi_V$ are exact and tensor compatible, in the sense that $\Phi_{V\otimes V'}=\Phi_V\otimes\Phi_{V'}$.  This is the same thing as a tensor-compatible family of isomorphisms (of $k'\otimes_k A$-modules) $\Phi_V':\varphi^\ast\eta(V)\xrightarrow{\sim} \eta(V)$.

Thus, we get an isomorphism $\Phi':\varphi^\ast\eta\xrightarrow{\sim}\eta$ of fiber functors $\Rep_k(G)\rightrightarrows\Proj_{k'\otimes_kA}$, and therefore an isomorphism of $G$-bundles $\Phi':\varphi^\ast X\xrightarrow{\sim} X$.

We can give another interpretation of $\Phi'$.  We may consider the Weil restriction $\Res_{k'/k}X$, which is a $\Res_{k'/k}(G_{k'})$-torsor over $A$, and we may use $\{\Phi_V\}$ to define a homomorphism 
\[	\Aut_{\Res_{k'/k}G_{k'}}(\Res_{k'/k}X)\rightarrow \Aut_{\Res_{k'/k}G_{k'}}(\Res_{k'/k}X)	\]
Concretely, if $V\in\Rep_k(G)$ and $g_V\in \Res_{k'/k}\Aut\eta(V)(A')=\Aut(\eta(V)\otimes_kA')$, then this homomorphism sends $g_V$ to $\Phi_V\circ g_V\circ\Phi_V^{-1}$.  Since $\Phi_V$ is semi-linear, this is a kind of ``twisted conjugation'' on $\GL(V)$.

\subsubsection{Continuous Galois representations}\label{cont-galois}

Let $E$ and $K$ be finite extensions of $\Q_p$, and let $G$ be an affine algebraic group over $E$.  Let $\omega:\Rep_E(G)\rightarrow \Vect_E$ be the forgetful fiber functor.  Suppose that for every $V\in\Rep_E(G)$ we have a continuous representation $\rho_V:\Gal_K\rightarrow \GL(V)$, and suppose that this family of representations is $\otimes$-compatible and exact, in the sense that $\rho_{V\otimes_EV'}=\rho_V\otimes\rho_{V'}$ and that if $0\rightarrow V'\rightarrow V\rightarrow V''\rightarrow 0$ is exact, then so is $0\rightarrow \rho_{V'}\rightarrow \rho_{V}\rightarrow \rho_{V''}\rightarrow 0$.  Then each $g\in\Gal_K$ defines a tensor automorphism of $\omega$, and therefore an element of $G(E)$.  

Thus, we get a homomorphism $\rho:\Gal_K\rightarrow G(E)$.  We wish to show that it is continuous.  But if $\sigma:G\rightarrow \GL(V)$ is a faithful representation, then considering $\sigma\circ\rho$ embeds the image of $\rho$ in the $E$-points of a closed subgroup of $\GL(V)$.  Since $\rho_V=\sigma\circ\rho$ is continuous by assumption, so is $\rho$.

\subsubsection{Families of de Rham representations}

Let $E$ and $K$ be finite extensions of $\Q_p$, let $A$ be an $E$-affinoid algebra, and let $G$ be an affine algebraic group over $E$.  Let $\rho:\Gal_K\rightarrow G(A)$ be a continuous homomorphism.  We say that $\rho$ is de Rham if $\sigma\circ\rho:\Gal_K\rightarrow \GL(V)$ is de Rham for every representation $\sigma\in \Rep_E(G)$.  By \cite[Theorem 5.1.2]{bellovin}, this is the case if and only if $(\sigma\circ\rho)_x:\Gal_K\rightarrow \GL(V_x)$ is de Rham for every $E$-finite artin local point $x:A\rightarrow B$.  In that case, $\D_{\dR}^K(\sigma\circ\rho)$ is an $A$-locally free $A\otimes_{\Q_p}K$-module such that $\D_{\dR}^K(\sigma\circ\rho)$ is filtered by sub-bundles $\Fil^\bullet\D_{\dR}^K(\sigma\circ\rho)$.  These sub-bundles are $A$-locally direct summands of $\D_{\dR}^K(V)$ as $A$-modules, but not necessarily as $A\otimes_{\Q_p}K$-modules.  The formation of $\D_{\dR}^K$ is exact and tensor compatible, so we get a $\Res_{E\otimes_{\Q_p}L/E}G$-torsor $\D_{\dR}^K(\rho)$ over $\Spec A$.  Furthermore, the filtrations $\Fil^\bullet\D_{\dR}^K(\sigma\circ\rho)$ define a $\otimes$-filtration on $\D_{\dR}^K(\rho)$.

We define the \emph{$p$-adic Hodge type} of $\rho$ to be the type of this $\otimes$-filtration.  Recall that the type is a geometric conjugacy class of cocharacters $\Gm\rightarrow (\Res_{E\otimes_{\Q_p}L/E}G)_{\overline{E}}$ which split the $\otimes$-filtration.  We showed that the type of a $\otimes$-filtration is locally constant on $\Spec A$, so the $p$-adic Hodge type of a family of de Rham representations is locally constant on $\Spec A$, as well.


\subsubsection{Potentially semi-stable Galois representations}

Let $E$ and $K$ be finite extensions of $\Q_p$, and let $G$ be an affine algebraic group defined over $E$.  Let $\rho:\Gal_K\rightarrow G(E)$ be a continuous homomorphism.  We say that $\rho$ is potentially semi-stable if $\sigma\circ\rho:\Gal_K\rightarrow \GL(V)$ is potentially semi-stable for every representation $\sigma\in \Rep_E(G)$.  

Let $\sigma_0$ be a faithful representation of $G$.  Then $\rho$ is potentially semi-stable if and only if $\sigma_0\circ\rho$ is.  This follows because $\sigma_0$ is a tensor generator of $\Rep_E(G)$.  More precisely, suppose that $\sigma_0\circ\rho$ becomes semi-stable when restricted to $\Gal_L$ for some finite extension $L/K$.  Then the formalism of admissible representations implies that $\sigma_0^\vee\circ\rho|_{\Gal_L}$ is semi-stable, as is $\sigma\circ\rho|_{\Gal_L}$ for any subrepresentation or quotient representation of $\sigma_0$.  Moreover, if $\sigma\circ\rho|_{\Gal_L}$ and $\sigma'\circ\rho|_{\Gal_L}$ are both semi-stable, then so is $(\sigma\otimes\sigma')\circ\rho|_{\Gal_L}$.  But since $\sigma_0$ is a tensor generator for $\Rep_E(G)$, this implies that $\sigma\circ\rho|_{\Gal_L}$ is semi-stable for any $\sigma\in\Rep_E(G)$.

\begin{remark}
A similar argument shows that that for any period ring $\B_\ast$, $\sigma\circ\rho$ is $\B_\ast$-admissible for every $\sigma\in\Rep_E(G)$ if and only if $\sigma_0\circ\rho$ is $\B_\ast$-admissible for an arbitrary faithful representation $\sigma_0:G\rightarrow \GL_n$.  Namely, the formalism of admissible representations implies that $\B_\ast$-admissibility is preserved under tensor products and duals, as well as passage to subrepresentations and quotient representations.  Since any faithful representation $\sigma_0$ is a tensor generator for $\Rep_E(G)$, $\B_\ast$-admissibility of $\sigma_0\circ\rho$ implies $\B_\ast$-admissibility of $\sigma\circ\rho$ for all $\sigma\in\Rep_E(G)$.
\end{remark}

\subsubsection{Filtered $(\varphi,N,\Gal_{L/K})$-modules}\label{tann-filt-phi-N}

Let $E$ and $K$ be finite extensions of $\Q_p$, let $L/K$ be a finite Galois extension, let $A$ be an $E$-algebra, and let $G$ be an affine algebraic group over $E$.  Let 
$$\eta:\Rep_EG\rightarrow\Proj_{A\otimes_{\Q_p}L_0}$$ 
be a fiber functor to the category of vector bundles over $A\otimes_{\Q_p}L_0$ which are $A$-locally free (i.e., the fiber at a point of $\Spec A$ is required to have constant rank), and suppose that $\eta(V)$ is equipped in an exact and tensor-compatible way with a semi-linear bijection $\Phi_V:\eta(V)\rightarrow\eta(V)$, a semi-linear action $\tau_V$ of $\Gal_{L/K}$ on $\eta(V)$, and an endomorphism $N_V$, and that $\{\eta(V)_L\}$ is equipped with a $\otimes$-filtration $\mathcal{F}_V^\bullet$ such that
\begin{itemize}
\item	$N_V = p\Phi_V\circ N_V\circ\Phi_V^{-1}$
\item	$N_V = \tau_V(g)\circ N_V\circ \tau_V(g)^{-1}$ for all $g\in\Gal_{L/K}$
\item	$\tau_V(g)\circ\Phi_V = \Phi_V\circ\tau_V(g)$ for all $g\in\Gal_{L/K}$
\item	$\mathcal{F}_V^\bullet$ is stable by the action of $\Gal_{L/K}$
\end{itemize}

The fiber functor $\eta:\Rep_EG\rightarrow\Proj_{A\otimes_{\Q_p}L_0}$ induces a $G$-torsor $X$ over $A\otimes_{\Q_p}L_0$, and therefore a $\Res_{E\otimes L_0/E}G$-torsor $\Res_{A\otimes L_0/A}X$ over $A$.

By Galois descent, the category of $A\otimes_{\Q_p}L$-vector bundles with a semi-linear $\Gal_{L/K}$-action is equivalent to the category of $A\otimes_{\Q_p}K$-vector bundles.  Therefore, the specification of a $\Gal_{L/K}$-stable filtration on $\eta(V)_L$ is the same as the specification of a filtration on $\eta(V)_L^{\Gal_{L/K}}$ (which is a finite projective $A\otimes_{\Q_p}K$-module).  Moreover, taking $\Gal_{L/K}$-invariants is exact and tensor-compatible, so $\{(\mathcal{F}^\bullet)^{\Gal_{L/K}}\}$ is a $\otimes$-filtration of $\{\eta(V)_L^{\Gal_{L/K}}\}$.

The family $\{N_V\}$ induces a point of $\Aut_{\Res_{E\otimes L_0/E}G}(\Res_{A\otimes L_0/A}X)(A[\varepsilon]/\varepsilon^2)$, as in Section~\ref{endos}.  Since the equations $N_V = p\Phi_V\circ N_V\circ\Phi_V^{-1}$ force $N_V$ to be nilpotent, $N$ is nilpotent as well.

The family $\{\Phi_V\}$ induces an isomorphism of $G$-torsors $\Phi':\varphi^\ast X\rightarrow X$, as well as a homomorphism 
\[	\Aut_{\Res_{E\otimes L_0/E}G}(\Res_{A\otimes L_0/A}X)\rightarrow\Aut_{\Res_{E\otimes L_0/E}G}(\Res_{A\otimes L_0/A}X)	\]
sending $g\in\Aut_{\Res_{E\otimes L_0/E}G}(\Res_{A\otimes L_0/A}X)(A')$ to $\Phi\circ \varphi^{\ast}(g)\circ\Phi^{-1}$, as in Section~\ref{tann-semilinear}.  We let $\underline\Ad(\Phi)$ be the induced map 
\[	\Lie\Aut_{\Res_{E\otimes L_0/E}G}(\Res_{A\otimes L_0/A}X)\rightarrow\Lie\Aut_{\Res_{E\otimes L_0/E}G}(\Res_{A\otimes L_0/A}X)	\]

Similarly, the families $\{\tau(g)\}$ for $g\in\Gal_{L/K}$ induce isomorphisms of $G$-torsors $\tau(g)':g^\ast X\rightarrow X$, homomorphisms 
\[	\Aut_{\Res_{E\otimes L_0/E}G}(\Res_{A\otimes L_0/A}X)\rightarrow\Aut_{\Res_{E\otimes L_0/E}G}(\Res_{A\otimes L_0/A}X)	\]
and maps
\[	\underline\Ad(\tau(g)):\Lie\Aut_{\Res_{E\otimes L_0/E}G}(\Res_{A\otimes L_0/A}X)\rightarrow\Lie\Aut_{\Res_{E\otimes L_0/E}G}(\Res_{A\otimes L_0/A}X)	\]

Then for any $g_1,g_2\in\Gal_{L/K}$, the isomorphism $\tau(g_1g_2)':(g_1g_2)^{\ast}X\rightarrow X$ is equal to the isomorphism $\tau(g_1)'\circ g_1^\ast\tau(g_2)':g_2^\ast g_1^\ast X\rightarrow X$, because this holds after pushing out by every representation $V\in\Rep_EG$.  Similarly, $\tau(g)'\circ g^\ast\Phi'=\Phi'\circ \varphi^\ast\tau(g)'$.


Finally, we observe that $\underline\Ad(\tau(g))(N)=N$ for all $g\in\Gal_{L/K}$ and $N=p\underline\Ad(\Phi) (N)$, since these equalities hold after pushing out by every representation $V\in\Rep_EG$.

To summarize, we have constructed 
\begin{itemize}
\item	a $\Res_{E\otimes L_0/E}G$-torsor $\Res_{A\otimes L_0/A}X$,
\item	a $\Res_{E\otimes K/E}G$-torsor $X_L^{\Gal_{L/K}}$,
\item	a parabolic subgroup scheme $P\subset \Aut_{\Res_{E\otimes K/E}G}(X_L^{\Gal_{L/K}})$,
\item	an isomorphism of $G$-torsors $\Phi':\varphi^\ast X\rightarrow X$, a nilpotent element 
\[	N\in\Lie\Aut_{\Res_{E\otimes L_0/E}G}\Res_{A\otimes L_0/A}X	\]
and a family of isomorphisms of $G$-torsors $\tau(g)':g^\ast X\rightarrow X$,
\end{itemize}
satisfying various compatibilities.

\subsubsection{Families of potentially semi-stable Galois representations}\label{tann-fam-pst}

Let $E$ and $K$ be finite extensions of $\Q_p$, let $A$ be an $E$-affinoid algebra, and let $G$ be an affine algebraic group over $E$.  Let $\rho:\Gal_K\rightarrow \Aut_G(X)(A)$ be a continuous homomorphism, where $X$ is a trivial $G$-torsor over $A$.  We say that $\rho$ is potentially semi-stable if $\sigma\circ\rho:\Gal_K\rightarrow \GL(V)$ is potentially semi-stable for every representation $\sigma\in \Rep_E(G)$.  By \cite[Theorem 5.1.2]{bellovin}, this is the case if and only if $(\sigma\circ\rho)_x:\Gal_K\rightarrow \GL(V_x)$ is potentially semi-stable for every $E$-finite artin local point $x:A\rightarrow B$.  In that case, $\D_{\st}^L(\sigma\circ\rho)$ is an $A$-locally free $A\otimes_{\Q_p}L_0$-module such that $\D_{\st}^L(\sigma\circ\rho)_L$ is filtered by sub-bundles $\Fil^\bullet\D_{\st}^L(\sigma\circ\rho)_L$, and $\D_{\st}^L(\sigma\circ\rho)$ is equipped with a bijection $\Phi_\sigma:\D_{\st}^L(\sigma\circ\rho)\rightarrow\D_{\st}^L(\sigma\circ\rho)$ which is semi-linear over $1\otimes\varphi$, an endomorphism $N_\sigma$ such that $N\circ\Phi=p\Phi\circ N$, and a semi-linear action of $\Gal_{L/K}$ which commutes with $\Phi$ and $N$ and stabilizes $\Fil^\bullet\D_{\st}^L(\sigma\circ\rho)_L$.

These structures are exact and $\otimes$-compatible in the senses discussed above, and so we get a $G$-torsor $\D_{\st}^L(\rho)$ over $\Spec A\otimes L_0$ 
 together with a isomorphism of $G$-torsors $\Phi':\varphi^\ast\D_{\st}^L(\rho)\rightarrow\D_{\st}^L(\rho)$, an element $N\in\Lie\Aut_{\Res_{E\otimes L_0/E}G}(\D_{\st}^L(\rho))$ satisfying $N=\frac{1}{p}\underline{\Ad}(\Phi)(N)$, a family of isomorphisms $\tau(g)':g^\ast\D_{\st}^L(\rho)\rightarrow\D_{\st}^L(\rho)$ commuting with $\Phi$ and $N$, and a geometric conjugacy class of cocharacters $\Gm\rightarrow \Res_{E\otimes K/E}G$.

\bibliographystyle{alpha}
\bibliography{thesis}

\end{document}